\documentclass[reqno]{amsart}
\usepackage{xcolor}
\usepackage{amsthm,amsmath,amssymb,latexsym,soul,cite,mathrsfs}
\usepackage{enumitem}
\usepackage{color,enumitem}
\usepackage[colorlinks=true,urlcolor=blue,
citecolor=red,linkcolor=blue,linktocpage,pdfpagelabels,
bookmarksnumbered,bookmarksopen]{hyperref}
\usepackage[english]{babel}

\usepackage{graphicx}

\usepackage[left=2.7cm,right=2.7cm,top=2.9cm,bottom=2.9cm]{geometry}

\usepackage[hyperpageref]{backref}

\numberwithin{equation}{section}
\newtheorem{theorem}{Theorem}[section]
\theoremstyle{plain}
\newtheorem{theoremletter}{Theorem}

\newtheorem{lemma}[theorem]{Lemma}
\newtheorem{lemmaletter}[theoremletter]{Lemma}
\newtheorem{corollary}[theorem]{Corollary}
\newtheorem{proposition}[theorem]{Proposition}

\newtheorem{remark}[theorem]{Remark}

\newcommand{\dx}{\,\mathrm{d}x}

\newcommand{\dt}{\,\mathrm{d}t}
\newcommand{\ds}{\,\mathrm{d}s}

\DeclareMathOperator{\supp}{supp}

\newcommand{\loca}{\operatorname{loc}}

\usepackage[norefs,nocites]{refcheck}

\title[]{Application of a profile decomposition theorem to elliptic equations with critical growth}

\author[D.~Ferraz]{Diego Ferraz}
\address{Department of Mathematics,
	Federal University of Rio Grande do Norte
	59078-970, Natal-RN, Brazil}
\email{diego.ferraz.br@gmail.com}

\thanks{Corresponding author: Diego Ferraz}

\date{\today}

\begin{document}
		
\begin{abstract}
	This paper introduces new variational methods centered on the direct application of a profile decomposition theorem for bounded sequences in Sobolev spaces. We employ these methods to prove the existence of ground state solutions for a class of semilinear elliptic equations in $\mathbb{R}^N$ with critical Sobolev growth, set in an asymptotically periodic framework where the coefficients converge to periodic functions at infinity. Our approach successfully addresses highly general nonlinearities, including a subcritical term that does not need to satisfy the classical Ambrosetti-Rabinowitz condition and a critical term that extends far beyond the standard pure power assumption to include functions with oscillatory behavior. We prove the existence of ground states under two alternative conditions: either a strict energy gap between the minimax levels of the original and asymptotic problems or a direct energy comparison between the associated functionals. Some restrictive assumptions, such as specific decay rates for the coefficients or monotonicity properties of the nonlinearities, are not required in our results.
\end{abstract}

	\maketitle
	%\tableofcontents
 
%%%%%%%%%%%%%%%%%%%%%%%%%%%%%%%%%%%%%%%%%%%%%%%%%%%%%%%%%%%%%%%%%%%%%%%%%%%%%%%%%INTRODUÇÃO%%%%%%%%%%%%%%%%%%%%%%%%%%%%%%%%%%%%%%%%%%%%%%%%%%%%%%%%%%%%%%%%%%%%%%%%%%%%%%%%%%%

%%%%%%%%%%%%%%%%%%%%%%%%%%%%%%%%%%%%%%%%%%%%%%%%%%%%%%%%%%%%%%%%%%%%%%%%%%%%%%%%%INTRODUÇÃO %%%%%%%%%%%%%%%%%%%%%%%%%%%%%%%%%%%%%%%%%%%%%%%%%%%%%%%%%%%%%%%%%%%%%%%%%%%%%%%%%%%
\section{Introduction}

The study of variational problems in unbounded domains is fundamentally linked to the loss of compactness in Sobolev embeddings. A breakthrough in understanding this phenomenon came with the seminal concentration-compactness principle, introduced by P.-L. Lions in his foundational works \cite{zbMATH04155283,MR778970,zbMATH04155284,zbMATH04155282}. A cornerstone of this theory, presented in \cite[Lemma I.1]{zbMATH04155283}, provides a precise description of the behavior of weakly convergent sequences. It establishes that any loss of strong convergence manifests as the concentration of the sequence's ``mass" at a countable set of points, forming profiles that behave as ``Dirac masses". Crucially, the lemma provides a quantitative relationship between the size of these concentrated masses and the corresponding concentration of gradient energy, becoming an indispensable tool in the analysis of elliptic equations with nonlinearities having critical growth.

Building upon the concentration-compactness principle, M. Struwe \cite{zbMATH03849694} offered a more detailed characterization of the behavior of non-convergent Palais-Smale sequences. His analysis focused on the classical Brézis-Nirenberg problem \cite{MR709644},
\begin{equation}\label{B-N}
	\left\{
		\begin{aligned}
			-\Delta u &= |u|^{2^\ast -2} u + \lambda u,\quad &&\text{in }\Omega,\\
			u&=0,\quad &&\text{on }\partial \Omega,
		\end{aligned}
		\right.
	\end{equation}
	where $\Omega \subset \mathbb{R}^N$ is a bounded smooth domain, $N \geq 3,$ $\lambda >0$ and $2^\ast = 2N/(N-2). $ In essence, Struwe's global compactness result shows that any loss of compactness is highly structured. A non-convergent Palais-Smale sequence can be decomposed, or ``split", into a sum consisting of its weak limit (which is a solution to the original problem) and a series of ``bubbles". These ``bubbles" are, up to translation and rescaling, solutions to the limiting problem at infinity $-\Delta u = |u|^{2^\ast -2} u,$ in $\mathbb{R}^N,$ $u(x)\rightarrow 0,$ as $|x|\rightarrow \infty.$ This description leads to the terminology ``Splitting Lemma", which provides a complete picture of all Palais-Smale sequences, demonstrating that the failure of convergence is precisely quantified by the appearance of solutions to the limit equation.

Following the global compactness result of M. Struwe \cite{zbMATH03849694}, a central question in the field became the precise characterization of the remainder term in the splitting of Palais-Smale sequences. A significant advance in this direction was made by S. Solimini in \cite{zbMATH00834238}. While the remainder was known to ``vanish" in some sense, S. Solimini provided a much finer analysis. He proved that although the remainder may not converge to zero in the optimal Lorentz space $L^{2^\ast, 2}(\mathbb{R}^N)$, it does converge to zero in $L^{2^\ast, q}(\mathbb{R}^N)$ for any $q > 2$. This result offered a more precise, quantitative description of the vanishing phenomenon, showing that the loss of compactness was more structured than previously understood.

This line of inquiry, aiming for a complete characterization of the remainder term and avoiding Dirac masses, culminated in the framework of \textit{profile decomposition}, whose central ideas were rigorously formulated by P. Gérard in \cite{zbMATH01215984} and later extended by S. Jaffard \cite{zbMATH01270852}. This type of result, now central to the analysis of non-compact problems, provides a complete description of why a bounded sequence in a Sobolev space may fail to converge strongly. It expresses any such sequence as a sum of two components: a series of well-chosen ``profiles", which are generated from a fixed set of functions through scaling and translation, and a remainder term that vanishes in a stronger sense (e.g., in a better Lebesgue space). The power of this framework was recognized, and the theory was extended from the specific context of Palais-Smale sequences to arbitrary bounded sequences in a wide range of function spaces. Since then, significant effort has been dedicated to generalizing profile decomposition to various settings, and the literature on this topic has become vast, we refer the reader to \cite{zbMATH07573810,tintabook,zbMATH06252858,zbMATH06419145,zbMATH06324011,zbMATH06505690} and the references therein.

While the foundational principles of concentration-compactness have become ubiquitous in nonlinear analysis, the modern and more technical framework of profile decomposition has been, in contrast, less frequently exploited to directly establish existence results for semilinear elliptic equations. Bridging this gap between the abstract theory and its application are the noteworthy contributions of K. Tintarev, I. Schindler and K.-H. Fieseler \cite{tintabook,MR2409935,MR2465979}. They not only refined the profile decomposition framework for the Sobolev space $D^{1,2}(\mathbb{R}^N)$ but also demonstrated its power as a direct tool for solving variational problems. A cornerstone of their work (\cite[Lemma 5.3 and Theorem 5.1]{tintabook}) is the following result,
\begin{theoremletter}\cite[Theorem 6.1]{MR2465979}\label{teo_tinta}
	Let $(u_k) \subset D^{1,2}(\mathbb{R}^N)$ be a bounded sequence in the standard $D^{1,2}(\mathbb{R}^N)$--norm and $\gamma > 1.$ There exist $(w^{(n)})_{n \in \mathbb{N}_\ast} \subset D^{1,2}(\mathbb{R}^N),$ $(y^{(n)}_k)_{k \in \mathbb{N}}\subset \mathbb{Z}^N,$ $(j_k^{(n)})_{k \in \mathbb{N}}\subset \mathbb{Z}$ and disjoint sets (if nonempty) $ \mathbb{N}_0,$ $\mathbb{N}_{+ },$ $\mathbb{N}_{-}\subset \mathbb{N},$ with $\mathbb{N}_\ast = \mathbb{N}_0\cup \mathbb{N}_+\cup \mathbb{N}_{-}$ such that, for a renumbered subsequence of $(u_k),$
	\begin{align}
		&\gamma ^{-\frac{N-2}{2} j_k ^{(n)} } u_k (\gamma ^{ - j_k ^{(n)}}\cdot + y_k ^{(n)}   ) \rightharpoonup  w^{(n)},\quad \text{as }k\rightarrow \infty,\ \text{in }D^{1,2}(\mathbb{R}^N),& \label{tinta1}\\
		&| j_k ^{(n)} - j_k^{(m)}| + | \gamma^{j_k ^{(n)}}(y_k ^{(n)}  -y_k ^{(m)}  ) | \rightarrow +\infty,\quad \text{as }k \rightarrow \infty,\ \text{for }m \neq n,& \label{tinta2}\\
		&\sum  _{n\in \mathbb{N}_\ast} \| \nabla w^{(n)}\|^2 _{2}  \leq \limsup_k \| \nabla  u_k \|^2_2,& \label{tinta3}\\
		&u_k - \sum _{n \in \mathbb{N}_\ast } \gamma ^{\frac{N-2}{2} j^{(n)}_k} w^{(n)} (\gamma^{j_k ^{(n)}}(\cdot - y_k^{(n)})  ) \rightarrow 0, \quad \text{as }k\rightarrow \infty,\ \text{in }L^{2^\ast }(\mathbb{R}^N),& \label{tinta4}
	\end{align}
	and the series in \eqref{tinta4} converges uniformly in $k.$ Furthermore, $1\in \mathbb{N}_0,$ $y_k^{(1)} = 0;$ $j_k^{(n)} = 0$ whenever $n \in \mathbb{N}_0;$ $j_k ^{(n)} \rightarrow - \infty,$ as $k \rightarrow \infty,$ whenever $n \in \mathbb{N}_{-};$ $j_k ^{(n)} \rightarrow + \infty,$ as $k \rightarrow \infty,$ whenever $n \in \mathbb{N}_{+}.$
\end{theoremletter}

Equipped with Theorem \ref{teo_tinta}, these authors provided direct applications to prove existence results for stationary Schr\"{o}dinger equations of the form
\begin{equation*}
	-\Delta u + \lambda u = h(x,u)\quad\text{in }\mathbb{R}^N,
\end{equation*}
where $N \geq 3,$ $\lambda \geq 0$ and $h(x,u)$ satisfies general Sobolev growth conditions. A key insight, particularly developed by K. Tintarev in \cite{MR2465979}, is a general criterion for compactness based on energy level comparisons. The profile decomposition \ref{teo_tinta} characterizes all possible ways a Palais-Smale sequence can lose compactness, with each ``way" corresponding to a simpler ``asymptotic problem" generated by the symmetries of the underlying space $D^{1,2}(\mathbb{R}^N)$. The central idea is that if the minimax level of the original problem is strictly below the infimum of the minimax levels of all possible asymptotic problems, then no loss of compactness can occur, and the sequence must contain a convergent subsequence.

On the other hand, a fundamental tool in the study of problems with the critical Sobolev exponent is the Pohozaev identity, which often provides a structural criterion for the nonexistence of solutions. On star-shaped domains, including $\mathbb{R}^N$ itself, this identity often acts as a non-existence theorem for problems with pure critical growth. A groundbreaking insight, due to H. Brézis and L. Nirenberg in their celebrated work \cite{MR709644}, was to show that this rigidity can be broken by introducing a lower-order perturbation. They demonstrated that adding a linear term $\lambda u$ to the equation (see \eqref{B-N}) on a bounded domain restores the possibility of existence for certain values of $\lambda$. This idea of using a subcritical perturbation to overcome the critical obstruction proved to be exceptionally fruitful. In the context of $\mathbb{R}^N,$ a natural extension is to consider power-type perturbations, leading to the study of equations such as
\begin{equation}\label{eq_basic}
	-\Delta u + u = \lambda |u|^{p-2}u + |u|^{2^\ast -2}u,\quad \text{in }\mathbb{R}^N,
\end{equation}
where $\lambda >0$ and $2<p<2^\ast.$ This class of equations, featuring competing subcritical and critical nonlinearities, has since become a vast and active area of research; see for instance \cite{zbMATH06723549,MR1455065,MR2532816,MR1843972} and the references therein.

Within the vast research landscape of problems with competing nonlinearities, a significant class of nonautonomous problems is given by equations with asymptotically periodic coefficients, such as
\begin{equation*}
	\left\{		
	\begin{aligned}
		&-\Delta u + V(x) u = a(x)u^{p-1} + b(x)u^{2^\ast -1}\quad\text{in }\mathbb{R}^N,\\
		&u \in H^1 (\mathbb{R}^N),\ u>0\text{ in }\mathbb{R}^N,
	\end{aligned}
	\right.
\end{equation*}
where $2<p<2^\ast$, and the positive, continuous functions $V,$ $a,$ and $b$ approach $\mathbb{Z}^N$--periodic functions $V_P,$ $a_P,$ and $b_P$ as $|x| \to \infty$. This general setting includes the well-known autonomous problem \eqref{eq_basic} as a particular case. A powerful strategy for this class of problems is to first establish the existence of a solution for the limiting periodic equation and then use this as a baseline to find a solution for the original problem. This approach was successfully employed by C. O. Alves, J. M. do {\'O} and O. Miyagaki in \cite{zbMATH01658785}. To implement this strategy, their work relies on a set of precise technical assumptions. Notably, they impose a specific behavior on the critical coefficient, requiring that $b(x) = \| b \|_\infty + O(|x|^{2-N})$ as $|x| \to \infty$. The cornerstone of their proof, however, is a variational condition that orders the coefficients with respect to their periodic limits: $V(x) \leq V_P(x),$  $a(x) \geq a_P(x),$ and $b(x) \geq b_P(x)$ in $\mathbb{R}^N,$ where at least one of the inequalities is strict on a set of positive measure. 

Further generalizations in this direction were made by H. F. Lins and E. A. B. Silva in \cite{MR2532816}. A key contribution of their work was to significantly relax the conditions on the subcritical nonlinearity, removing the need for the classical Ambrosetti-Rabinowitz condition,
\begin{equation}\label{AR}
	\text{There exists } \mu > 2 \text{ such that } \mu F(x,s) \leq f(x,s)s, \quad \forall\, (x,s) \in \mathbb{R}^N \times \mathbb{R}.
\end{equation}
They established the existence of nontrivial solutions for equations of the form
\begin{equation}\label{elvis}
	\left\{  
	\begin{aligned}
		&-\Delta u + V(x) u = f(x,u) + b(x)|u|^{2^\ast -2}u\quad\text{in }\mathbb{R}^N,\\
		&u \in H^1 (\mathbb{R}^N),
	\end{aligned}
	\right.
\end{equation}
where $f(x,u)$ is a general subcritical term not necessarily satisfying \eqref{AR} and $b$ satisfies $b(x) = \| b \|_\infty + \mathcal{O}(|x-x_0|^{N-2} ),$ for some $x_0 \in \mathbb{R}^N.$ Their framework is also asymptotically periodic, but they impose a more technical condition on the convergence to the periodic limit, requiring that $|f(x,s) - f_P(x,s)| \leq \eta_0(x)|s|^q$, where $f_P$ is the $\mathbb{Z}^N$--periodic limit and $\eta_0(x)$ is a function that decays sufficiently fast at infinity, with $2<q<2^\ast.$ 

	A guiding strategy in the study of nonautonomous problems like \eqref{elvis}, as seen in \cite{MR1455065,MR2532816,MR1843972}, is to compare the original problem with its periodic counterpart at infinity. The key idea for proving existence is often to establish an ``energy gap", i.e., a strict inequality between the minimax levels of the two problems. This gap serves as a sufficient condition to ensure the compactness of Palais-Smale sequences, ultimately leading to the existence of a solution. However, H.F. Lins and E. A. B. Silva \cite{MR2532816} went a step further, investigating the challenging case where the energy levels may coincide. They proved that a solution can still be found in this scenario if a specific, well-chosen path attains the minimax level as a maximum. In such a case, a critical point is guaranteed to exist along that path by abstract topological results (see \cite[Theorem 2.1]{MR2532816}).

The preceding discussion highlights a clear trajectory in the literature which is to progressively relax the hypotheses on the subcritical term $f$ for the class of problems \eqref{elvis}. Based on this, a natural question arises:
\begin{quote}
	\textit{What is the most general class of critical nonlinearities $k(x,u)$ that still guarantees the existence of a solution, moving beyond the classical pure power assumption $b(x)|u|^{2^\ast-2}u$?}
\end{quote}
Inspired by the work of K. Tintarev in \cite{MR2465979}, we aim to answer this question by studying existence of nontrivial solutions, via an application of the profile decomposition result Theorem \ref{teo_tinta}, for the following semilinear elliptic equation
\begin{equation}\tag{$\mathscr{P}$}\label{P}
	-\Delta u + V(x) u = f(x,u) + k(x,u)\quad\text{in }\mathbb{R}^N,
\end{equation}
with $N \geq 3,$ under a combination of general assumptions. Here $V:\mathbb{R}^N \rightarrow \mathbb{R}$ is a nonnegative continuous potential allowed to vanish at some points of $\mathbb{R}^N,$ $f$ is a general subcritical term that does not need to satisfy the Ambrosetti-Rabinowitz condition, and, most importantly, $k$ is a critical term belonging to a broad class of functions that allows for oscillatory behavior. The precise hypotheses are detailed below.

Our approach generalizes the works previously discussed in several key aspects. For instance, unlike the framework in \cite{MR1843972}, we do not require the critical term to have the restrictive pure power form $b(x)|u|^{2^\ast - 2}u$ with its associated structural conditions on the weight $b(x)$. Furthermore, we dispense with the specific decay rates for the subcritical term $f(x,u)$ that were required in \cite{MR2532816}. Instead, our framework relies only on the natural asymptotic convergence of the problem to its periodic counterpart:
\begin{enumerate}[label=($H_\infty$),ref=$(H_\infty)$]
	\item \label{h_infinito}$\lim _{|x| \rightarrow \infty} |V(x) - V_P(x)| = \lim _{|x| \rightarrow \infty} | f(x,s) - f_P(x,s)| = \lim _{|x| \rightarrow \infty} | k(x,s) - k_P(x,s)|= 0,$
\end{enumerate}
where the continuous functions $V_P, f_P,$ and $k_P$ are $\mathbb{Z}^N$-periodic, $V_P \in L ^\infty (\mathbb{R}^N)$ and the limits hold uniformly on compact subsets of $\mathbb{R}$ for the nonlinear terms. The corresponding periodic problem is thus given by
\begin{equation}\tag{$\mathscr{P}_P$}\label{PP}
	-\Delta u + V_P(x) u = f_P(x,u) + k_P(x,u)\quad\text{in }\mathbb{R}^N.
\end{equation}

The main contribution of our work is to establish a new method, based on the profile decomposition Theorem \ref{teo_tinta}, for proving the existence of ground state solutions for the problem \eqref{P} and its $\mathbb{Z}^N$--periodic limit \eqref{PP}. Our framework is designed for a general asymptotically nonautonomous setting and accommodates highly flexible assumptions on the nonlinearities. It allows for oscillatory critical and subcritical terms while removing the need for classical monotonicity assumptions, such as the requirement that the function $s \mapsto (f(x,s)+k(x,s))/s$ be increasing.

Our analysis proceeds in two main stages, distinguished by the structure assumed for the critical term $k(x,s)$. We first address the semi-autonomous case, where $k(x,s) = b(x)g(s)$, establishing the core methodology in this more structured setting. Subsequently, we demonstrate how these techniques can be adapted, with necessary modifications, to handle the fully nonautonomous case. This division allows us to introduce our novel techniques in a more concrete context initially, deferring a lengthy and potentially abstract list of assumptions required for the fully nonautonomous case.

Employing the profile decomposition technique, we demonstrate that the concentration-compactness analysis fundamentally relies on the relationship between the problem's \eqref{P} minimax level $c(I)$ and the minimax level $c(J_{\ast} )$ associated with the limit problem $-\Delta u = g_{\ast}(u)$ in $\mathbb{R}^N,$ where $g_\ast$ is a suitable self-similar function (see \ref{g_selfsimilar}). This limit problem originates from the loss of compactness due to concentrating profiles arising from the action of dilations. While the inequality $c(I) \leq c(J_{\ast} )$ holds generally (see Proposition \ref{p_jota}), establishing the strict inequality $c(I) < c(J_{\ast })$ serves as a crucial sufficient condition for recovering compactness, particularly in the fully nonautonomous setting. At this point, one can either follow the classical Br\'{e}zis-Nirenberg approach, proving a direct estimate for $c(I)$ under suitable conditions on $f(x,s)$ and $k(x,s)$, or assume the abstract condition $c(I) < c(J_{\ast})$ directly.

%%%%%%%%%%%%%%%%%%%%%%%%%%%%%%%%%%%%%%%%%%%%%%%%%%%%%%%%%%%%%%%%%%%%%%%%%%%%%%%%STANDARD HYPOTHESES
%%%%%%%%%%%%%%%%%%%%%%%%%%%%%%%%%%%%%%%%%%%%%%%%%%%%%%%%%%%%%%%%%%%%%%%%%%%%%%

\subsection{Standard hypotheses} In both cases, we make use of some standard hypotheses. The asymptotic condition \ref{h_infinito} is always assumed to hold, and we denote $F(x,s) = \int _0 ^s f(x, t ) \dt$ and $K(x,s) = \int _0 ^s k(x, t) \dt.$ Our main assumption on the potential $V$ is that it is nonnegative and satisfies the following condition, which is fundamental to ensure a well-posed variational setting:
\begin{enumerate}[label=($V_1$),ref=$(V_1)$]
	\item\label{V_autovalor} $V(x) \geq 0$ for all $x \in \mathbb{R}^N$ and
	\begin{equation*}
		d _1 := \inf \left\lbrace    \int _{\mathbb{R}^N} |\nabla u| ^2\mathrm{d}x +V(x)|u|^2\dx  :  u\in C_0 ^\infty  (\mathbb{R}^N)\text{ and } \int _{\mathbb{R}^N } u ^2\dx =1 \right\rbrace >0.
	\end{equation*}
\end{enumerate}
Motivated by \cite[Lemma 1.9]{tintabook}, we impose a general growth condition on the nonlinearity $f$. The next hypothesis controls the behavior of $f$ with quasi-critical growth near the origin and allowing it to oscillate on the Sobolev subcritical range at infinity.
\begin{enumerate}[label=($f_1$),ref=$(f_1)$]
	\item \label{f_geral} Given $\varepsilon>0,$ there exist $C_\varepsilon >0$ and $p_\varepsilon \in (2,2^\ast)$ such that
	\begin{equation*}
		|f(x,s)| \leq \varepsilon (|s| + |s|^{2^\ast -1}) + C_\varepsilon |s|^{p_\varepsilon - 1}.
	\end{equation*}
\end{enumerate}
To ensure that any related Palais-Smale sequence of \eqref{P} is bounded, we adopt a technical condition on $f$ introduced by X. H. Tang in \cite{MR3194360}. This hypothesis is weaker than the classical Ambrosetti-Rabinowitz condition.
\begin{enumerate}[label=($f_2$),ref=$(f_2)$]
	\item \label{f_tang} There is $\theta _0 \in (0,1)$ such that
	\begin{equation*}
		\frac{1-\theta ^2}{2}f(x,s) s \geq F(x,s)- F(x,\theta s), \quad \forall \, \theta \in [0,\theta _0].
	\end{equation*}
\end{enumerate}
We assume that the critical nonlinearity $k(x,s)$ satisfies a general growth condition governed by the critical Sobolev exponent.
\begin{enumerate}[label=($k_1$),ref=$(k_1)$]
	\item\label{k_geral} There is $\hat{a} _\ast >0$ such that $|k(x,s)| \leq \hat{a}_\ast |s|^{2^\ast-1}.$
\end{enumerate}
Moreover, the above conditions are assumed to hold for their periodic counterparts in Eq. \eqref{PP}. We impose a sufficient condition that enables a crucial comparison between the energy levels of \eqref{P} and \eqref{PP}.
\begin{enumerate}[label=($h_\ast$),ref=$(h_\ast)$]
	\item\label{h_minimax} Either the function $ s \mapsto \left(f_P(x,s)+k_P(x,s) \right)|s|^{-1}$ is strictly increasing, for all $x \in \mathbb{R}^N;$ or Eq. \eqref{PP} is independent of $x,$ that is, $V_P(x) = V_P >0,$ $f_P(x,s) = f_P(s)$ and $k_P(x,s) = k_P(s).$
\end{enumerate}

As discussed, the key to proving compactness is to establish an ``energy gap" between the original problem \eqref{P} and its periodic counterpart \eqref{PP}. Classically, this gap is guaranteed by imposing direct comparison conditions on the coefficients such as the ones described above in \cite{MR1455065,MR2532816,MR1843972} (see also Section \ref{s_remark}). However, following a more abstract and powerful approach by K. Tintarev \cite{MR2465979}, we instead assume this energy gap directly as a fundamental hypothesis. To formalize this concept, we now define the energy functional $I$ associated with \eqref{P} and the functional $I_P$ associated with \eqref{PP}.

We consider the Sobolev space $H^1_V(\mathbb{R}^N)=\overline{C_0^\infty (\mathbb{R}^N) }^{\| \, \cdot \, \| _V},$ as the completion of $C_0^\infty (\mathbb{R}^N)$ with respect to the norm
\begin{equation*}
	\| u \|_V = \left( \int _{\mathbb{R}^N} |\nabla u | ^2 + V(x) u^2 \dx \right)^{1/2}.
\end{equation*}
In Proposition \ref{p_sobolev} we prove that this space is well defined, with
\begin{equation}\label{charac}
	H^1 _V(\mathbb{R}^N) = \left\{  u \in H^1 (\mathbb{R}^N): \int _{\mathbb{R}^N} V(x) u^2\dx < + \infty\right\}.
\end{equation}
The Sobolev space $H^1 _{V} (\mathbb{R}^N)$ is a Hilbert space when endowed with the corresponding inner product
\begin{equation*}
	(u,v)_V :=  \int _{\mathbb{R}^N}   \nabla u \cdot \nabla v + V(x) u v  \dx.
\end{equation*}
Similarly, $H^1 _{V_P} (\mathbb{R}^N)$ is well defined, with the same characterization in \eqref{charac} being true, replacing $V$ by $V_P.$  The spaces $H^1 _{V} (\mathbb{R}^N)$ and $H^1 _{V_P} (\mathbb{R}^N)$ coincide with $H^1(\mathbb{R}^N) = H^1 _{V} (\mathbb{R}^N) =H^1 _{V_P} (\mathbb{R}^N).$ Next, we define
\begin{equation*}
	I(u) = \frac{1}{2}\int _{\mathbb{R}^N} |\nabla u | ^2 + V(x) u^2 \dx - \int _{\mathbb{R}^N} F(x,u) + K(x,u) \dx,\quad u \in H^1(\mathbb{R}^N),
\end{equation*}
and introduce the minimax level related to $I$ as $c(I) : = \inf _{\xi \in \Gamma _I} \sup_{t \geq 0} I( \xi (t)),$ where
\begin{equation*}
	\Gamma _I = \left\{  \xi \in C([0,\infty), H^1(\mathbb{R}^N) ) : \xi (0)=0\text{ and }\lim _{t \rightarrow \infty} I(\xi (t)) = - \infty \right\}.
\end{equation*}
The functional $I_P$ and the minimax level $c(I_P)$ are defined in an analogous way, replacing $V$ with $V_P$ and $F$ with $F_P(x,s) = \int _0 ^s f_P(x,t) \dt.$ As we establish in Lemma \ref{l_mpgeometry} and Remark \ref{r_compar}, these minimax levels are well-defined, positive, and finite. Moreover, under our assumptions, we prove that $c(I) \leq c(I_P)$ always holds (see Proposition \ref{p_ccp}). Our main compactness criterion is the following assumption:
\begin{enumerate}[label=($\mathscr{C}$),ref=$(\mathscr{C})$]
	\item\label{ce} $c(I) < c(I_P).$
\end{enumerate}

%%%%%%%%%%%%%%%%%%%%%%%%%%%%%%%%%%%%%%%%%%%%%%%%%%%%%%%%%%%%%%%%%%%%%%%%%%%%%%%%%INTRODUÇÃO %%%%%%%%%%%%%%%%%%%%%%%%%%%%%%%%%%%%%%%%%%%%%%%%%%%%%%%%%%%%%%%%%%%%%%%%%%%%%%%%%%%

%%%%%%%%%%%%%%%%%%%%%%%%%%%%%%%%%%%%%%%%%%%%%%%%%%%%%%%%%%%%%%%%%%%%%%%%%%%%%%%%%First case: Semi-autonomous critical nonlinearity %%%%%%%%%%%%%%%%%%%%%%%%%%%%%%%%%%%%%%%%%%%%%%%%%%%%%%%%%%%%%%%%%%%%%%%%%%%%%%%%%%%

\subsection{First case: Semi-autonomous critical nonlinearity} We begin our analysis with the semi-autonomous case, where the critical nonlinearity has the structure $k(x,u) = b(x)g(u)$. Here $g$ is an autonomous nonlinearity, while the coefficient $b \in C(\mathbb{R}^N) \cap L^\infty(\mathbb{R}^N)$ is positive, for which we denote $b_0 = \inf _{x \in \mathbb{R}^N} b(x)>0.$ The general problem \eqref{P} thus takes the specific form,
\begin{equation}\tag{$\mathscr{Q}$}\label{Q}
	-\Delta u + V(x) u = f(x,u) + b(x)g(u)\quad\text{in }\mathbb{R}^N.
\end{equation}
This problem is studied alongside its corresponding asymptotic periodic equation, which is naturally induced by the general condition \ref{h_infinito},
\begin{equation}\tag{$\mathscr{Q}_P$}\label{QP}
	-\Delta u + V_P(x) u = f_P(x,u) + b_P(x)g(u)\quad\text{in }\mathbb{R}^N,
\end{equation}
where the positive coefficient $b_P\in C(\mathbb{R}^N) \cap L^\infty(\mathbb{R}^N)$ is the $\mathbb{Z}^N$--periodic limit of $b,$ i.e.,
\begin{equation}\label{b_limite}
	\lim _{|x|\rightarrow \infty } |b(x) - b_P(x)| = 0 .
\end{equation}

In our first case, the subcritical perturbation $f$ must be sufficiently strong to ensure the energy functional has the desired minimax structure. This is a standard requirement in the variational methods developed in \cite{MR709644,MR1455065,MR1843972}, and it is guaranteed by the next superquadratic condition:
\begin{enumerate}[label=($f_3$),ref=$(f_3)$]
	\item \label{f_porbaixo} There are $\lambda >0$ and $p_0 \in (2,2^\ast )$ such that $F(x,s) \geq \lambda  |s|^{p_0}.$ Additionally, either one of the following conditions holds,
	\begin{enumerate}[label=\roman*):]
		\item $N\geq4;$
		\item $N=3$ and $4<p_0<2^\ast;$ or
		\item $N=3,$ $2<p_0 \leq 4$ and $\lambda $ sufficiently large.
	\end{enumerate}
\end{enumerate}
For the critical term $g$, we work within the class of \textit{self-similar} functions, introduced and developed in \cite{MR2409935,tintabook,MR2465979,MR2409928}. This choice provides a significant generalization of the classical pure power nonlinearity, $g(s) = |s|^{2^\ast -2}s.$ Crucially, this class includes functions that exhibit oscillatory behavior, which is a central feature of the problem studied herein:
\begin{enumerate}[label=($g_1$),ref=$(g_1)$]
	\item \label{g_selfsimilar} There exists $\gamma>1$ such that $G(s) = \gamma ^{-Nj}G(\gamma ^{\frac{N-2}{2}j }s),$ for all $j \in \mathbb{Z}$ and $s \in \mathbb{R},$
\end{enumerate}
where $G(s) = \int _0 ^s g(t) \dt.$ This can be equivalently formulated as $g(s) = \gamma ^{-\frac{N+2}{2} j   }  g(\gamma ^{ \frac{N-2}{2}j  }   s ).$ Moreover, we verify in Section \ref{s_ss} that if $g$ satisfies \ref{g_selfsimilar}, then
\begin{enumerate}[label=($\hat{g}_1$),ref=$(\hat{g}_1)$]
	\item\label{g_geral} There is $a_\ast>0$ such that $ |g(s)| \leq a_\ast |s|^{2^\ast -1}.$
\end{enumerate}
In addition, we must also ensure that the critical term is sufficiently strong from below. In analogy with the subcritical condition \ref{f_porbaixo}, we impose the following requirement on $G$:
\begin{enumerate}[label=($g_2$),ref=$(g_2)$]
\item\label{g_porbaixo} $g(s)s>0$ and there is $\lambda _\ast >0$ such that $G (s) \geq \lambda _\ast |s|^{2^\ast}.$
\end{enumerate}
We now introduce a refined Ambrosetti-Rabinowitz type condition. We define $\bar{G}(s) =  g(s)s$ and consider its associated Sobolev-type constant
\begin{equation*}
	\mathbb{S}_{\bar{G}} = \inf \left\{ \|\nabla u\|_2^2 : u \in D^{1,2}(\mathbb{R}^N) \text{ and } \int_{\mathbb{R}^N} \bar{G}(u) \dx = 1 \right\}.
\end{equation*}
Because $\bar{G}$ is positive and also self-similar, it is known that this constant is positive and attained (see \cite[Theorem 5.2]{tintabook} and \cite[Proposition 2.2]{MR2465979}). Introducing also the classical Sobolev constant
\begin{equation}\label{constS}
	\mathbb{S} = \inf \left\{  \| \nabla u \|_2 ^2 : u \in D^{1,2}(\mathbb{R}^N)  \text{ and }\| u \|_{2^\ast} = 1\right\},
\end{equation}
we state our condition as follows,
\begin{enumerate}[label=($g_3$),ref=$(g_3)$]
	\item\label{g_AR} $\mu_\ast G(s) \leq g(s)s$ for some $\mu_\ast > 2$ such that $\mu_\ast \geq 2N/(N-2\kappa_\ast)$, where
	\begin{equation*}
		\kappa_\ast := \left( \frac{(\mathbb{S}/\mathbb{S}_{\bar{G}})^{N/(N-2)}}{2^\ast \lambda_\ast} \frac{\| b \|_\infty }{ b_0} \right)^{\frac{N-2}{2}} < \frac{N}{2}.
	\end{equation*}
\end{enumerate}
We note that in the classical case where $b(x) = 1$ and $g(s) = |s|^{2^\ast - 2}s$, we have $\bar{G}(s) = |s|^{2^\ast}$, which implies $\mathbb{S}_{\bar{G}} = \mathbb{S},$ $\kappa_\ast = 1$ and $\mu_\ast = 2^\ast.$ Next, we recall that a solution $u$ is a \textit{ground state solution} if it is a critical point of the energy functional $I$ that achieves the minimum energy among all nontrivial solutions. This level is defined by
\begin{equation}\label{defgs}
	\mathcal{G}_S (I) = \inf _{u \in \mathrm{Crit}(I) }  I(u),\ \text{where }\mathrm{Crit}(I) = \left\{ u \in H^1(\mathbb{R}^N) \setminus \{ 0 \} : I'(u)  = 0 \right\}.
\end{equation}

Our first result establishes the existence of a ground state for the purely periodic problem.
\begin{theorem}[$\mathbb{Z}^N$--periodic case]\label{teo_periodic}
	If $V = V_P,$ $f= f_P$ and $b=b_P,$ then Eq. \eqref{Q} admits a ground state solution. If, in addition, \ref{h_minimax} holds, then this solution is at the mountain pass level, i.e., $I(u) =\mathcal{G}_S(I)= c(I).$
\end{theorem}

Next, we provide a compactness result for the asymptotically periodic problem \eqref{Q} under the ``energy gap" condition \ref{ce}.
\begin{theorem}[Compactness]\label{teo_compact}
	If \ref{h_minimax} and \ref{ce} hold, then any sequence $(u_k) \subset H^1 (\mathbb{R}^N)$ such that $I(u_k) \rightarrow c(I)$ and $I'(u_k) \rightarrow 0,$ has a convergent subsequence. In particular, Eq. \eqref{Q} has a nontrivial solution.
\end{theorem}

Finally, we state our main existence theorem for the asymptotically periodic case, which provides two alternative conditions for finding a ground state. In what follows, the notation $I \le I_P$ means that $I(v) \le I_P(v)$ for all $v \in H^1(\mathbb{R}^N)$.
\begin{theorem}[Ground state]\label{teo_gs}
	Assume \ref{h_minimax} holds. If either
	\begin{enumerate}[label=\bf \roman*):]
		\item The energy gap condition \ref{ce} holds, or;
		\item $I \leq I_P$;
	\end{enumerate}
	then Eq. \eqref{Q} has a ground state solution.
\end{theorem}

\subsection{Second case: The fully nonautonomous problem} We now extend our analysis to the fully nonautonomous Eq. \eqref{P}. The standing hypotheses on the potential $V$, the subcritical term $f$ (namely, \ref{V_autovalor}, \ref{f_geral} and \ref{f_tang}), and the critical growth control \ref{k_geral} are maintained for both the coefficients and their $\mathbb{Z}^N$--periodic counterparts. The fundamental structural assumption is that $k$ exhibits self-similar behavior at infinity.
\begin{enumerate}[label=($k_2$),ref=$(k_2)$]
	\item\label{k_ss} For $\nu \in \mathbb{Z}^N$ or $\nu = \infty,$ the following limit defines the self-similar function $g_\nu$ (in the sense of \ref{g_selfsimilar}), holding uniformly in compact sets of $\mathbb{R}^N\times \mathbb{R},$ for some $\gamma >1:$
	\begin{enumerate}[label=\roman*):]
	\item For $\nu \in \mathbb{Z}^N,$ $g_\nu (s) := \lim _{j \rightarrow \infty,\, j \in \mathbb{Z}} \gamma ^{-\frac{N+2}{2}j} k(\gamma ^{-j} x + \nu, \gamma^{\frac{N-2}{2} j} s).$
	\item For $\nu = \infty$, $g_\infty(s) := \lim _{m \rightarrow \infty } \gamma ^{-\frac{N+2}{2}j_m} k(\gamma ^{-j_m} x +y_m, \gamma^{\frac{N-2}{2} j_m} s)$ exists and is independent of the choice of sequences $(j_m)\subset \mathbb{Z}$ and $(y_m)\subset \mathbb{Z}^N$ satisfying $j_m \rightarrow \infty $ and $|y_m|\rightarrow \infty.$
	\end{enumerate}
\end{enumerate}
While the following assumption may appear abstract, it is in fact a natural generalization of a property inherent to the semi-autonomous case (cf. Section \ref{s_remark}). Indeed, we prove in Lemma \ref{l_geprop} that, if $k(x,s)=b(x)g(s)$ (with $g$ being self-similar), this condition is automatically satisfied.
\begin{enumerate}[label=($k_3$),ref=$(k_3)$]
	\item\label{k_tec} Given $a_1, \ldots, a_M \in \mathbb{R}$ there exists $C=C(M)>0$ such that
	\begin{equation*}
		\left|  K\left(x , \sum _{m=1}^M a_m \right)  - \sum _{m=1}^M K(x, a_m)\right| \leq C \sum _{m \neq n}|a_m|^{2^\ast -1}|a_n|,\quad \forall \, x \in \mathbb{R}^N.
	\end{equation*}
\end{enumerate}
To establish the correct minimax structure for the energy functional, we impose a superquadratic growth on the subcritical term $f.$ 
\begin{enumerate}[label=($f_0$),ref=$(f_0)$]
	\item\label{f_zero} $F(x,s) \geq 0$ and $\lim _{|s|\rightarrow \infty} F(x,s)s^{-2} = +\infty.$
\end{enumerate}
We also impose the classical Ambrosetti-Rabinowitz condition on the critical term $k$. We require these conditions to hold for both the original nonlinearities and their $\mathbb{Z}^N$--periodic counterparts.
\begin{enumerate}[label=($k_4$),ref=$(k_4)$]
	\item\label{k_AR} There are $\hat{\mu} _\ast >2$ and $\hat{\lambda}_\ast>0$ such that $\displaystyle \hat{\mu} _\ast \hat{\lambda}_\ast |s|^{2^\ast} \leq  \hat{\mu} _\ast K(x,s) \leq k(x,s)s.$
\end{enumerate}
Our primary challenge is to address the lack of compactness introduced by the general critical term $k(x,s)$, particularly due to dilation invariance captured by the profiles $w^{(n)}, n \in \mathbb{N}_+$. To overcome this, we propose two alternative sets of hypotheses as follows. 
\begin{enumerate}[label=($H_1$),ref=$(H_1)$]
	\item \label{h_um}
	\begin{enumerate}[label= \roman*):]
		\item There is a self-similar function $\hat{G} \in C(\mathbb{R})$ satisfying $k(x,s)s\leq \hat{G}(s)$ (see \ref{g_selfsimilar}).
		\item Let 
		\begin{equation*}
			\mathbb{S}_{\hat{G}} = \inf \left\{ \|\nabla u\|_2^2 : u \in D^{1,2}(\mathbb{R}^N) \text{ and } \int_{\mathbb{R}^N} \hat{G}(u) \dx = 1 \right\}.
		\end{equation*}
		We suppose that 
		$\hat{ \mu} _\ast \geq 2N / (N-2 \hat{ \kappa }_\ast ) $ where 
		\begin{equation*}
			\hat{\kappa}_\ast := \left( \frac{(\mathbb{S}/\mathbb{S}_{\hat{G}})^{N/(N-2)}}{2^\ast \hat{\lambda}_\ast} \right)^{\frac{N-2}{2}} < \frac{N}{2}.
		\end{equation*}
		\item \ref{f_porbaixo} and there exists a self-similar function $\underline{G} \in C(\mathbb{R})$ such that $K(x,s) \geq \underline{G} (s) \geq \hat{\lambda} _\ast |s|^{2^\ast}.$
	\end{enumerate}
\end{enumerate}
Hypothesis \ref{h_um} provides a route to compactness based on structural conditions and fine energy estimates (following the lines of Lemma \ref{l_minimax}), in the spirit of the Br\'{e}zis-Nirenberg approach. In contrast, inspired by our profile decomposition technique, the next hypothesis offers an alternative path by directly assuming an \textit{abstract energy gap condition}.  Consider the functional $J_\nu : D^{1,2}(\mathbb{R}^N) \rightarrow \mathbb{R}$ given by
\begin{equation}\label{jotaast}
	J_\nu (u) = \frac{1}{2} \int _{\mathbb{R}^N} |\nabla u |^2 \dx-\int _{\mathbb{R}^N} G_\nu (u) \dx,\quad G_\nu (s) = \int _0 ^s g_\nu (t) \dt,
\end{equation}
and its minimax level $c(J_\nu ) : = \inf _{\xi \in \Gamma _{J_\nu }} \sup_{t \geq 0} J_\nu ( \xi (t)),$ where
\begin{equation*}
	\Gamma _{J_\nu } = \left\{  \xi \in C([0,\infty), D^{1,2}(\mathbb{R}^N) ) : \xi (0)=0\text{ and }\lim _{t \rightarrow \infty} J_\nu (\xi (t)) = - \infty \right\}.
\end{equation*}
\begin{enumerate}[label=($H_2$),ref=$(H_2)$]
	\item \label{h_dois}$c(I) < c(J_\nu).$
\end{enumerate}

Under these more general assumptions, we establish a parallel set of existence and compactness results that perfectly mirror those obtained for the semi-autonomous case. The energy functional for \eqref{P} is the same as the functional $I$ previously defined, with the semi-autonomous term $b(x)g(s)$ replaced by its nonautonomous counterpart $k(x,s).$
\begin{theorem}[$\mathbb{Z}^N$--periodic case]\label{teo_periodic2}
	Supposing \ref{h_um}, if $V = V_P,$ $f= f_P$ and $k=k_P$ then Eq. \eqref{P} admits a ground state solution $u \in H^1 (\mathbb{R}^N)$. If in addition \ref{h_minimax} holds, then the ground state is at the mountain pass level, i.e. $I(u) =\mathcal{G}_S(I)= c(I).$
\end{theorem}
\begin{theorem}[Compactness]\label{teo_compact2}
	Under the assumption that either \ref{h_um} or \ref{h_dois} is satisfied, if \ref{h_minimax} and \ref{ce} hold, then any sequence $(u_k) \subset H^1 (\mathbb{R}^N)$ such that $I(u_k) \rightarrow c(I)$ and $I'(u_k) \rightarrow 0,$ has a convergent subsequence. In particular, Eq. \eqref{P} has a nontrivial solution.
\end{theorem}

\begin{theorem}[Ground state]\label{teo_gs2}
	Assuming \ref{h_minimax}, Eq. \eqref{P} has a ground state solution in either of the following cases:
	\begin{enumerate}[label=\bf \roman*):]
		\item \ref{h_um} holds, and either \ref{ce} or $I \leq I_P$ is satisfied.
		\item \ref{h_dois} and \ref{ce}.
	\end{enumerate}
\end{theorem}

Throughout the text, it will become clear that the proofs of these theorems follow by adapting the arguments presented for the first case, and as such, we only detail the necessary modifications. While the study of ground states for equations like \eqref{P} with autonomous nonlinearities is not new, see for instance \cite{zbMATH06723549,zbMATH06375868}, our approach is novel in several aspects. We establish our results under nearly optimal conditions and, crucially, provide a compactness analysis for the associated Palais-Smale sequences. Furthermore, we prove that the condition $I\leq I_P$ is also sufficient to guarantee a ground state for Eq. \eqref{P}. The main novelty, however, is our treatment for both nonautonomous cases, for which the methods used in \cite{zbMATH06723549,zbMATH06375868} do not seem to apply directly. In summary, we demonstrate that for non-coercive problems of the type studied herein, the profile decomposition theorem is the ideal tool to establish existence results under this high level of generality.

Although profile decomposition theorems provide a deep understanding of the lack of compactness in various functional settings, their direct application to establish existence results for nonlinear elliptic partial differential equations remains relatively scarce in the literature, particularly beyond the seminal works already cited. We conjecture that this tool is particularly well-suited for a broader class of problems than currently explored. Whenever a profile decomposition theorem is established for a general homogeneous Sobolev-type space associated with a differential operator, we believe the variational methods developed herein can be adapted to prove the existence of solutions for equations featuring nonlinearities with general oscillatory, self-similar growth compatible with the underlying symmetries revealed by the decomposition (such as those satisfying \ref{g_selfsimilar}).

\subsection{Remarks on the assumptions}\label{s_remark} Before we proceed, some comments on our hypotheses are necessary.
\begin{enumerate}[label=\bf \roman*):]
\item We first note that the inequality $c(I) \le c(I_P)$ is a natural consequence of our standing assumptions, a fact that we formally prove in Proposition \ref{p_ccp} (Appendix \ref{s_app_minimax}). To obtain the strict inequality required for the ``energy gap" condition \ref{ce}, we introduce the following sufficient condition (see Proposition \ref{A_estrito} for the proof):
\begin{enumerate}[label=($H_\ast$),ref=$(H_\ast)$]
	\item\label{Hast} The coefficients satisfy $V(x)\leq V_P(x),$ $F_P(x,s) \leq F(x,s),$ and $K_P(x,s) \leq K(x,s)$ for all $(x,s) \in \mathbb{R}^N \times \mathbb{R}$. Furthermore, at least one of the following holds:
	\begin{enumerate}[label=\alph*):]
		\item One of the three inequalities is strict on a set of positive measure.
		\item The inequalities for the primitives $F$ and $K$ hold strictly on an open interval containing the origin.
	\end{enumerate}
\end{enumerate}
\item Because $V$ and $V_P$ are continuous, condition \ref{h_infinito} implies $V-V_P \in L^\infty (\mathbb{R}^N).$ Thus, $V \in L^\infty (\mathbb{R}^N)$ if and only if $V_P \in L^\infty (\mathbb{R}^N).$
 \item In both studied cases we have the following superquadratic growth condition:
\begin{enumerate}[label=($h_1$),ref=$(h_1)$]
	\item \label{h_tang1} $\displaystyle \lim _{|s| \rightarrow \infty} ( F(x,s)+K(x,s) ) s^{-2} = +\infty,$ uniformly in $x \in \mathbb{R}^N.$
\end{enumerate}
\item Hypothesis \ref{f_geral} was recently utilized by M. Okumura in \cite[Theorem 4.3]{zbMATH07573810}. It is satisfied by a broad class of continuous functions that exhibit superlinear behavior at the origin and a general subcritical growth condition at infinity, which may involve a variable exponent. More precisely, hypothesis \ref{f_geral} holds if $f$ satisfies the following two conditions:
	\begin{itemize}
		\item[$(f'_1):$] $\lim _{s \rightarrow 0 }f(x,s)(|s| + |s|^{2^\ast -1})^{-1} = 0 ,$ uniformly in $x \in \mathbb{R}^N$;
		\item[$(f''_1):$] There exists $\varrho\in C(\mathbb{R})\cap L^\infty (\mathbb{R})$ with $\inf_{|s| \leq 1}\varrho(s)\geq \sup_{|s| \geq 1}\varrho(s)$, $2 < \inf_{s \in \mathbb{R}} \varrho(s) \leq \sup_{s \in \mathbb{R}} \varrho(s) < 2^\ast$ and $|f(x,s)| \leq C (1+|s|^{\varrho(s) -1}).$
    \end{itemize}
\item In Appendix \ref{s_app} we construct a suitable function $\varrho\in C^1(\mathbb{R})\cap L^\infty (\mathbb{R})$ for which
\begin{align}
    f (s) = \lambda p_0 |s|^{p_0 -2} s &+ \lambda \left( \varrho ' (s) s \ln |s|+  \varrho (s)  \right)|s|^{\varrho (s) - 2}s,\ \lambda >0,\label{exampleB}\\
     f(0) &:=0,\ f(1) = -f(-1) = \lambda( p_0 + \varrho(1)),\ p_0 \in (2,2^\ast),\nonumber
\end{align}
is a continuous nonlinearity satisfying \ref{f_geral}, \ref{f_tang} and \ref{f_porbaixo}.
\item Conditions \ref{g_selfsimilar}--\ref{g_AR} accommodate small oscillatory perturbations of the pure critical power $s \mapsto |s|^{2^\ast -2}s.$ For a given $b \in C(\mathbb{R}^N) \cap L^\infty (\mathbb{R}^N),$ a typical example is given by
\begin{equation}\label{example_ss}
	g(s) = ( \theta(s) + E) |s|^{2^\ast -2} s,\text{ with }\theta(s) = A + B \sin (\omega \ln |s| ),
\end{equation}
where $A,$ $B,$ $E >0$ are suitably chosen to be small and $\omega >0.$ Consequently, the hypotheses of Theorem \ref{teo_gs} are satisfied for $k(x,s) = b(x) g(s)$ provided that the periodic limit $b_P$ (satisfying \eqref{b_limite}) fulfills $0<\inf b_P\leq b_P(x) \leq b(x)$ with strict inequality on a set of positive measure. We demonstrate how this choice can be made in Appendix \ref{s_app}. More generally, applying the same ideas, the nonlinearity given by
\begin{equation*}
	k(x,s) = b(x)( D(x) \theta(s) + E) |s|^{2^\ast -2} s,
\end{equation*}
where $D \in C(\mathbb{R}^N) \cap L^\infty (\mathbb{R}^N)$ is $\mathbb{Z}^N$-periodic with $\inf D >0$, is a fully nonautonomous nonlinearity satisfying the conditions of Theorem \ref{teo_gs2}.

\item Hypothesis \ref{V_autovalor} is a general condition on the positivity of the Schr\"{o}dinger operator. A sufficient condition for \ref{V_autovalor} to hold, established in \cite{ferrazpenal}, is that the set where the potential is small is not too large, in the following sense. Let $V \in L ^1 _{\loca}(\mathbb{R}^N)$ be a function such that $V(x)\geq 0$ for a.e. $x\in \mathbb{R}^N$ and $V-\delta_0 \in L^q([V<\delta _0]),$ for some $q >N/2$, where $\delta_0 >0$ is such that
\begin{equation*}
\| V - \delta_0 \| _{L^{q} ([V<\delta_0 ])} < 2^{-1}(\min\{ 1,\delta_0  \})C_{q},
\end{equation*}
with
\begin{equation*}
0<C_{q} = \inf \left\lbrace  \| u \|^2 \| u \|^{-2} _{2q'} : u \in C^\infty _0 (\mathbb{R}^N) \right\rbrace <+\infty,
\end{equation*}
and $q ' = q/(q-1).$ Then $V$ satisfies \ref{V_autovalor}. In turn, this technical condition is satisfied by any nonnegative potential $V\in C(\mathbb{R}^N)$ such that 
\begin{equation*}
	\lim _{\delta \rightarrow 0_+} \frac{1}{\delta }\| V - \delta \| _{L^{q} ([V<\delta ])} =0,\quad \text{for some }q >1.
\end{equation*}
In particular, the non-coercive potential $V(x) = 1 - e^{-|x|^{2} },$ which has $V_P(x)\equiv 1$ as its periodic limit, satisfies \ref{V_autovalor}.
\item Functions satisfying condition \ref{g_selfsimilar} are called \textit{self-similar}. This class of functions, introduced in \cite{MR2409935} and further explored in \cite{tintabook,MR2465979,MR2409928} to investigate scalar field equations, can exhibit oscillatory behavior around the critical power $|t|^{2N/(N-2)}$. Notably, self-similar functions do not necessarily satisfy the condition that $s \mapsto |s|^{-1}g(s)$ is increasing, as demonstrated in the above examples.
\item Self-similar functions exhibit fractal-like behavior, that is, they are completely determined once their value is known on a single interval. For instance, let $I_j = [\gamma ^{\frac{N-2}{2} j}, \gamma ^{\frac{N-2}{2}(j+1)} ),$ $\gamma >1,$ $j \in \mathbb{Z}$ and $G_\ast: I_0 \rightarrow \mathbb{R}$ be a continuous function satisfying $\lim _{s \rightarrow \gamma ^{(N-2)/2 } }G_\ast (s) = \gamma ^N G_\ast (1)$. Defining $G(s) = \gamma ^{Nj} G_\ast(\gamma ^{-\frac{N-2}{2} j} s)$ for $s\in I_j$ and $j \in \mathbb{Z}$, and extending it to negative values by setting $G(s) := G(-s),$ for $s<0,$ we obtain a self-similar function.
\item Hypothesis \ref{f_tang}, introduced by X. H. Tang \cite{MR3194360}, is utilized for the subcritical term because it efficiently ensures the boundedness of Palais-Smale sequences. This choice allows us to concentrate our analysis on our main novelties: the generalization of the critical term to a highly general, nonautonomous class and the application of the profile decomposition result.
\item We employ Theorem \ref{teo_tinta}, a slight generalization of \cite[Theorem 5.1]{tintabook} where the specific base $2$ is replaced by an arbitrary base $\gamma>1$ and the translation group $\mathbb{R}^N$ is restricted to $\mathbb{Z}^N.$ The proof remains essentially the same. This formulation is necessary to align with the self-similarity conditions \ref{g_selfsimilar} and \ref{k_ss} within our asymptotically periodic setting.
\item A technical novelty of our work lies in the proof of the minimax estimate in Lemma \ref{l_minimax}. Classical arguments, such as those in \cite{MR709644, MR1455065,zbMATH01658785,MR2532816}, are tailored for homogeneous nonlinearities like the pure power $|s|^{2^\ast-2}s$.
\item We note that several choices in our presentation were made to favor clarity and a more direct line of argument. For instance, the minimax estimate in Lemma \ref{l_minimax} is proven for the general functional $I$. A similar result holds for $I_P$, but since $c(I) \le c(I_P)$, the former is sufficient for our purposes. Likewise, while some hypotheses could be further generalized, we have presented a version that avoids an overly technical or convoluted set of conditions, aiming to keep the framework as clear as possible while still accommodating a broad class of functions.
\item Consider 
\begin{equation}\label{sup_tinta}
	\mathbb{K} _{\bar{G}  } = \sup \left\lbrace \int _{\mathbb{R}^N} \bar{G}(u) \dx: u \in D^{1,2}(\mathbb{R}^N)\text{ and }\| \nabla u \|_2 = 1 \right\rbrace .
\end{equation}
This maximization problem is studied in \cite[Proposition 2.2]{MR2465979}, where the existence of a maximizer is established. Since $\mathbb{S}_{\bar{G}} = \mathbb{K}^{-2/2^\ast }_{\bar{G} }$ (see Appendix \ref{s_app} and \cite[Theorem 5.2]{tintabook}), it confirms that the existence of a maximizer for that problem implies the existence of a minimizer for ours.
\end{enumerate}

\subsection{Outline} The remainder of this paper is organized as follows. Our main strategy is to first develop the complete analysis for the semi-autonomous case (Eq. \eqref{Q}) and then demonstrate how these arguments can be adapted for the fully nonautonomous problem. Therefore, unless stated otherwise, we always assume the hypotheses \ref{V_autovalor}, \ref{f_geral}--\ref{f_porbaixo} and \ref{g_selfsimilar}--\ref{g_AR} with $k(x,s) = b(x) g(s)$. Section \ref{s_pre} establishes the variational framework and proves some preliminary results. In Section \ref{s_profile}, we describe additional properties of the profile decomposition theorem (Theorem \ref{teo_tinta}) and derive several of its key consequences for the nonlinear terms of our functional. The proofs of our main results for the semi-autonomous critical case, Theorems \ref{teo_periodic}, \ref{teo_compact}, and \ref{teo_gs}, are then presented in Sections \ref{s_teo_periodic}, \ref{s_comp}, and \ref{s_gsp}, respectively. Subsequently, in Section \ref{s_full}, we prove the results for the fully nonautonomous case by detailing the necessary adaptations to the preceding arguments. Finally, Appendices \ref{s_app} and \ref{s_app_minimax} are devoted to some complementary results, including the relation between the minimax levels and the construction of a nontrivial example for the nonlinearities $f$ and $k.$ For the reader's convenience, we have aimed to keep the text as self-contained as possible, recalling key definitions and results where necessary.\\

\noindent \textbf{Notation:} In this paper, we use the following notations:
	\begin{itemize}
		\item The usual norm in $L^{p}(\mathbb{R}^N)$ is denoted by $\|\cdot  \| _p;$
		\item $B_R(x_0)$ is the $N$-ball of radius $R$ and center $x_0;$ $B_R:=B_R(0);$
		\item  $C_i$ denotes (possibly different) any positive constant;
		\item $\mathcal{X}_A$ is the characteristic function of the set $A \subset \mathbb{R}^N;$
		\item $A^c =\mathbb{R}^N \setminus A,$ for $A \subset \mathbb{R}^N;$ 
		\item $|A|$ is the Lebesgue measure of the measurable set $A \subset \mathbb{R}^N;$
        \item $\| u \| = \left( \| \nabla u \| ^2 + \| u \|^2 \right)^{1/2},$ $u \in H^1 (\mathbb{R}^N);$
        \item $a_k = b_k + o_k (1)$ if and only if $\lim _{k \rightarrow \infty }(a_k - b_k) = 0.$
	\end{itemize}
		
%%%%%%%%%%%%%%%%%%%%%%%%%%%%%%%%%%%%%%%%%%%%%%%%%%%%%%%%%%%%%%%%%%%%%%%%%%%%%%%%%%%%%%%%%%%%%%%%%%%%%%
%%%%%%%%%%%%%%%%%%%%%%%%%%%%%%%%%%%%%%%%%%%%%%%%%%%%%%%%%%%%%%%%%%%%%%%%%%%%%%%%%%%%%%%%%%%%%%%%%%%%%%

\section{Preliminaries}\label{s_pre}

This section is devoted to establishing the variational framework for our problem. We begin by precisely defining the weighted Sobolev space $H_V^1 (\mathbb{R}^N)$ and proving its main properties.
\begin{proposition}\label{p_sobolev}
    The spaces $H_V^1 (\mathbb{R}^N)$ and $H_{V_P}^1 (\mathbb{R}^N)$ are well defined and are continuously embedded in $H^1 (\mathbb{R}^N).$ Moreover, $H_V^1 (\mathbb{R}^N) = H_{V_P}^1 (\mathbb{R}^N) = H^1 (\mathbb{R}^N)$ and the standard norms of these spaces are equivalent. 
\end{proposition}
\begin{proof}
We prove the first part of the statement using an argument from \cite{sirakov2000}. Condition \ref{V_autovalor} clearly implies
$\| u \|_2 ^2 \leq d_1^{-1} \| u \|^2_V,$ for any $u \in C^\infty _0 (\mathbb{R}^N),$ and hence
\begin{equation}\label{embedi}
    \| u \| \leq (1+d_1^{-1})^{1/2}\| u \|_V,\quad \forall \, u \in C^\infty _0 (\mathbb{R}^N).
\end{equation}
In particular, any sequence $(\varphi _k)_k \subset C^\infty _0 (\mathbb{R}^N)$ satisfying $\| \varphi _k - \varphi _l \|_V \rightarrow 0,$ also satisfies $\| \varphi _k - \varphi _l \| \rightarrow 0,$ as $k,l \rightarrow \infty.$ This ensures that $H_V^1 (\mathbb{R}^N)$ is well defined. An application of Fatou's lemma then yields
\begin{equation}\label{inclusao}
    H^1 _V(\mathbb{R}^N) \subset \left\{ u \in H^1 (\mathbb{R}^N) : \int _{\mathbb{R}^N} V(x)u^2 \dx < + \infty \right\}.
\end{equation}
Furthermore, \eqref{embedi} shows that $H_V^1 (\mathbb{R}^N)$ is continuously embedded in $H^1 (\mathbb{R}^N).$ To prove the converse of \eqref{inclusao}, let $u \in H^1 (\mathbb{R}^N)$ with $\int _{\mathbb{R}^N} V(x)u^2 \dx < + \infty .$ We first assume that $\supp(u)$ is compact. Let $(\varrho _k)$ be the standard sequence of mollifiers with support in the unit ball, and consider $u_k = \varrho _k \ast u \in C^\infty _0 (\mathbb{R}^N).$ Let $K$ be a compact set containing $\supp(u)$ and the support of $u_k$ for all $k.$ By the classical Friedrichs theorem we have $u_k \rightarrow u$ and $\nabla u_k \rightarrow \nabla u $ in $L^2 (\mathbb{R}^N).$ Moreover,
\begin{equation*}
\int _{\mathbb{R}^N } V(x) |u_k - u|^2 \dx \leq \| V \|_{L ^\infty (K)} \int _{\mathbb{R}^N } |u_k - u|^2\dx\rightarrow 0,\quad \text{as }k \rightarrow \infty.
\end{equation*}
This shows that $u \in H^1 _V(\mathbb{R}^N).$ For the general case, where $\supp(u)$ may not be compact, consider a truncation function $\eta \in C^\infty _0 (\mathbb{R}^N)$ such that $\eta =1 $ in $B_1$ and $\eta =0$ in $B^c_2.$ Define $\eta _k (x) = \eta (x/k)$ and $u_k = \eta _k u \in H^1 (\mathbb{R}^N).$ Note that each $u_k$ has compact support. By the previous case, $(u_k) \subset H^1_V(\mathbb{R}^N).$ Since $|\eta _k|\leq 1$ and $V(x)u^2 \in L^1(\mathbb{R}^N),$ the Lebesgue convergence theorem implies $\| u_k - u \|_V \rightarrow 0.$ Since $H^1_V(\mathbb{R}^N)$ is a Banach space, we conclude that $u \in H^1 _V(\mathbb{R}^N).$ The same argument applies to $H^1_{V_P} (\mathbb{R}^N),$ by replacing $V$ with $V_P.$

To prove the second part, we choose $R>0$ such that $-1 + V_P(x) < V(x) < 1+V_P(x),$ for $ x \in B_R^c.$ Then, for any $u \in H^1 _{V_P} (\mathbb{R}^N),$
\begin{align*}
\int _{\mathbb{R}^N} V(x) u^2 \dx & \leq \int _{B_R} V(x) u^2 \dx + \int _{B_R^c} (V_P(x) + 1) u^2 \dx \\
&\leq  \int _{\mathbb{R}^N } V_P(x) u^2 \dx + \sup _{x \in B_R} |V(x) - V_P(x)| \int _{B_R} u^2 \dx + \int _{B^c_R} u^2 \dx \\
& \leq \int _{\mathbb{R}^N } V_P(x) u^2 \dx + M_0 \| u \|_2^2,
\end{align*}
where $M_0 =  \max \left\{     \sup _{x \in B_R} |V(x) - V_P(x)| , 1  \right\}.$ Using \ref{V_autovalor}, we obtain
\begin{equation*}
\| u \| ^2 _V \leq (1 + M_0 d_1 ^{-1}) \| u \|^2_{V_P}.
\end{equation*}
By \eqref{charac}, we have $H^1 _{V_P}(\mathbb{R}^N) \subset  H^1 _{V}(\mathbb{R}^N).$ Replacing $V$ with $V_P,$ we similarly obtain $H^1 _{V}(\mathbb{R}^N) \subset  H^1 _{V_P}(\mathbb{R}^N),$ with $\| u \| ^2 _{V_P} \leq (1 + M_0 d_1 ^{-1}) \| u \|^2_{V}.$ Finally, since $V \in L^\infty (\mathbb{R}^N)$ we have $\| u \|_V \leq C \| u \|,$ for any $u \in H^1 (\mathbb{R}^N),$ and some $C>0.$ The characterization \eqref{charac} then implies that $H^1 (\mathbb{R}^N) \hookrightarrow H_V^1 (\mathbb{R}^N).$
\end{proof}
Next we describe the minimax structure of the energy functional $I.$
\begin{lemma}\label{l_mpgeometry}
	$I$ has the mountain pass geometry. Precisely,
 \begin{enumerate}[label=\roman*):]
     \item There exist $r,$ $b>0$ such that $I(u) \geq b,$ whenever $\|u\|_V = r;$
     \item There is $e \in H^1 (\mathbb{R}^N)$ with $\| e \|_V > r$ and $I(e) < 0.$
 \end{enumerate}
\end{lemma}
\begin{proof}
\textit{i)}: The first part is standard: Using \ref{f_geral} and \ref{g_selfsimilar} together with some Sobolev inequalities,
\begin{equation*}
    I(u) \geq \| u \|^2_V\left( \frac{1}{2} - \varepsilon C_1 - C_2 (\varepsilon) \| u \|^{p_\varepsilon - 2}_V - C_3 (\varepsilon) \| u \|_V^{2^\ast - 2}\right)>0,
\end{equation*}
for suitable $\varepsilon >0$ and $\| u \|_V>0$ small enough.

\textit{ii)}: Let $u_0 \in H^1(\mathbb{R}^N) \setminus \{ 0 \}.$ By \ref{h_tang1} and Fatou's lemma, we have
\begin{equation*}
\frac{I(t u_0)}{t^2} = \frac{1}{2}\| u_0 \|^2_V - \int _{\mathbb{R}^N} \frac{F(x,tu_0) + b(x) G(tu_0)}{(tu_0)^2}  u_0^2\dx\rightarrow - \infty,\quad \text{as }t \rightarrow \infty.
\end{equation*}
Hence $\lim _{t \rightarrow \infty }I(t u_0) = \lim _{t \rightarrow \infty }(I(t u_0)/t^2) t^2 = - \infty.$
\end{proof}
\begin{remark}
The minimax level defined above coincides with the usual minimax level of the mountain pass type geometry, i.e., define $\hat{c}(I) = \inf _{\hat{\xi} \in \hat{\Gamma} _I} \sup _{t \in [0,1]} I(\hat{\xi} (t)),$ where
\begin{equation*}
    \hat{ \Gamma } _I = \left\{ \hat{ \xi } \in C([0,1], H^1(\mathbb{R}^N) ): \hat{ \xi } (0)=0,\ \| \hat{ \xi }(1)\| >r\text{ and }I(\hat{ \xi }(1)) <0 \right\}.
\end{equation*}
Then, by \cite[Theorem 2.1]{MR2532816}, $\hat{c}(I)$ is well defined and $c(I) = \hat{c}(I).$ In fact, given $\xi \in \Gamma _I,$ we know of the existence of $t_0>0$ such that $\| \xi (t_0) \| > r$ and $I(\xi (t_0)) < 0.$ Define $\hat{\xi } (t) = \xi (t t_0),$ for $t \in [0,1].$ Then $\hat{\xi} \in \hat{\Gamma}_I$ and $\hat{c}(I) \leq \sup _{t \in [0,1]} I(\hat{\xi} (t))    = \sup _{t \geq 0} I(\xi (t)).$ Because $\xi \in \Gamma _I$ is arbitrary, we have $\hat{c}(I)\leq c(I).$ Conversely, taking $\hat{\xi } \in \hat{\Gamma} _I,$ the path defined by $\xi (t) = \hat{\xi }(t),$ for $t \in [0,1],$ and $\xi (t) = t \hat{\xi}(1),$ for $t \geq 1,$ belongs to $\Gamma _I.$ Moreover, $c(I) \leq \sup _{t \geq 0} I(\xi (t))    = \sup _{t \in [0,1]} I(\hat{\xi} (t)).$ Likewise, $c(I) \leq \hat{c}(I) .$
\end{remark}
\begin{remark}\label{r_compar}
Let $u_0 \in H^1 (\mathbb{R}^N)$ be a critical point of $I_P$ and suppose \ref{h_minimax}. Consider the first case of \ref{h_minimax} and take the path $\xi (t) = t u_0.$ It is standard to prove that $t=1$ is the unique critical point of $t \mapsto I_P(t u_0)$ and by the proof of Lemma \ref{l_mpgeometry}, $c(I_P) \leq \sup _{t \geq 0} I_P(\xi (t)) = I_P(u_0).$ For the second case of \ref{h_minimax}, we use the well-known fact \cite{MR1400007} that the following Pohozaev identity holds,
\begin{equation*}
\int _{\mathbb{R}^N} F_P(u_0)+b_P G_P(u_0) - \frac{V_P}{2} u^2_0\dx= \frac{N-2}{2N} \int _{\mathbb{R}^N} |\nabla u_0|^2 \dx,
\end{equation*}
to see that the path $\zeta(t) = u_0 (\cdot /t ),$ $\zeta(0) :=0,$ belongs to $\Gamma_{I_P}$ and possesses a unique critical point $t=1.$ In particular, $c(I_P) \leq \sup _{t \geq 0} I_P(\zeta (t)) = I_P(u_0).$
In summary, condition \ref{h_minimax} implies $c(I_P) \leq I_P(u_0).$
\end{remark}
We now establish our minimax estimate for $c(I),$ involving the general critical nonlinearity $b(x)g(s).$ For the following, we use the notation $\bar{G}_{b }(s) := \|b \|_\infty \bar{G}(s) = \|b \|_\infty g(s)s$ and
\begin{equation*}
	\mathbb{S}_{\bar{G}_{b } } := \inf \left\{ \|\nabla u\|_2^2 : u \in D^{1,2}(\mathbb{R}^N) \text{ and } \int_{\mathbb{R}^N} \bar{G}_{b }(u) \dx = 1 \right\}>0.
\end{equation*}
\begin{lemma}\label{l_minimax}
$c(I) < ((\mu_\ast -2)/(2\mu _\ast) ) \mathbb{S}^{N/2}_{\bar{G} _b}.$
\end{lemma}
\begin{proof}
Our argument involves the use of the functions
\begin{equation*}
	U _\varepsilon (x) = C(N) \frac{\varepsilon ^{ (N-2)/ 4 }}{   \left( \varepsilon + |x|^{2}  \right) ^{(N-2) / 2}  }    ,\ \varepsilon >0, \text{ where }C(N) = \left(  N \left( N-2  \right)    \right)^{(N-2) / 4},
\end{equation*}
that attain the infimum of \eqref{constS}, see \cite{MR463908}. Let $\psi \in C_0 ^\infty(\mathbb{R}^N),$ with $0\leq \psi \leq 1 $ and such that $\psi(x) =1,$ if $|x| \leq  \varrho/2,$  and $\psi(x) =0,$ if $|x| \geq \varrho.$ For 
	\begin{equation*}
	W_\varepsilon (x) = (\varepsilon + |x|^{2}) ^{(2-N)/2}, \text{ denote } v  _\varepsilon := \psi W _\varepsilon \text{ and } w _\varepsilon =  v  _\varepsilon   \|  v  _\varepsilon  \|^{-1}_{2 ^\ast}.
	\end{equation*}
Following \cite{MR1009077,MR3834729}, we obtain the following asymptotic behavior of $w_\varepsilon$ for small values of  $\varepsilon\in (0,1)$:
	\begin{align}
	\| \nabla w  _\varepsilon \| _2 ^2 &=  \mathbb{S} + \mathcal{O}(\varepsilon ^{(N-2) / 2}),\label{ozao1}\\
	\|w  _\varepsilon  \| _q ^q
	&=\left\{
	\begin{aligned}
	&\mathcal{O} (\varepsilon^{(N/2) (1-q/2^\ast) } ) , &\text{ if }q>2^\ast/2,&\\
	&\mathcal{O}(\varepsilon ^{(N-2) (q/4)} |\ln \varepsilon|), &\text{ if }q=2^\ast/2,&\\
	&\mathcal{O}(\varepsilon ^{(N-2) (q/4) }), &\text{ if }q<2^\ast/2,&
	\end{aligned}
	\right.\\
	\|w  _\varepsilon\| _2 ^2 & = \left\{
	\begin{aligned}
	&\mathcal{O} (\varepsilon  ) , &\text{ if }N>4,&\\
	&\mathcal{O} ( \varepsilon |\ln \varepsilon| ), &\text{ if }N=4,&\\
	&\mathcal{O} ( \varepsilon ^{ (N-2) / 2} ), &\text{ if }N=3,&
	\end{aligned}
	\right.\label{ozao3}
	\end{align}
	where the Bachmann–Landau notation $\beta (\varepsilon)= \mathcal{O} (\alpha (\varepsilon))$ stands for the existence of $c_1, c_2> 0$ such that $c_1 \leq \beta (\varepsilon)/\alpha (\varepsilon) \leq c_2.$ In particular $\lim _{\varepsilon \rightarrow 0 }\| w _\varepsilon \|_q^q = 0 $ for all $2\leq q<2^\ast.$ On the other hand, let us consider the path $\zeta ^{(\varepsilon)} _k (t)(x)= \gamma ^{ \frac{N-2}{2} j_k} w_\varepsilon (\gamma ^{j_k} (x/t) ),$ $\zeta ^{(\varepsilon)} _k (0):=0,$ where $j_k \rightarrow +\infty ,$ $k \in \mathbb{N}.$ Then, $	I (  \zeta ^{(\varepsilon)} _k (t)  ) = (1/2)t^{N-2}  \| \nabla w_\varepsilon \|_2^2 - t^N E_k^{(\varepsilon)},$ and by \ref{g_selfsimilar},
	\begin{equation*}
		E_k^{(\varepsilon)} = \int _{\mathbb{R}^N } \gamma ^{-Nj_k} F(\gamma ^{-j_k} t x , \gamma ^{ \frac{N-2}{2} j_k}  w_\varepsilon  ) \dx  + \int _{\mathbb{R}^N} b(\gamma ^{-j_k} t x ) G(w_\varepsilon ) \dx  - \int _{\mathbb{R}^N }\gamma ^{-2j_k} V(\gamma ^{-j_k} t x ) w_\varepsilon ^2\dx.
	\end{equation*}
	Notice that hypotheses \ref{f_porbaixo} and \ref{g_porbaixo}, together with \eqref{ozao1}--\eqref{ozao3} imply the existence of $C_0>0$ such that $E_k^{(\varepsilon)} > C_0,$ for any $\varepsilon \in (0,\varepsilon_0),$ $\varepsilon_0$ small enough and $k$ sufficiently large. Consequently, for this choice of $k$ (fixed) and $\varepsilon_0,$ the path $\zeta ^{(\varepsilon)} _k  $ belongs to $\Gamma _{I},$ for any $\varepsilon \in (0, \varepsilon _0).$ Thus it suffices to prove that $\max _{t \geq 0} I (\zeta ^{(\varepsilon)} _k (t)) < ((\mu_\ast -2)/(2\mu _\ast) ) \mathbb{S}^{N/2}_{\bar{G}_b}.$ Let
	\begin{align*}
		\psi _\varepsilon (t) &= \frac{1}{2}  t^{N-2} \| \nabla w_\varepsilon \|_2^2 + \frac{B_k}{2} t^N \| w_\varepsilon \|_2 ^2 - \lambda A_k t^N \| w_\varepsilon \|^{p_0}_{p_0} - t^N \int _{\mathbb{R}^N} b_0 G(w_\varepsilon ) \dx \\
		&=\frac{1}{2} \| \nabla w_\varepsilon \|_2^2 t^{N-2} - C_k^{(\varepsilon)} t^N ,\ t \geq 0,
	\end{align*}
	where $A_k = \gamma ^{j_k (\frac{N-2}{2}p_0 - N  )},$ $B_k = \gamma ^{-2j_k} \| V \| _\infty $ and $C_k^{(\varepsilon)} = \lambda A_k \| w_\varepsilon \|^{p_0}_{p_0}+\int _{ \mathbb{R}^N } b_0G(w_\varepsilon ) \dx -(B_k /2) \| w_\varepsilon \|^2_2>0.$ Then, $\max _{t \geq 0} I (\zeta ^{(\varepsilon)} _k (t)) \leq \max _{t \geq 0} \psi _\varepsilon (t) = \psi _\varepsilon (t_\varepsilon), $ for $t _\varepsilon = ( \|\nabla w_\varepsilon \|_2^2/(2^\ast C_k^{(\varepsilon)})  )^{1/2}.$ In particular, we see that there are $a_1,$ $a_2>0$ with $a_1< t_\varepsilon <a_2,$ for any $\varepsilon \in (0,\varepsilon_0).$ This yields the estimate,
	\begin{equation*}
		\psi_\varepsilon (t_\varepsilon) \leq \max _{t \geq 0} \varphi _{\varepsilon }( t) + \frac{B_k}{2} a_2^N \| w_\varepsilon \|_2 ^2- \lambda A_k a_1^N \| w_\varepsilon \|^{p_0}_{p_0},
	\end{equation*}
	where $\varphi _\varepsilon (t) = ((\| \nabla w_\varepsilon \|_2^2)/2) t^{N-2} - \left(  \int _{ \mathbb{R}^N } b_0G(w_\varepsilon ) \dx \right) t^N.$ Furthermore, computing the explicit maximum of $\varphi _{\varepsilon },$ we have
	\begin{equation*}
		\psi_\varepsilon (t_\varepsilon) \leq \frac{1}{N} \left( 2^\ast \int _{\mathbb{R}^N }b_0G(w_\varepsilon )  \dx  \right)^{- \frac{N-2}{2}}  \| \nabla w_\varepsilon \|^N_2+ \frac{B_k}{2} a_2^N \| w_\varepsilon \|_2 ^2- \lambda A_k a_1^N \| w_\varepsilon \|^{p_0}_{p_0},
	\end{equation*}
Next, using the elementary inequality 
 \begin{equation*}
     (a+b)^{\alpha } \leq a^{\alpha } + \alpha (a+b)^{\alpha -1}b,\ \alpha \geq 1,\ a,\, b>0,
 \end{equation*}
and \eqref{ozao1} we have $\| \nabla w_\varepsilon \|_2^N \leq \mathbb{S}^{N/2} + \mathcal{O}(\varepsilon ^{(N-2)/2}).$ Summing up,
 \begin{equation*}
     c(I) \leq \frac{1}{N}\left( 2^\ast \int _{\mathbb{R}^N }b_0G(w_\varepsilon )  \dx  \right)^{- \frac{N-2}{2}} \mathbb{S}^{N/2} + \mathcal{O}(\varepsilon^{(N-2)/2}) + \frac{B_k}{2} a_2^N \| w_\varepsilon \|_2 ^2- \lambda A_k a_1^N \| w_\varepsilon \|^{p_0}_{p_0}.
 \end{equation*}
Nevertheless, from \ref{g_porbaixo} and \ref{g_AR}, the following inequality holds
\begin{equation*}
	\frac{1}{N}\left( 2^\ast \int _{\mathbb{R}^N }b_0G(w_\varepsilon )  \dx  \right)^{- \frac{N-2}{2}} \mathbb{S}^{N/2}  \leq \frac{\mu_\ast - 2 }{2\mu_\ast } \mathbb{S}^{N/2}_{\bar{G}_b}.
\end{equation*}
Using that $\mathbb{S}_{\bar{G}_{b } }  = \| b \|^{-2/2^\ast }_\infty \mathbb{S}_{\bar{G}  }$ and considering each case of \ref{f_porbaixo}--i), \ref{f_porbaixo}--ii) and \ref{f_porbaixo}--iii), together with \eqref{ozao1}--\eqref{ozao3}, we can deduce that for $\varepsilon >0$ sufficiently small or for sufficiently large $\lambda$, $\mathcal{O}(\varepsilon^{(N-2)/2})  + (B_k/2) a_2^N \| w_\varepsilon \|_2 ^2- \lambda A_k a_1^N \| w_\varepsilon \|^{p_0}_{p_0}<0.$ The lemma is proved.
\end{proof}
Notice that if $b(x)\equiv 1$ and $g(s) = |s|^{2^\ast -2}s,$ then $\mu_\ast = 2^\ast$ and $\mathbb{S}_{\bar{G}_b} = \mathbb{S}.$ Consequently, we recover the well established estimate $c(I) <(1/N)\mathbb{S}^{N/2}.$
\section{Profile decomposition for bounded sequences}\label{s_profile}
The profile decomposition given by Theorem \ref{teo_tinta} stems from the general theory developed by K. Tintarev and F. Fieseler in \cite{tintabook}.  In their work, the authors establish an abstract framework for profile decomposition in separable Hilbert spaces by introducing the concepts of $D$-weak convergence and dislocation spaces relative to a suitable group of unitary operators $D$. Theorem \ref{teo_tinta}, proven in \cite[Chapter 5]{tintabook}, is the application of this framework to the specific case of $H=D^{1,2}(\mathbb{R}^N)$, where the group $D$ is generated by translations in $\mathbb{Z}^N$ and dilations by integer powers of $\gamma,$
\begin{equation*}
    D = \left\{ d_{y,j} : D^{1,2}(\mathbb{R}^N) \rightarrow D^{1,2} (\mathbb{R}^N):   d_{y,j}\varphi  = \gamma ^{\frac{N-2}{2}} \varphi (\gamma ^{j} (\cdot - y ) ),\ y \in \mathbb{Z}^N,\ j \in \mathbb{Z}   \right\}.
\end{equation*}
For the remainder of the paper, it is convenient to introduce a notation for the action of the group elements appearing in the decomposition. For a given $(y_k^{(n)}, j_k^{(n)}) \in \mathbb{Z}^N\times \mathbb{Z}$, we define the operator $d_{k,n}$ and its inverse $d_{k,n}^{-1}$ by
\begin{equation}\label{notacao}
	d_{k,n} \varphi = \gamma^{\frac{N-2}{2}j^{(n)}_k} \varphi(\gamma^{j_k^{(n)}}(\cdot - y^{(n)}_k)) \quad \text{and} \quad d^{-1}_{k,n} \varphi= \gamma^{-\frac{N-2}{2}j_k^{(n)}} \varphi(\gamma^{-j_k^{(n)}}\cdot + y_k^{(n)}).
\end{equation}
Using this notation, we can express \eqref{tinta1} and \eqref{tinta4} as $d^{-1}_{k,n} u_k \rightharpoonup w^{(n)}$ and
\begin{equation*}
    u_k - w^{(1)} - \sum _{n \in \mathbb{N}_\ast \setminus \{ 1 \}} d_{k,n}w^{(n)}  \rightarrow 0,\text{ when }k \rightarrow \infty,\text{ in }L^{2^\ast} (\mathbb{R}^N),
\end{equation*}
respectively. This convergence illustrates that if $w^{(n)} = 0 $ for all $n \in \mathbb{N}_\ast\setminus\{ 1 \},$ then $u_k \rightarrow w^{(1)}$ in $L^{2^\ast }(\mathbb{R}^N).$ Consequently, the term $\sum _{n \in \mathbb{N}\ast \setminus \{ 1 \}} d_{k,n}w^{(n)} $ can be interpreted as capturing the lack of compactness in the convergence of $(u_k)$ to its weak limit $w^{(1)}$ in $L^{2^\ast}(\mathbb{R}^N).$
\begin{remark}\cite[Lemma 5.4]{tintabook}\label{rem_tinta}
Under the conditions of Theorem \ref{teo_tinta}, if in addition $(u_k) \subset H^1 (\mathbb{R}^N)$ and $(\| u_k\|_2)$ is bounded, then $w^{(n)} = 0 $ for any $n \in \mathbb{N}_{-}.$ Thus, one can take $ \mathbb{N}_{-}= \emptyset.$ Moreover, for any $p \in (2,2^\ast),$
\begin{equation}\label{er_tinta}
    u_k - \sum _{n \in \mathbb{N} _0} w^{(n)} (\cdot + y_k^{(n)}) \rightarrow 0,\quad \text{as }k \rightarrow \infty,\ \text{in }L^p(\mathbb{R}^N),
\end{equation}
and the series in \eqref{er_tinta} converges absolutely in $H^1(\mathbb{R}^N)$ and uniformly in $k.$ Additionally, for $m,n \in \mathbb{N}_0,$
\begin{align}
	&  u_k (\cdot + y_k ^{(n)}   ) \rightharpoonup  w^{(n)},\quad \text{as }k\rightarrow \infty,\ \text{in }D^{1,2}(\mathbb{R}^N),\label{ehesse}& \\
	& |y_k ^{(n)}  -y_k ^{(m)} | \rightarrow +\infty,\quad \text{as }k \rightarrow \infty,\ \text{for }m \neq n. \nonumber &
  \end{align}
\end{remark}

\begin{remark}\label{r_salve}
	The translation sequence $(y^{(n)}_k)_k$ can be redefined as $0$ provided $(| \gamma ^{j_k^{(n)}} y_k^{(n)} |)_k,$ $n \in \mathbb{N}_\ast,$ is bounded. This assertion is established in the third part of the proof of \cite[Theorem 5.1]{tintabook}. 
\end{remark}

\subsection{Existence of a bounded Palais-Smale sequence}\label{s_bseq} Lemma \ref{l_mpgeometry} guarantees the existence of a Cerami sequence $(u_k) \subset H^1(\mathbb{R}^N),$ that is, $I(u_k) \rightarrow c(I)$ and $(1+ \| u_k \|_V)\| I'(u_k)\|_\ast \rightarrow 0.$ (see \cite[Theorem 6, p. 140]{MR1051888}). The main goal of this section is to prove that this sequence is bounded. To this end, we first establish the following technical lemma.
\begin{lemma}\label{l_theta}$  I(u) \geq I(tu) + 2^{-1}(1-t^2)I'(u) \cdot u,$ for all $ u \in H^1(\mathbb{R}^N)$ and $t \in [0,\theta_0].$
\end{lemma}
\begin{proof}
After some computations, we obtain
    \begin{multline*}
        I(u) - I(tu) - \frac{1-t^2}{2}I'(u) \cdot u   = \int_{\mathbb{R}^N} \frac{1-t^2}{2} f(x,u)u  - \left( F(x,u)- F(x,tu)  \right)\dx \\+  \int_{\mathbb{R}^N} \frac{1-t^2}{2} b(x)g(u)u  - b(x)\left( G(u)- G(tu) \right)\dx,\quad \forall \, t \in [0,\theta_0].
    \end{multline*}
The proof now follows by using \ref{f_tang} and \ref{g_AR}.
\end{proof}
We use the weak convergence decomposition given in Theorem \ref{teo_tinta} to prove that $(u_k)$ is a bounded sequence even in the presence of the critical term $g.$
\begin{proposition}\label{p_psbounded}
$(u_k)$ is a bounded sequence.
\end{proposition}
\begin{proof}
We proceed by contradiction and assume that $\| u_k \|_V \rightarrow \infty$ (up to a subsequence). Define the normalized sequence $v_k = u_k\| u_k \|^{-1}_V.$ Applying Theorem \ref{teo_tinta} and Remark \ref{rem_tinta} to $(v_k),$ we obtain a profile decomposition $(w^{(n)})_{n \in \mathbb{N} _\ast}$ as described in \eqref{tinta1}, \eqref{tinta4}, and \eqref{er_tinta}. Suppose first that $w^{(n)} = 0,$ for any $n \in \mathbb{N}_\ast.$ Then, by \eqref{tinta4} and \eqref{er_tinta}, we have $v_k \rightarrow 0,$ in $L^p(\mathbb{R}^N)$ for $p\in (2,2^\ast].$ Next, we fix $R>0$ and use \ref{f_geral} to get the estimate $\limsup _{k \rightarrow \infty} \int _{\mathbb{R}^N } F(x,Rv_k) \dx \leq \varepsilon c_2 R^2 ,$
where $c_2 = \limsup_{k \rightarrow \infty} \|v_k \|^2_2.$ Choosing $\varepsilon = 1/(4 c_2),$ we get
\begin{equation}\label{limsup}
\limsup _{k \rightarrow \infty} \int _{\mathbb{R}^N } F(x,Rv_k) \dx \leq \frac{R^2}{4}\quad\text{and}\quad \limsup _{k \rightarrow \infty} \int _{\mathbb{R}^N } b(x)G(Rv_k) \dx = 0.
\end{equation}
Now, define $t_k = R\| u_k \|_V^{-1}.$ For sufficiently large $k$, we have $t_k \in (0,\theta_0].$ By Lemma \ref{l_theta},
\begin{equation}\label{climsup}
    c(I) + o_k(1) = I(u_k) \geq I(t_k u_k) + \frac{1-t_k^2}{2}I'(u_k) \cdot u_k = I(Rv_k) + o_k(1).
\end{equation}
Furthermore,
\begin{equation*}
    I(R v_k) = \frac{R^2}{2} - \int_{\mathbb{R}^N} F(x,Rv_k) \dx - \int _{\mathbb{R}^N } b(x)G(R v_k) \dx,
\end{equation*}
Combining this with \eqref{limsup}, we obtain $\limsup_{k \rightarrow \infty} I(R v_k) \geq R^2/4.$ Taking the $\limsup_{k\rightarrow\infty }$ in \eqref{climsup} we arrive at the contradiction: $c(I) \geq R^2/4,$ for any $R>0.$ Therefore, there exists a nonzero $w^{(n)},$ for some $n \in \mathbb{N}_\ast = \mathbb{N}_0 \cup \mathbb{N}_{+}.$ In particular, up to a subsequence, one can find a set $U$ with positive Lebesgue measure such that
\begin{equation}\label{positivem}
    0< |w^{(n)}(x)| = \lim_{k \rightarrow \infty} |d^{-1}_{k,n} v_k (x)| = \lim_{k \rightarrow \infty}\frac{|d^{-1}_{k,n}u_k (x) |}{\|u_k \|_V},\quad \forall \, x \in U,
\end{equation}
where $d_{k,n} ^{-1}$ is given by \eqref{notacao}. This implies $|d^{-1}_{k,n}u_k (x)| \rightarrow + \infty, $ for all $x \in U.$ If $n \in \mathbb{N}_+,$ then \ref{g_selfsimilar} yields
\begin{equation*}
    \int _{\mathbb{R}^N} b(x)G(d^{-1}_{k,n}u_k) \dx  = \int _{\mathbb{R}^N}b(\gamma ^{j_k ^{(n)}}   (x-y_k^{(n)})) G(u_k) \dx.
\end{equation*}
By Fatou's Lemma, together with \ref{g_porbaixo} (or \ref{g_AR}) and $b_0>0$, there exists a subsequence, which we still denote by $(u_k),$ such that
\begin{equation}\label{gen_um}
\lim _{k \rightarrow \infty}\int _{\mathbb{R}^N }  \frac{b(x) G(u_k)}{\| u_k \|^2_V} \dx = \lim _{k \rightarrow \infty} \int _{U } b(\gamma ^{j_k ^{(n)}}   (x-y_k^{(n)}))  \frac{G(d^{-1}_{k,n} u_k )}{(d^{-1}_{k,n} u_k)^2} (d_{k,n} ^{-1} v_k)^2   \dx = + \infty.
\end{equation}
Consequently, we obtain the contradiction
\begin{equation}\label{boundedPSc}
    0 = \lim _{k \rightarrow \infty} \frac{c(I)}{\| u_k \|^2_V} = \lim _{k \rightarrow \infty} \frac{I(u_k)}{\| u_k \|^2_V}= \lim _{k \rightarrow \infty} \left[ \frac{1}{2} - \int _{\mathbb{R}^N } \frac{F(x,u_k)}{\| u_k \| ^2 _V} \dx - \int _{\mathbb{R}^N }\frac{b(x) G(u_k)}{\| u_k \|^2_V} \dx\right] = -\infty.
\end{equation}
The only remaining possibility is the existence of $w^{(n)} \neq 0,$ with $n \in \mathbb{N}_0.$ In this case, Fatou's Lemma, \ref{f_porbaixo}, and \eqref{positivem} imply
\begin{equation*}
    \lim _{k \rightarrow \infty}\int _{\mathbb{R}^N} \frac{F(x,u_k)}{\| u_k \|^2 _V}\dx = \lim _{k \rightarrow \infty } \int _{\mathbb{R}^N } \frac{F(\cdot + y_k^{(n)}, u_k (\cdot + y_k ^{(n)}  ))}{(u_k (\cdot + y_k ^{(n)}))^2} (v_k (\cdot + y_k ^{(n)}) )^2\dx = + \infty,
\end{equation*}
up to a subsequence. This leads to the same contradiction as in \eqref{boundedPSc}. In conclusion, the sequence $(u_k)$ cannot be unbounded.
\end{proof}

\subsection{Behavior of the functional under profile decomposition}\label{s_ss}
We now analyze the behavior of the energy functional $I$ with respect to the profile decomposition from Theorem \ref{teo_tinta}. Our analysis adapts the concentration-compactness arguments of \cite{MR2465979} and \cite[Lemmas 1.5, 1.7, 3.4, 5.5 and Corollary 5.2]{tintabook} to describe the asymptotic behavior of each component of the functional. For the reader's convenience and to make the exposition self-contained, we detail the key arguments. We note that a related analysis for a different setting was recently performed by M. Okumura in \cite{zbMATH07573810}.
\begin{lemma}\label{l_converge}
If $(u_k)$ and $(v_k)$ are bounded sequences in $H^1(\mathbb{R} ^N)$ such that $u_k - v_k \rightarrow 0$ in $L^p (\mathbb{R}^N)$ for some $p \in (2,2^\ast),$ then
\begin{equation}\label{Fconverge}
	\lim _ {k\rightarrow \infty} \int _{\mathbb{R}^N} F(x,u_k) - F(x,v_k) \dx=0.
\end{equation}
In particular, if $v_k = u,$ then
\begin{equation*}
\lim _{k \rightarrow \infty}\int _{\mathbb{R}^N} f(x,u_k) u_k \dx = \int _{\mathbb{R}^N} f(x,u)u \dx.
\end{equation*}
\end{lemma}
\begin{proof}
By an interpolation inequality, if $q<p$ then $\|u_k -v_k \| _q \leq \|u_k - v_k\|_2 ^{\alpha} \|u_k - v_k\| _{p} ^{1- \alpha}$ where $1/q = \alpha / 2  + (1-\alpha ) / p,$ and if $q>p$ then $\|u_k -v_k \| _q \leq \|u_k - v_k\|_p ^{\alpha} \|u_k - v_k\| _{2 ^\ast} ^{1- \alpha} $ for $1/q = \alpha / p  + (1-\alpha ) / 2 ^\ast.$ Summing up, $u_k - v_k \rightarrow 0$ in $L ^q (\mathbb{R}^N)$ for all $q \in (2,2 ^\ast).$ We use \ref{f_geral} as follows,
\begin{multline}
    \left| \int _{\mathbb{R}^N}  F(x,u_k) - F(x,v_k) \dx\right|  \leq \int _{\mathbb{R}^N} \int _{v_k} ^{u_k}  |f(s)| \ds   \dx\\
     \leq \int _{\mathbb{R}^N} \varepsilon \left( \frac{1}{2} \big| |u_k|u_k - |v_k| v_k \big| + \frac{1}{2^\ast} \big||u_k|^{2^\ast -1 }u_k - |v_k|^{2^\ast -1} v_k \big| \right) + \frac{C_\varepsilon}{p_\varepsilon}  \big| |u_k|^{p_\varepsilon -1} u_k - |v_k|^{p_\varepsilon - 1}v_k \big| \dx. \nonumber
\end{multline}
Applying the following inequality
\begin{equation}\label{ineq}
    \left| |a|^{q-1}a - |b|^{q-1} b \right| \leq q 2^{q-2}(|a|^{q-1} + |a-b|^{q-1})| a- b |,\quad q\geq 2,
\end{equation}
with $q=p_\varepsilon$ and using Hölder's inequality with exponents $p_\varepsilon$ and $p_\varepsilon / (p_\varepsilon - 1),$ we obtain
\begin{equation*}
  \int _{\mathbb{R}^N} \left| |u_k|^{p_\varepsilon -1} u_k - |v_k|^{p_\varepsilon - 1}v_k  \right| \dx \leq p_\varepsilon 2^{p_\varepsilon - 2} \left( \| u_k \|_{p_\varepsilon} ^{p_\varepsilon - 1} \| u_k - v_k \|_{p_\varepsilon} + \| u_k - v_k \|_{p_\varepsilon }^{p_\varepsilon}  \right).
\end{equation*}
Since $u_k -v_k \rightarrow 0$ in $L^{p_\varepsilon}(\mathbb{R}^N),$ the right-hand side of the above inequality converges to zero as $k \rightarrow \infty .$ Therefore,
\begin{equation*}
   \limsup_{k \rightarrow \infty} \left| \int _{\mathbb{R}^N}  F(x,u_k) - F(x,v_k) \dx\right|  \leq \varepsilon C ,\quad \forall \, \varepsilon >0,
\end{equation*}
for suitable $C>0$ depending only on the $\limsup $ of $\| u _k \|_q$ and $\| v_k \|_q,$ for $q=2$ or $q=2^\ast.$ Now, observe that $\hat{F}(x,s) := f(x,s) s $ satisfies \ref{f_geral}. Hence, the last statement is obtained by an application of the first part \eqref{Fconverge}. This finishes the proof.
\end{proof}
The following result describes the asymptotic behavior of the subcritical energy term. It shows that the limit of $\int_{\mathbb{R}^N} F(x, u_k) \dx$ splits into a sum that distinguishes between the weak limit $w^{(1)},$ which interacts with $F,$ and the other profiles, which interact with the periodic potential $F_P.$
\begin{proposition}\label{p_bef}
Let $(u_k)\subset H^1 (\mathbb{R}^N) $ be a bounded sequence and $(w ^{(n)}) _{n \in \mathbb{N}_{0} }$ the profiles given by Theorem \ref{teo_tinta} (cf. Remark \ref{rem_tinta}). Then
\begin{equation*}
\lim _{k \rightarrow \infty}\int _{\mathbb{R}^N} F(x, u_k) \dx =\int _{\mathbb{R}^N}F(x,w^{(1)}) \dx + \sum _{n \in \mathbb{N}_0\setminus \{ 1\} } \int _{\mathbb{R} ^N} F _P (x,w^{(n)}) \dx.
\end{equation*}
\end{proposition}
\begin{proof}
Let us denote $\Phi(u) = \int_{\mathbb{R}^N} F(x,u) \dx,$ $u \in H^1 (\mathbb{R}^N).$ We first prove that
\begin{equation}\label{remains}
\lim _{k \rightarrow \infty} \left[ \Phi \Big(\sum _{n \in \mathbb{N} _0} w^{(n)} (\cdot - y_k ^{(n)} ) \Big) -\sum _{n \in \mathbb{N} _0} \Phi \left(w^{(n)} (\cdot - y_k ^{(n)}) \right)  \right] = 0.
\end{equation}
Because the convergence in \eqref{er_tinta} is uniform in $k$ and $\Phi \in C^1,$ the series $\sum _{n \in \mathbb{N} _0} \Phi (w^{(n)} (\cdot - y_k ^{(n)}) )$ is also uniformly convergent in $k$ and we can reduce the proof of \eqref{remains} to the case where $\mathbb{N}_0$ is finite. Moreover, using the density of $C^\infty _0 (\mathbb{R}^N)$ in $H^1 _V(\mathbb{R}^N)$ we can assume that $w^{(n)} \in C^\infty _0 (\mathbb{R}^N)$ in \eqref{remains}. Consequently, by \eqref{tinta2} we have,
\begin{equation*}
\supp (w^{(n)} (\cdot - y_k ^{(n)})) \cap  \supp (w^{(m)} (\cdot - y_k ^{(m)}) ) = \emptyset,\quad \text{ for }m \neq n\text{ and }k\text{ large enough.}
\end{equation*}
Convergence \eqref{remains} follows by taking $k$ sufficiently large and denoting $W_k =  \bigcup _{n\in \mathbb{N}_0} \supp (w^{(n)} (\cdot - y_k ^{(n)}) ),$ to obtain
\begin{align*}
\int _{\mathbb{R} ^N} F\Big(x, \sum _{n \in \mathbb{N} _0} w^{(n)} (x - y_k ^{(n)} ) \Big) \dx &=  \int _{W_k} F \Big(x, \sum _{m \in \mathbb{N}_0}  w^{(m)}(\cdot - y_k ^{(m)}) \Big) \dx\\
&=\sum _{n \in \mathbb{N}_0} \int _{\supp (w^{(n)} )} F(x+y_k ^{(n)}, w^{(n)} ) \dx.
\end{align*}
We now proceed with the proof by noticing that Remark \ref{rem_tinta} and Lemma \ref{l_converge} yield
\begin{equation*}
\lim _{k \rightarrow \infty} \left[ \Phi(u_k ) - \Phi \Big(\sum _{n \in \mathbb{N} _0} w^{(n)} (\cdot - y_k ^{(n)} ) \Big) \right] = 0.
\end{equation*}
Next, following the same reasoning above, the series $\sum _{n \in \mathbb{N} _0} \Phi (w^{(n)} (\cdot - y_k ^{(n)} ) )$ converges uniformly in $k$ and by taking $\Phi_P (u) := \int_{\mathbb{R}^N} F _P(x,u)\dx,$ $u \in H^1_V (\mathbb{R}^N),$ one can use \ref{f_geral}, \ref{h_infinito} and the dominated convergence theorem to get
\begin{equation}\label{rod2}
\lim_{k \rightarrow \infty} \left[ \sum _{n \in \mathbb{N} _0} \Phi \left(w^{(n)} (\cdot - y_k ^{(n)} ) \right)-\Phi (w^{(1)}) - \sum _{\mathbb{N}_0\setminus \{ 1\} } \Phi_P (w^{(n)}) \right] = 0.
\end{equation}
The proof follows by combining \eqref{remains}--\eqref{rod2}.
\end{proof}
An immediate consequence of Proposition \ref{p_bef} is the next result for the periodic potential.
\begin{corollary}
$ \displaystyle
\lim _{k \rightarrow \infty}\int _{\mathbb{R}^N} F_P(x, u_k) \dx = \sum _{n \in \mathbb{N}_0} \int _{\mathbb{R} ^N} F_P (x,w^{(n)}) \dx.
$
\end{corollary}
Next we prove that the functional $u \mapsto \int_{\mathbb{R}^N}V(x)u^2\dx$ is sequentially weakly lower semicontinuous with respect to the profile decomposition of Theorem \ref{teo_tinta}.
\begin{proposition}\label{p_convv}
Let $(u_k)$ be a bounded sequence in $H^1 (\mathbb{R} ^N)$ and $(w ^{(n)}) _{n \in \mathbb{N}_{0} }$ given by Theorem \ref{teo_tinta}. Then
\begin{equation*}
\liminf _{k\rightarrow \infty}\int_{\mathbb{R}^N}V(x)u_k^2\dx \geq \int_{\mathbb{R}^N}V(x)|w^{(1)}|^2\dx+\sum _{n \in \mathbb{N}_0\setminus \{ 1\} }\int _{\mathbb{R}^N} V_P(x) |w^{(n)}|^2\dx.
\end{equation*}
\end{proposition}
\begin{proof}
Up to a suitable renumbering of $\mathbb{N}_0,$ and using successively the classical Brezis-Lieb lemma \cite{MR699419}, it suffices to prove that
\begin{multline}\label{comp_l2}
\int_{\mathbb{R}^N} V(x) u^2_k\dx = \int _{\mathbb{R}^N} \big((V(x))^{1/2} (u_k -w^{(1)} ) - (V_P(x))^{1/2} \sum _{n=2} ^m w^{(n)} (\cdot - y_k ^{(n)}) \big)^2\dx\\
+ \int_{\mathbb{R}^N}V(x)|w^{(1)}|^2\dx+\sum _{n=2}^m\int _{\mathbb{R}^N} V_P(x) |w^{(n)}|^2\dx+o_k(1).
\quad \forall \, m \in \mathbb{N}_0,
\end{multline}
We start by checking that \eqref{comp_l2} holds for $m=2.$ Indeed, because $u_k \rightharpoonup w^{(1)}$ in $H^1 (\mathbb{R}^N)$ (see \eqref{tinta1}) we have
\begin{equation}\label{comp_l2_part1}
\int _{\mathbb{R}^N} V(x) u^2_k\dx =\int_{\mathbb{R}^N} V(x) |w^{(1)}|^2\dx +  \int_{\mathbb{R}^N}V(x) (u_k - w^{(1)})^2 \dx+o_k(1).
\end{equation}
By Remark \ref{rem_tinta} and the fact that $u_k (\cdot +y_k ^{(2)} )\rightharpoonup w^{(2)}$ in $H^1 (\mathbb{R}^N),$
\begin{align}
    \int_{\mathbb{R}^N} V(x) (u_k - w^{(1)})^2 \dx& =  \int _{\mathbb{R}^N}  V(x+y_k ^{(2)}) (u_k (\cdot + y_k ^{(2)}) - w^{(1)} (\cdot + y_k ^{(2)}) ) ^2 \dx\nonumber \\
    & =\int _{\mathbb{R}^N} \big( (V(x+y_k ^{(2)}))^{1/2} (u_k (\cdot + y_k ^{(2)}) - w^{(1)}(\cdot + y_k ^{(2)}) ) - (V_P(x))^{1/2} w^{(2)}  \big)^2\dx \nonumber\\
    &\qquad+\int _{\mathbb{R}^N} V_P(x) |w^{(2)}|^2 \dx +o_k(1).\label{comp_l2_part2}
\end{align}
Substituting \eqref{comp_l2_part2} into \eqref{comp_l2_part1}, we obtain \eqref{comp_l2} for $m=2.$ We now prove that \eqref{comp_l2} holds for $m+1$ provided that it is true for $m.$ To do this, we argue as in \eqref{comp_l2_part2}, replacing $y_k ^{(2)}$ by $y_k ^{(m+1)},$
\begin{multline}\label{l_melhor}
    \int _{\mathbb{R}^N} \big((V(x))^{1/2} (u_k -w^{(1)} ) - (V_P(x))^{1/2} \sum _{n=2} ^m w^{(n)} (\cdot - y_k ^{(n)}) \big)^2\dx =\\ \int _{\mathbb{R}^N} \big((V(x))^{1/2} (u_k -w^{(1)} ) - (V_P(x))^{1/2} \sum _{n=2} ^{m+1} w^{(n)} (\cdot - y_k ^{(n)}) \big)^2\dx \\ +\int _{\mathbb{R}^N} V_P(x) |w^{(m+1)}|^2\dx + o_k(1).
\end{multline}
Next we use the induction hypothesis \eqref{comp_l2} again to obtain
\begin{multline}\label{comp_l2_part3}
\int _{\mathbb{R}^N}\big((V(x))^{1/2}(u_k - w^{(1)}) - (V_P(x))^{1/2}\sum _{n=2} ^m w^{(n)} (\cdot - y_k ^{(n)})\big)^2 \dx -\int _{\mathbb{R}^N} V_P(x) |w^{(m+1)}| ^2 \dx\\= \int _{\mathbb{R}^N} V(x) u^2_k \dx   - \int _{\mathbb{R} ^ N } V(x) |w^{(1)}|^2 \dx  -  \sum _{n=2} ^{m+1} \int _{\mathbb{R}^N}V_P(x) |w^{(n)}|^2 \dx  + o_k(1).
\end{multline}
Substituting \eqref{l_melhor} into \eqref{comp_l2_part3} we conclude that  \eqref{comp_l2} holds for $m+1.$
\end{proof}
We now collect some key properties of the self-similar nonlinearity $g$ that are essential for our analysis. These results are established in \cite[Section 5.2]{tintabook} and \cite{MR2409928,MR2409935}.
\begin{lemma}\label{l_geprop}
    If $g \in C(\mathbb{R})$ is a self-similar function \ref{g_selfsimilar}, then
    \begin{enumerate}[label=\roman*):]
        \item $g(s) = \gamma ^{-\frac{N+2}{2} j   }  g(\gamma ^{ \frac{N-2}{2}j  }   s )$ and condition \ref{g_geral} holds.
        \item For a given $M \in \mathbb{N},$ there exists $C=C(M)>0$ such that for any $a_1, \ldots, a_M \in \mathbb{R}$ we have
\begin{equation*}
	\left|  G\left( \sum _{m=1}^M a_m \right)  - \sum _{m=1}^M G(a_m)\right| \leq C \sum^M _{m \neq n}|a_m|^{2^\ast -1}|a_n|.
\end{equation*}
    \end{enumerate}
\end{lemma}
\begin{proof}
    i): Take the interval $L=[\gamma ^{-\frac{N-2}{2}}, \gamma ^{\frac{N-2}{2}}].$ By continuity, there exists $C=C(L)>0$ such that $|g(s)| \leq C |s|^{2 ^\ast -1 },$ for all $s \in L.$ Let $0<s<\gamma ^{-\frac{N-2}{2}}$ or $s>\gamma ^{\frac{N-2}{2}}.$ Then, in either case, there exists $j \in \mathbb{Z}$ such that $\gamma ^{\frac{N-2}{2} j}s \in L.$ Consequently, 
	\begin{equation*}
	\gamma ^{\frac{N+2}{2} j}|g(s)| = |g(\gamma ^{\frac{N-2}{2}j}s)|\leq \gamma ^{\frac{N+2}{2} j} C |s|^{2^\ast -1}.
	\end{equation*}
	The case where $s<0$ is analogous.\\
	ii): We start by proving the existence of $C>0$ such that
	\begin{equation}\label{case_um}
	\left|G(a_1 + a_2) - G(a_1) - G(a_2)\right| \leq C \left( |a_1| |a_2| ^{2^\ast  -1} + |a_1| ^{2^\ast  -1} |a_2| \right).
	\end{equation}
	To do this, we consider the class of intervals $L_k = [-\gamma ^{\frac{N-2}{2} k}, \gamma ^{\frac{N-2}{2} k}],$ and take $k,$ $k_0 \in \mathbb{N}$ in such a way that $\gamma ^{\frac{N-2}{2} k_0} > 2$ and $\gamma ^{\frac{N-2}{2} (k-k_0)} >1.$ With this choice, one can see that if $a_1$ and $a_2 \in L_{k-k_0},$ then $a_1 + a_2 \in L_{k}.$ 	Because $G$ is a locally Lipschitz function, there is $C >0$ such that $|G(s) - G(t)| \leq C |s-t|,$ for any $s,t \in L_k.$ The proof follows by considering several cases.\\
	\textit{Case 1:} Assume that $a_1$ and $a_2 \in L_{k-k_0}$ with $|a_1| \leq 1 \leq |a_2|.$ Thus
	\begin{equation*}
	\left|G(a_1 + a_2) - G(a_1) - G(a_2)\right| \leq C (|a_1| + |G(a_1)| ).
	\end{equation*}
	By condition \ref{g_geral} we can estimate
	\begin{equation*}
	|a_1| + |G(a_1)| \leq C(|a_1| |a_2| ^{2^\ast  -1} + |a_1| ^{2^\ast  -1} |a_2| ).
	\end{equation*}
	\textit{Case 2:} Assume that $a_1$ and $a_2 \in L_{k-k_0}$ with $|a_1| ,$ $ |a_2| \geq 1.$ Then, there exists $j_1 \in \mathbb{Z},$ $j_1 \leq 0,$ such that for $b_1 := \gamma ^{\frac{N-2}{2}j_1} a_1,$ one has $|b_1| \leq 1.$ Because $b_1  \in L_{k-k_0},$ we know that $b_1 + a_2 \in I,$ hence by the first case, we have the following estimate
	\begin{align*}
	|G(b_1 + a_2) - G(b_1) - G(a_2)| &\leq \gamma ^{\frac{N-2}{2}j_1} C(|a_1| ^{2^\ast  -1} |a_2| + |a_1| |a_2| ^{2^\ast  -1} )\\
	&\leq C(|a_1| ^{2^\ast  -1} |a_2| + |a_1| |a_2| ^{2^\ast  -1} ).
	\end{align*}
	This allows us to obtain the following inequality
	\begin{multline*}
	\left|G(a_1 + a_2) - G(a_1) - G(a_2)\right| \leq \\ \left|G(b_1 + a_2) - G(b_1) - G(a_2)\right| + \left|G(a_1 + a_2) - G(b_1 + a_2) + G(b_1) - G(a_1) \right|,
	\end{multline*}
	with
	\begin{equation*}
	|G(a_1 + a_2) - G(a_1) - G(b_1 + a_2) + G(b_1)| \leq 2C |a_2| \leq C |a_1| ^{2 ^\ast } |a_2|.
	\end{equation*}
	\textit{Case 3:} Assume that $a_1$ and $a_2 \in L_{k-k_0}$ with $|a_1|,$ $|a_2| \leq 1.$ Let $b_1  = \gamma ^{\frac{N-2}{2}j_0} a_1$ and $b_2 = \gamma ^{\frac{N-2}{2}j_0} a_2.$ Take $j_0 \in \mathbb{Z}$ such that $|b_1| \geq 1 $ or $|b_2| \geq 1.$ Consequently we can use the first or the second case to get that
	\begin{align*}
	\gamma ^{Nj_0} |G(a_1 + a_2) - G(a_1) - G(a_2)| &= |G(b_1+b_2) - G(b_1) - G(b_2)|\\
	& \leq \gamma ^{Nj_0} C (|a_1| ^{2^\ast  -1} |a_2| + |a_1| |a_2| ^{2^\ast  -1} ).
	\end{align*}
	\textit{Case 4:} We suppose that $a_1\not \in L_{k-k_0}$ or $a_2 \not \in L_{k-k_0}.$ In this case we can argue as before, taking $j_\ast \in \mathbb{Z}$ such that $b_1  = \gamma ^{\frac{N-2}{2}j_\ast} a_1,$ $b_2 = \gamma ^{\frac{N-2}{2}j_\ast } a_2 \in L_{k - k_0},$ and applying one of the previous cases. This proves \eqref{case_um}.
	\textit{The general case:} Next, we proceed using an induction argument. Suppose that the following inequality holds for $l \in \mathbb{N},$
	\begin{equation*}
		\left|  G\left( \sum _{m=1}^l a_m \right)  - \sum _{m=1}^l G(a_m)\right| \leq C \sum _{m \neq n}^l|a_m|^{2^\ast -1}|a_n|,
	\end{equation*}
	for some $C = C(l).$ Given $a_{l+1} \in \mathbb{R},$ one can use the induction hypothesis together with \eqref{case_um} to get
		\begin{align*}
		\left|  G\left( \sum _{m=1}^{l+1} a_m \right)  - \sum _{m=1}^{l+1} G(a_m)\right| & \leq \left|  G\left( \sum _{m=1}^{l} a_m  + a_{l+1}  \right)  - G\left( \sum _{m=1}^{l} a_m   \right) -  G(a_{l+1})  \right| \\ &\qquad +\left|  G\left( \sum _{m=1}^l a_m \right)  - \sum _{m=1}^l G(a_m)\right| \leq C \sum _{m \neq n}^{l+1}|a_m|^{2^\ast -1}|a_n|. \qedhere
	\end{align*}
\end{proof}
We now establish the counterpart to Lemma \ref{l_converge} for the critical nonlinearity $g.$
\begin{lemma}\label{l_Gconverge}
	Let $(u_k)$ and $(v_k)$ be bounded sequences in $D^{1,2}(\mathbb{R}^N),$ such that $u_k - v_k \rightarrow 0$ in $L^{2^\ast }(\mathbb{R}^N).$ Then
\begin{equation*}
	\lim _ {k\rightarrow \infty} \int _{\mathbb{R}^N} b(x)(G(u_k) - G(v_k) )\dx=0.
\end{equation*}
Particularly, if $v_k = u,$ then
\begin{equation*}
	\lim _ {k\rightarrow \infty} \int _{\mathbb{R}^N} b(x)g(u_k)u_k  \dx=  \int _{\mathbb{R}^N}  b(x)g(u) u \dx.
\end{equation*}
\end{lemma}
\begin{proof}
The proof is quite similar to that of Lemma \ref{l_converge}. By \eqref{ineq} we have,
\begin{align*}
	\left| \int _{\mathbb{R}^N}  b(x)(G(u_k) - G(v_k)) \dx\right|  & \leq \| b \| _\infty \int _{\mathbb{R}^N} \int _{u_k} ^{v_k}  |g(s)| \ds   \dx \\
	&\leq \| b \| _\infty\int _{\mathbb{R}^N}  \frac{a_\ast}{2^\ast} \big|  |u_k|^{2^\ast -1 }u_k - |v_k|^{2^\ast -1} v_k \big| \dx\\
	 & \leq \| b \| _\infty a_\ast  2^{2^\ast - 2} \left( \| u_k \|_{2^\ast} ^{2^\ast - 1} \| u_k - v_k \|_{2^\ast} + \| u_k - v_k \|_{2^\ast }^{2^\ast}  \right)\rightarrow 0,\quad \text{as }k \rightarrow \infty.
\end{align*}
The last statement follows from the fact that $\bar{G}(s)=g(s)s$ also satisfies \ref{g_geral}.
\end{proof}
We now provide the counterpart to Proposition \ref{p_bef} for the critical energy term.
\begin{proposition}\label{p_convast}
	Let $(u_k)\subset D^{1,2}(\mathbb{R}^N)$ be a bounded sequence and $(w ^{(n)}) _{n \in \mathbb{N}_{\ast} }$ be the profiles given by Theorem \ref{teo_tinta}. Then
	\begin{equation*}
		\lim _{k \rightarrow \infty} \int _{\mathbb{R}^N } b(x) G(u_k) \dx =  \int _{\mathbb{R}^N} b(x)G(w^{(1)}) \dx  + \sum _{n \in \mathbb{N} _{ 0} \setminus \{ 1 \}} \int _{\mathbb{R}^N} b_P(x)G(w^{(n)}) \dx + \sum _{n \in \mathbb{N} _{ +}} \int _{\mathbb{R}^N} b^{(n)}  G(w^{(n)}) \dx,
	\end{equation*}
	where $b^{(n)} = b(z^{(n)})$ if $(y_k^{(n)})_k$ is bounded, with $z^{(n)} = \lim_{k \rightarrow \infty } y_k^{(n)}$, and $b^{(n)} = b_P(0)$ if $|y_k^{(n)}| \to \infty,$ as $k \rightarrow \infty$ (up to a subsequence).
\end{proposition}
\begin{proof}
    The idea is similar to the one used in the proof of Proposition \ref{p_bef}. Considering $\Phi_{\ast}(u)=\int_{\mathbb{R}^N}b(x) G (u) \dx, $ $u \in D^{1,2}(\mathbb{R}^N),$ by Lemma \ref{l_Gconverge} and \eqref{tinta4}, we have 
	\begin{equation}\label{gen_abacaxi}
	\lim _{k \rightarrow \infty}\left[\Phi _\ast (u_k) - \Phi _\ast\left( \sum _{n \in \mathbb{N}_{\ast }} d_{k,n}  w^{(n)}\right) \right] =0,
	\end{equation}
 where $d_{k, n} w^{(n)}$ is given by \eqref{notacao}. The uniform convergence of the series in \eqref{tinta4} allows us to reduce to the case where $\mathbb{N}_\ast$ is finite. If $(| \gamma ^{j_k^{(n)}} y_k^{(n)} |)_k,$ $n \in \mathbb{N}_+,$ is bounded, we have $z^{(n)} = 0$ (see Remark \ref{r_salve}). If not, $b^{(n)} = b(z^{(n)}),$ for $z^{(n)} = \lim_{k \rightarrow \infty } y_k^{(n)}$ and $b^{(n)} = b_P(0),$ for $|y_k^{(n)}| \to \infty,$ up to a subsequence. Using this, defining $\Phi _{\ast,P}(u) = \int _{\mathbb{R}^N }  b_P(x) G(u)\dx$ and $\Phi_{\ast , n } (u)   = \int _{\mathbb{R}^N} b^{(n)} G(u) \dx,$ $u \in D^{1,2}(\mathbb{R}^N),$ we are going to prove that
	\begin{equation}\label{gen_banana}
	\lim _{k \rightarrow \infty}\left[ \sum _{n \in \mathbb{N}_{\ast }}  \Phi_\ast( d_{k, n} w^{(n)}) - \Phi _\ast (w ^{(1)}) -  \sum _{n \in \mathbb{N}_{ 0 }   \setminus \{ 1  \}}  \Phi _{\ast,P} (w ^{(n)}) - \sum _{n \in \mathbb{N}_{ + }}  \Phi _{\ast,n } (w ^{(n)})  \right]=0.
	\end{equation}
	We make use of the expression
	\begin{equation*}
		\Phi_\ast( d_{k, n} w^{(n)}) = \int b(\gamma ^{-j_k^{(n)}}x  + y_k^{(n)}    ) G(w^{(n)}) \dx.
	\end{equation*}
	Let $n \in \mathbb{N}_+. $ Our analysis is split into two cases: $\lim_{k \rightarrow \infty } y_k^{(n)} = z^{(n)}$ or $|y_k^{(n)}| \rightarrow \infty.$ For the first case, by \ref{g_selfsimilar} and the Lebesgue theorem, we have $\Phi_\ast( d_{k, n} w^{(n)}) \rightarrow \Phi _{\ast,n } (w ^{(n)}).$ For the second case,
	\begin{align*}
			\Phi_\ast( d_{k, n} w^{(n)}) &= \int \left( b(\gamma ^{-j_k^{(n)}}x  + y_k^{(n)}    ) - b_P  (\gamma ^{-j_k^{(n)}}x  + y_k^{(n)}    ) \right) G(w^{(n)}) \dx  + \int _{\mathbb{R}^N} b_P  (\gamma ^{-j_k^{(n)}}x    ) G(w^{(n)}) \dx \\
			&\quad \rightarrow  \int _{\mathbb{R}^N} b_P  ( 0     ) G(w^{(n)}) \dx = \Phi _{\ast,n } (w ^{(n)}),\quad \text{as }k \rightarrow \infty.
	\end{align*}
	The case where $n \in \mathbb{N}_0$ and $\Phi_\ast( d_{k, n} w^{(n)}) \rightarrow \Phi _{\ast,P} (w^{(n)}),$ $k \rightarrow \infty,$ follows as above considering $j_k ^{(n)} = 0 .$ It remains to prove that
	\begin{equation}\label{giant1}
	\lim _{k\rightarrow\infty}\left[ \Phi _\ast \left( \sum _{n \in \mathbb{N}_{\ast }} d_{k, n} w^{(n)}\right)-  \sum _{n \in \mathbb{N}_{\ast }} \Phi _\ast ( d_{k, n} w^{(n)} ) \right]=0.
	\end{equation}
	Indeed, by Lemma \ref{l_geprop} we have
	\begin{equation*}
	\left| \Phi _\ast \left( \sum _{n \in \mathbb{N}_{\ast }} d_{k, n} w ^{(n)}\right)-  \sum _{n \in \mathbb{N}_{\ast }} \Phi _\ast (d_{k, n} w ^{(n)} ) \right| \leq C\sum _{_{\substack{m \neq n\\ m,n  \in \mathbb{N} _{\ast }}} } \int _{\mathbb{R}^N} |d_{k, n} |^{2  ^\ast -1}  | d_{k, m} | \dx.
	\end{equation*}
	On the other hand, by a change of variable it holds
	\begin{equation*}
	\int _{\mathbb{R}^N} | d_{k, n}  |^{2  ^\ast -1}  | d_{k, m} | \dx = \int _{\mathbb{R}^N} |w^{(n)}| ^{2  ^\ast -1} g_k(|w^{(m)}|) \dx,
	\end{equation*}
	where 
	\begin{equation*}
	g_k(|w^{(m)}|) = \gamma ^{\frac{N-2}{2} (j _k ^{(m)} - j _k ^{(n)} )}  w^{(m)} \left( \gamma ^{j _k ^{(m)} - j _k ^{(n)} } (\cdot-\gamma ^{j_k ^{(n)}} (y_k ^{(m)} - y _k ^{(n)})  )\right) \rightharpoonup 0\text{ in } D^{1,2} (\mathbb{R}^N),
	\end{equation*}
	due to \eqref{tinta2} and \cite[Lemma 5.1]{tintabook}. Since $v \mapsto \int _{\mathbb{R}^N} |w^{(n)}| ^{2 ^\ast -1} v \dx $ is a continuous linear functional in $D^{1,2} (\mathbb{R}^N)$ we conclude \eqref{giant1}.
\end{proof}

\begin{corollary}\label{c_convbe}
	$ \displaystyle
	\lim _{k \rightarrow \infty}\int _{\mathbb{R}^N} b_P(x)G( u_k) \dx = \sum _{n \in \mathbb{N}_0 } \int _{\mathbb{R} ^N} b_P(x) G (w^{(n)}) \dx+ \sum _{n \in \mathbb{N} _{ +}} \int _{\mathbb{R}^N} b_P(0 )G(w^{(n)}) \dx.
	$
\end{corollary}
\section{$\mathbb{Z}^N$--Periodic case: Proof of Theorem \ref{teo_periodic}}\label{s_teo_periodic}
Our proof of the existence of ground states for \eqref{QP} relies on the well-known fact that nontrivial critical points are bounded away from zero in norm.
\begin{lemma}\label{l_salvou}
	There exists $C>0$ such that $\|u\|_{V_P} \geq C,$ for all $u \in \mathrm{Crit}(I_P).$
\end{lemma}
\begin{proof}
By \ref{f_geral} and \ref{g_geral}, for $u \in \mathrm{Crit}(I_P),$ we have $\| u \|^2_{V_P} \leq \varepsilon  \| u \|_2^2 +C_\varepsilon \| u \|_{p_\varepsilon } ^{p_\varepsilon} + (\varepsilon +  a_\ast \|b_P\|_\infty ) \|u\|^{2^\ast }_{2^\ast }.$ The conclusion follows by the appropriate use of Sobolev inequalities.
\end{proof}
Next we consider $\mathcal{G}_S(I_P),$ the ground state level associated with $I_P,$ which is defined as in \eqref{defgs}, replacing $I$ by $I_P.$
\begin{proof}[Proof of Theorem \ref{teo_periodic} completed] By the results of Section \ref{s_bseq}, there is a bounded sequence $(u_k) \subset H^1 _{V_P}(\mathbb{R}^N)$ such that $I_P(u_k) \rightarrow c(I_P)$ and $\| I'_P(u_k) \|_\ast \rightarrow 0.$ Let $(w^{(n)})_{n \in \mathbb{N}_\ast},$ $(y_k^{(n)})_k$ and $(j_k^{(n)}),$ $n \in \mathbb{N}_\ast,$ be the profiles given by Theorem \ref{teo_tinta} and Remark \ref{rem_tinta}. The proof is divided into several steps.

\noindent \textit{Step 1)}: Each $w^{(n)},$ $n \in \mathbb{N}_0,$ is a critical point of $I_P.$ In order to prove this we use a density argument, taking $\varphi \in C^{\infty}_0(\mathbb{R}^N).$ By \eqref{tinta1} and \eqref{tinta2} together with \ref{f_geral} and \ref{g_geral}, Lebesgue convergence theorem yields
\begin{multline*}
    \lim _{k \rightarrow \infty }(u_k, \varphi (\cdot - y_k^{(n)}))_{V_P} = (w^{(n)} , \varphi)_{V_P} \text{ and }\\ \lim _{k \rightarrow \infty} \int _{\mathbb{R}^N} \left(f_P(x , u_k    ) + b_P(x)g(u_k  ) \right)\varphi (\cdot - y_k^{(n)})  \dx = \int _{\mathbb{R}^N}\left( f_P(x,w^{(n)}) + b_P(x)g(w^{(n)} ) \right) \varphi\dx.
\end{multline*}
 up to a subsequence. In particular, $0 = \lim _{k \rightarrow \infty }\left( I'_P (u_k) \cdot (\varphi (\cdot - y_k ^{(n)})) \right) = I'_P(w^{(n)}) \cdot \varphi.$ Because $C^\infty _0 (\mathbb{R}^N)$ is dense in $H^1 _{V_P}(\mathbb{R}^N),$ we have $I_P'(w^{(n)}) =0,$ for any $n \in \mathbb{N}_0.$

\noindent  \textit{Step 2)}: There is $n_0 \in \mathbb{N}_0$ with $w^{(n_0)} \neq 0.$ Suppose, contrary to our claim, that $w^{(n)} = 0 $ for all $n \in \mathbb{N}_0.$ By Remark \ref{rem_tinta}, one has $u_k \rightarrow 0$ in $L^p (\mathbb{R}^N),$ for any $p \in (2,2^\ast).$ Consequently, we can apply Lemma \ref{l_converge} to get
\begin{equation}\label{eq_minimax}
c(I_P) = \frac{1}{2} \| u_k \|^2_{V_P} - \int _{\mathbb{R}^N} b_P(x) G(u_k)\dx + o_k(1)\quad  \text{and}\quad 0 = \| u_k \|^2_{V_P} - \int _{\mathbb{R}^N} b_P(x)g(u_k) u_k \dx + o_k(1).
\end{equation}
Let us denote
\begin{equation*}
    l_\ast  := \limsup _{k \rightarrow \infty } \| u_k \|^2_{V_P} = \limsup _{k \rightarrow \infty } \int _{\mathbb{R}^N} b_P(x) g(u_k) u_k \dx\leq \limsup _{k \rightarrow \infty } \int _{\mathbb{R}^N} \bar{G} _{b_P} (u_k) \dx.
\end{equation*}
On the other hand, since 
\begin{equation*}
	\left(  \int _{\mathbb{R}^N}  \bar{G}_{b_P} (u)  \dx \right) ^{ 2/2^\ast}\mathbb{S}_{\bar{G} _{b_P}}   \leq \| \nabla u\|_2 ^2 ,\  u \in D^{1,2} (\mathbb{R}^N),
\end{equation*}
one can use the fact that $V_P \geq 0,$ to obtain the inequality 
\begin{equation*}
l_\ast \leq \limsup _{k \rightarrow \infty } \int _{\mathbb{R}^N} \bar{G} _{b_P} (u_k) \dx \leq \left( \mathbb{S}^{-1}_{\bar{G}_{b_P}}l _\ast \right) ^{2^\ast /2}.
\end{equation*}
Because $c(I_P)>0,$ $l_\ast >0$ and we have $l_\ast \geq \mathbb{S}_{\bar{G}_{b_P}}^{N/2}.$ Substituting this inequality into the first equation of \eqref{eq_minimax} and using \ref{g_AR}, one has
\begin{equation*}
    c(I_P)\geq \left( (\mu _\ast -2) /(2 \mu _\ast ) \right)\mathbb{S}_{\bar{G}_{b_P}}^{N/2},
\end{equation*}
which is a contradiction with Lemma \ref{l_minimax}.

\noindent  \textit{Step 3)}: $\mathcal{G}_S(I_P)\leq I_P(w^{(n_0)}) \leq c(I_P).$ Indeed, by Propositions \ref{p_bef} and \ref{p_convv} together with Corollary \ref{c_convbe}, we have
\begin{align}
    c(I_P) &= \lim _{k \rightarrow \infty }\left[ \frac{1}{2} \| u_k \|^2_{V_P} - \int _{\mathbb{R}^N} F_P(x,u_k) \dx - \int _{\mathbb{R}^N}  b_P(x)G(u_k)\dx \right]\nonumber\\
    &\geq \sum _{n \in \mathbb{N}_0} I_P(w^{(n)}) + \sum _{n \in \mathbb{N}_+} J(w^{(n)}),\label{c_ddoisP}
\end{align}
where $J$ is the $C^1$--functional given by
\begin{equation*}
    J(u) =\frac{1}{2}\int _{\mathbb{R}^N} |\nabla u |^2 \dx - \int _{\mathbb{R}^N} b_P(0)G(u) \dx,\quad u \in D^{1,2}(\mathbb{R}^N).
\end{equation*}
By an argument similar to that of Step 2), $ 0= \lim _{k \rightarrow \infty }\left( I'_P (u_k) \cdot (d_{k,n} \varphi ) \right) = J (w^{(n)}) \cdot \varphi =0,$ for any $\varphi \in C^\infty _0 (\mathbb{R}^N)$ and $n \in \mathbb{N}_+.$ Because of this, $J' (w^{(n)}) = 0,$ $n \in \mathbb{N}_+.$ Moreover, from \ref{g_AR}, 
 \begin{equation}\label{jotamais}
      J(w^{(n)}) = \frac{1}{2} \int _{\mathbb{R}^N} g(w^{(n)})w^{(n)} - 2 G(w^{(n)}) \dx \geq 0, \quad \forall \, n \in \mathbb{N}_+.
 \end{equation}
Using this fact in \eqref{c_ddoisP}, we obtain $c(I_P) \geq I_P(w^{(n_0)}).$

\noindent\textit{Step 4)}: $\mathcal{G}_S(I_P) >0$ and it is attained. Let $(v_k) \subset \mathrm{Crit}(I_P)$ be a minimizing sequence for $\mathcal{G}_S(I_P),$ i.e., $I_P(v_k) \rightarrow \mathcal{G}_S(I_P)$ and $I'_P(v_k) = 0.$ Hypotheses \ref{f_tang} and \ref{g_AR} guarantee that the functions 
\begin{equation*}
	\mathcal{F}_P(x,s) := \frac{1}{2}f_P(x,s)s - F_P(x,s)\quad\text{and}\quad\mathcal{K}_P(x,s) := \frac{1}{2}b_P(x)g(s)s -b_P(x) G(s),
\end{equation*}
are nonnegative. Furthermore, $I_P(u) = \int _{\mathbb{R}^N }\mathcal{F}_P(x,u) + \mathcal{K}_P(x,u) \dx,$ for any $u \in \mathrm{Crit}(I_P).$ Thus $\mathcal{G}_S(I_P)\geq 0$ and we can use the same argument from Proposition \ref{p_psbounded}, to see that $(v_k)$ is bounded in $H^1_{V_P} (\mathbb{R}^N).$ To simplify notation, we use the same letters $(w^{(n)})_{n \in \mathbb{N}_\ast},$ $(y_k ^{(n)})_k$ and $(j_k ^{(n)})_k,$ $n \in \mathbb{N}_\ast$ to indicate the profiles given in Theorem \ref{teo_tinta} for the sequence $(v_k).$ If $w^{(n)} = 0 $ for any $n \in \mathbb{N}_0,$ we can use Lemma \ref{l_salvou} and proceed as in Step 2) to obtain the estimate
\begin{equation*}
    c(I_P)\geq \mathcal{G}_S(I_P)\geq \left( (\mu _\ast -2) /(2 \mu _\ast ) \right)\mathbb{S}_{\bar{G}_{b_P}}^{N/2}.
\end{equation*}
Likewise, this is a contradiction with Lemma \ref{l_minimax}. Thus there is $w^{(m_0)} \neq 0,$ for some $m_0 \in \mathbb{N}_0.$ On the other hand, since $0=I'_P(v_k) \cdot \varphi (\cdot - y_k^{(m_0)}) = I'_P(v_k(\cdot + y_k ^{(m_0)})) \cdot \varphi,$ for any $\varphi \in H^1_{V_P} (\mathbb{R}^N),$ we can argue as in Step 1) to conclude that $I'_P(w^{(m_0)}) = 0.$ Moreover, the convergence in \eqref{ehesse} leads to $v_k(x + y_k ^{(m_0)}) \rightarrow w^{(m_0)}(x)$ a.e. in $\mathbb{R}^N,$ up to a subsequence. Because $(v_k(\cdot + y_k ^{(m_0)}) ) \subset \mathrm{Crit}(I_P),$ we can apply Fatou's lemma as follows
\begin{align*}
    \mathcal{G}_S(I_P) &= \lim _{k \rightarrow \infty} I_P (v_k(\cdot + y_k ^{(m_0)}) )\\
    &= \liminf_{k \rightarrow \infty}  \int _{\mathbb{R}^N}  \mathcal{F}_P(x, v_k(\cdot + y_k ^{(m_0)}) ) + \mathcal{K}_P(x,v_k(\cdot + y_k ^{(m_0)})) \dx \\
    & \geq \int _{\mathbb{R}^N} \mathcal{F}_P(x,w^{(m_0)}) + \mathcal{K} _P(x,w^{(m_0)}) \dx = I_P(w^{(m_0)} ) \geq \mathcal{G}_S(I_P).
\end{align*}

\noindent \textit{Step 5)}: Assuming \ref{h_infinito}, one has $c(I_P) = \mathcal{G}_S(I_P).$ Indeed, we can use Remark \ref{r_compar} to obtain $c(I_P)\leq I_P (w^{(m_0)} ) = \mathcal{G}_S(I_P).$ The conclusion follows from Step 3).
\end{proof}

\section{Asymptotic $\mathbb{Z}^N$--periodic case: Proof of Theorem \ref{teo_compact}}\label{s_comp}
\begin{proof}
The idea of the proof is similar to the one used to prove Theorem \ref{teo_gs} and once more we divide it into some steps. Let $(u_k) \subset H^1 _{V}(\mathbb{R}^N)$ be a bounded sequence such that $I(u_k) \rightarrow c(I)$ and $\| I'(u_k) \|_\ast \rightarrow 0,$ together with its profiles $(w^{(n)})_{n \in \mathbb{N}_\ast},$ $(y_k^{(n)})_k$ and $(j_k^{(n)})_k,$ $n \in \mathbb{N}_\ast,$ given by Theorem \ref{teo_tinta} and Remark \ref{rem_tinta}. We compare the minimax level $c(I)$ with the level of the associated critical limit problem, which is generated by the dilation profiles ($n \in \mathbb{N}_{+}$). More precisely, consider the functional $J_n: D^{1,2}(\mathbb{R}^N) \rightarrow \mathbb{R}$ for this limit problem, given by
\begin{equation*}
	J_n(u) = \frac{1}{2} \int _{\mathbb{R}^N} |\nabla u |^2 \dx-\int _{\mathbb{R}^N} b^{(n)}G (u) \dx,
\end{equation*}
where $b^{(n)} = b(z^{(n)})$ if $(y_k^{(n)})_k$ is bounded, with $z^{(n)} = \lim_{k \rightarrow \infty } y_k^{(n)}$, and $b^{(n)} = b_P(0)$ if $|y_k^{(n)}| \to \infty,$ as $k \rightarrow \infty,$ up to a subsequence. Here the corresponding minimax level is given by
\begin{equation*}
	c(J_n ) : = \inf _{\xi \in \Gamma _{J_{ n } }} \sup_{t \geq 0} J_{ n } ( \xi (t)),\quad  n \in \mathbb{N}_+,
\end{equation*}
where $ \Gamma _{J_n } = \left\{  \xi \in C([0,\infty), D^{1,2}(\mathbb{R}^N) ) : \xi (0)=0\text{ and }\lim _{t \rightarrow \infty} J _n (\xi (t)) = - \infty \right\}.$ The following lemma, established in Proposition \ref{p_jota} (Appendix \ref{s_app_minimax}), provides the crucial energy comparison.
\begin{lemma}\label{l_jota}
	$c(I) \leq c(J_n).$
\end{lemma}

\noindent\textit{Step 1)}: $w^{(1)}$ is a critical point of $I,$ $w^{(n)}$ is a critical point of $I_P,$ for $n \in \mathbb{N}_0,$ and $w^{(n)}$ is a critical point of $J,$ for $n \in \mathbb{N}_+.$ To show this, one can use the Lebesgue convergence theorem together with the density of $C_0^\infty (\mathbb{R}^N)$ in $H^1 _V(\mathbb{R}^N)$ and $D^{1,2}(\mathbb{R}^N)$ to obtain, up to a subsequence: $I'(w^{(1)})\cdot \varphi = \lim _{k \rightarrow \infty} I'(u_k) \cdot \varphi =0;$ $I'_P(w^{(n)}) \cdot \varphi = \lim _{k \rightarrow \infty} I'(u_k)\cdot (\varphi (\cdot - y_k ^{(n)})) = 0,$ $n \in \mathbb{N}_0,$ and $J_n'(w^{(n)}) \cdot \varphi= \lim _{k \rightarrow \infty }I'(u_k) \cdot d_{k,n} \varphi =0,$ $n \in \mathbb{N}_+,$ where $d_{k,n}\varphi$ is given by \eqref{notacao} and $\varphi \in C^\infty _0 (\mathbb{R}^N).$ In particular, by \ref{f_tang} and \ref{g_AR}, we have \eqref{jotamais},
\begin{equation*}
    I(w^{(1)}) =  \int _{\mathbb{R}^N }\mathcal{F}(x,w^{(1)}) +\mathcal{K}(x,w^{(1)})\dx\geq 0 \quad\text{and}\quad  I_P(w^{(n)}) =  \int _{\mathbb{R}^N }\mathcal{F}_P(x,w^{(n)})+  \mathcal{K}_P(x,w^{(n)})\dx\geq 0, 
\end{equation*}
for $n \in \mathbb{N}_0,$ where $\mathcal{F}(x,s) := (1/2)f(x,s)s - F(x,s)$ and $\mathcal{K}(x,s) := (1/2)b(x)g(s)s - b(x)G(s).$

\noindent\textit{Step 2)}: $w^{(n)} = 0,$ for any $n \in \mathbb{N}_0 \setminus \{ 1 \}.$ Indeed, assume the existence of $w^{(n_0)} \neq 0,$ for some $n_0 \in \mathbb{N}_0\setminus \{ 1 \}.$ As proven in Step 3) of Section \ref{s_teo_periodic} (see \eqref{jotamais}),
\begin{align}
    c(I) &= \lim _{k \rightarrow \infty }\left[ \frac{1}{2} \| u_k \|_{V} - \int _{\mathbb{R}^N} F(x,u_k) \dx - \int _{\mathbb{R}^N}  b(x)G(u_k)\dx \right]\nonumber\\
    &\geq I(w^{(1)}) + \sum _{n \in \mathbb{N}_0 \setminus \{ 1 \}} I_P(w^{(n)}) + \sum _{n \in \mathbb{N}_+} J_n(w^{(n)}).\label{c_ddois}
\end{align}
Thus, we can use Remark \ref{r_compar} to conclude that $c(I) \geq I_P(w^{(n_0)}) \geq c(I_P).$ This is a contradiction with \ref{ce}.

\noindent\textit{Step 3)}: $w^{(1)} \neq 0.$ If $w^{(1)} =0,$ then $w^{(n)} = 0,$ for all $n \in \mathbb{N}_0.$ Hence, we can use the same argument as above in Step 2) of Section \ref{s_teo_periodic} (because $c(I)>0$), that Lemma \ref{l_converge} implies $u_k \rightarrow 0$ in $L^p (\mathbb{R}^N),$ for any $p \in (2,2^\ast)$ and
\begin{equation*}
    c(I)\geq \left( (\mu _\ast -2) /(2 \mu _\ast ) \right) \mathbb{S}^{N/2}_{\bar{G} _b},
\end{equation*}
which is a contradiction with Lemma \ref{l_minimax}.

\noindent\textit{Step 4)}: $w^{(n)} = 0,$ for any $n \in \mathbb{N}_+.$ Suppose, contrary to our claim, that there is $w^{(m_0)} \neq  0,$ for some $m_0 \in \mathbb{N}_+.$ Because $J_{m_0}'(w^{(m_0)})=0,$ a Pohozaev type identity holds, 
\begin{equation*}
\frac{N-2}{2N} \int _{\mathbb{R}^N} | \nabla w^{(m_0)}|^2 \dx = \int _{\mathbb{R}^N} b^{(m_0)}G(w^{(m_0)})\dx.
\end{equation*}
In particular, the path $\zeta (t) = w^{(m_0)} (\cdot / t),$ $\zeta (0) := 0,$ belongs to $\Gamma _{J_{m_0}}$ and $c(J_{m_0}) \leq \sup _{t\geq 0}J_{m_0}(\zeta (t)) = J_{m_0}(w^{(m_0)}).$ Therefore, by Lemma \ref{l_jota} and \eqref{c_ddois} we get
\begin{equation*}
    c(J_{m_0}) = c(I) = I(w^{(1)})  + \sum _{n \in \mathbb{N}_+} J_n(w^{(n)}).
\end{equation*}
On the other hand, by \ref{f_tang}, \ref{g_porbaixo} and \ref{g_AR},
\begin{equation}\label{w1dif}
    I(w^{(1)}) \geq   \int _{\mathbb{R}^N } \mathcal{K}(x,w^{(1)})\dx \geq \frac{1}{2}(\mu _\ast -2)\lambda _\ast b_0 \int _{\mathbb{R}^N}  |w^{(1)}|^{2^\ast}\dx>0,
\end{equation}
which leads to the contradiction: $c(J_{m_0}) > J_{m_0}(w^{(m_0)}).$

\noindent\textit{Step 5)}: $(u_k)$ has a convergent subsequence in $H^1_V(\mathbb{R}^N).$ Indeed, the convergence in \eqref{tinta4} ensures $u_k \rightarrow w^{(1)}$ in $L^p (\mathbb{R}^N),$ for $p \in (2,2^\ast],$ and by Lemmas \ref{l_converge} and \ref{l_Gconverge}, we have
\begin{align*}
	\| u_k \|_V^2 + o_k(1)  & = \int_{\mathbb{R}^N} f(x,u_k ) u_k \dx + \int _{\mathbb{R}^N}b(x) g(u_k ) u_k \dx   \\ 
	& =\int_{\mathbb{R}^N} f(x,w^{(1)} ) w^{(1)} \dx + \int _{\mathbb{R}^N}b(x) g(w^{(1)} ) w^{(1)} \dx + o_k(1) = \|w^{(1)}\|_V^2+ o_k(1)
\end{align*}
We conclude that $\lim _{k \rightarrow \infty}\|u_k \|^2_V = \|w^{(1)}\|^2_V,$ and this completes the proof.
\end{proof}
\section{Existence of ground states: Proof of Theorem \ref{teo_gs}}\label{s_gsp}
We first show that \ref{ce} implies a ``ground state gap" between the original and the limit problems.
\begin{lemma}\label{lgs}
	$\mathcal{G}_S(I) < \mathcal{G}_S(I_P),$ provided that \ref{ce} holds.
\end{lemma}
\begin{proof}
	Theorems \ref{teo_periodic} and \ref{teo_compact} imply the existence of $u_P \in \mathrm{Crit}(I_P)$ and $u_0 \in \mathrm{Crit}(I)$ such that $I_P(u_P) = c(I_P)=\mathcal{G}_S(I_P)$ and $I(u_0) = c(I).$ Hence
	\begin{equation*}
		\mathcal{G}_S(I) \leq c(I) < c(I_P) =\mathcal{G}_S(I_P).\qedhere
	\end{equation*}
\end{proof}
Next, we state the analogous version of Lemma \ref{l_salvou} for the functional $I$. The proof is identical.
\begin{lemma}\label{l_salvoufim}
	There is $C>0$ such that $\|u\|_{V} \geq C,$ for all $u \in \mathrm{Crit}(I).$
\end{lemma}
To address the case where the energy levels may coincide ($I \leq I_P$), we rely on the following lemma, which guarantees the existence of a critical point when the minimax level is attained along a chosen path.
\begin{lemmaletter}\cite[Theorem~2.3]{MR2532816}\label{l_localmp}
	Let $E$ be a real Banach space. Suppose that $I \in C^1(E)$ satisfies
	\begin{enumerate}[label=\roman*):]
		\item $I(0)=0;$
		\item There exist $r,\ b>0$ such that $I(u) \geq b,$ whenever $\|u\| = r; $
		\item There is $e\in E$ with $\|e\| > r$ and $I(e)<0;$ 
	\end{enumerate}
	Define 
	\begin{equation*}
		\hat{\Gamma}_I = \left\lbrace \xi \in C([0,1], E) : \xi (0) = 0,\ \|\xi(1)\| > r,\ I(\xi (1))<0 \right\rbrace\quad\text{and}\quad \hat{c}(I) = \inf_{\xi \in \hat{\Gamma}_I} \sup_{t\in [0,1]} I(\xi (t)).
	\end{equation*}
	If there exists $\xi _0 \in \hat{\Gamma}_I$ such that $\hat{c}(I) = \max _{t \in [0,1]} I(\xi _0 (t) ),$ then $I$ possesses a nontrivial critical point $u\in \xi _0 ([0,1]) $ such that $I(u) = \hat{c}(I).$
\end{lemmaletter}
	\begin{proof}[Proof of Theorem \ref{teo_gs} completed]
		Let $(u_k) \subset H^1_V(\mathbb{R}^N)$ such that $I(u_k) \rightarrow \mathcal{G}_S(I)$ with $I'(u_k) = 0.$ Similarly as above (Step 4) in Section \ref{s_teo_periodic}) we have $\mathcal{G}_S(I) \geq 0$ and following the proof of Proposition \ref{p_psbounded}, the sequence $(u_k)$ is bounded. Consider the profiles $(w^{(n)})_{n \in \mathbb{N}_\ast},$ $(y_k^{(n)})_k$ and $(j_k^{(n)})_k,$ $n \in \mathbb{N}_\ast,$ described in Theorem \ref{teo_tinta} and Remark \ref{rem_tinta}. We proceed using the same ideas above dividing the proof into some steps. We first consider the case where \ref{ce} holds.
		
		\noindent\textit{Step 1): } $w^{(1)}$ is a critical point of $I,$ $w^{(n)}$ is a critical point of $I_P,$ for $n \in \mathbb{N}_0,$ and $w^{(n)}$ is a critical point of $J_n,$ for $n \in \mathbb{N}_+.$ In particular, $I(w^{(1)} )\geq0,$ $I_P(w^{(n)} ) \geq 0,$ for $n \in \mathbb{N}_0,$ and $J_n (  w^{(n)}   ) \geq 0,$ for $n \in \mathbb{N}_+.$	
			
		\noindent \textit{Step 2)}: $w^{(n)} = 0,$ for any $n \in \mathbb{N}_0 \setminus \{ 1 \}.$ Suppose the existence of $n_0 \in \mathbb{N}_0$ with $w^{(n_0)} \neq 0.$ Similarly as above,
		\begin{align}
			\mathcal{G}_S(I) &= \lim _{k \rightarrow \infty }\left[ \frac{1}{2} \| u_k \|_{V} - \int _{\mathbb{R}^N} F(x,u_k) \dx - \int _{\mathbb{R}^N} b(x) G(u_k)\dx \right] \nonumber\\
			&\geq I(w^{(1)}) + \sum _{n \in \mathbb{N}_0 \setminus \{ 1 \}} I_P(w^{(n)}) + \sum _{n \in \mathbb{N}_+} J_n(w^{(n)}).\label{gsfim}
		\end{align}
		Thus $\mathcal{G}_S(I) \geq I_P(w^{(n_0)}) \geq \mathcal{G}_S(I_P).$ This is a contradiction with Lemma \ref{lgs}.
		
		\noindent\textit{Step 3)}: $w^{(1)} \neq 0.$ If $w^{(1)} =0,$ then Lemma \ref{l_converge} implies $u_k \rightarrow 0$ in $L^p (\mathbb{R}^N),$ for any $p \in (2,2^\ast).$ Therefore, applying Theorem \ref{teo_compact} and adapting the argument from Step 2) in Section \ref{s_teo_periodic} with the help of Lemma \ref{l_salvoufim},
		\begin{equation*}
			c(I)\geq \mathcal{G}_S(I)\geq  \left( (\mu _\ast -2) /(2 \mu _\ast ) \right)\mathbb{S}_{\bar{G}_{b}}^{N/2},
		\end{equation*}
		which is a contradiction with Lemma \ref{l_minimax}.
		
		\noindent\textit{Step 4)}: $w^{(n)} = 0,$ for any $n \in \mathbb{N}_+.$ We argue as in Section \ref{s_comp}. If there is $m_0 \in \mathbb{N}_+$ such that $w^{(m_0)}\neq 0,$ then $c(J_{m_0}) \leq J_{m_0}(w^{(m_0)}).$ On the other hand, since $\mathcal{G}_S(I) \leq c(I),$ by Lemma \ref{l_jota} and \eqref{gsfim} we have 
		\begin{equation*}
			c(J_{m_0}) = \mathcal{G}_S(I) = I(w^{(1)})  + \sum _{n \in \mathbb{N}_+} J_n(w^{(n)}).
		\end{equation*}
		Furthermore, because $w^{(1)}\neq 0$ we have $I(w^{(1)}) > 0 $ (see \eqref{w1dif}) and we obtain the contradiction that $c(J_{m_0}) > J_{m_0}(w^{(m_0)}).$
		
		\noindent\textit{Step 5)}: $(u_k)$ has a convergent subsequence in $H^1_V(\mathbb{R}^N).$ Since $u_k \rightarrow w^{(1)}$ in $L^p (\mathbb{R}^N),$ for $p \in (2,2^\ast],$ we can follow the same lines in \textit{Step 5)} of Section \ref{s_comp} to obtain $\lim _{k \rightarrow \infty}\| u_k \|^2_V = \| w^{ (1) }  \|^2_V,$ up to a subsequence. Hence $I (w^{ (1) } ) = \mathcal{G}_S(I).$
		
		\noindent\textit{Step 6)}: Now we suppose $I \leq I_P$ instead of \ref{ce}. We call attention to the fact that Lemma \ref{lgs} is only used in the conclusion of Step 2) to ensure that $w^{(n)} = 0,$ for any $n \in \mathbb{N}_0\setminus \{ 1\}.$ At this point we proceed differently. Let us assume the existence of $w^{(n_0)} \neq 0,$ for some $n_0 \in \mathbb{N}_0\setminus \{1 \}.$ Suppose that the first part of \ref{h_minimax} holds. Condition $I \leq I_P$ implies that the path $\xi(t) = t w^{(n_0)},$ $t \geq 0,$ belongs to $\Gamma _I.$ Thus, by Proposition \ref{p_ccp} (Appendix \ref{s_app_minimax}), \eqref{gsfim} and Remark \ref{r_compar} we have
		\begin{equation*}
			\mathcal{G}_S(I) \leq c(I) \leq \max_{t \geq 0} I(\xi(t)) \leq \max_{t \geq 0 } I_P(\xi(t)) = c(I_P)\leq \mathcal{G}_S(I).
		\end{equation*}
		In particular, $\mathcal{G}_S(I) = c(I)=\max_{t \geq 0} I(\xi (t)).$ Applying Lemma \ref{l_localmp}, the existence of $\hat{w} \in \mathrm{Crit}(I)$ such that $I(\hat{w})=  c(I) =\mathcal{G}_S(I) $ is guaranteed. By replacing $\xi(t)$ with the path $\zeta (t) = w^{(n_0)} (\cdot /t)$ and assuming the second part of \ref{h_minimax}, the same conclusion remains true. On the other hand, if $w^{(n)} = 0,$ for any $n \in \mathbb{N}_0 \setminus \{1\},$ we can follow the same lines of Steps 3)--5) to conclude that $I (w^{ (1) } ) = \mathcal{G}_S(I).$
	\end{proof}

\section{Existence of ground states for the fully nonautonomous case: Proof of Theorems \ref{teo_periodic2}--\ref{teo_gs2}}\label{s_full}

In this section, we prove the existence results for the fully nonautonomous problem \eqref{P}. The arguments closely follow those developed for the semi-autonomous case of Sections \ref{s_teo_periodic}, \ref{s_comp}, and \ref{s_gsp}. We verify that the core idea of each proof remains valid. Instead of repeating each step in detail, we focus on highlighting the essential modifications.

Throughout this section, we always suppose that conditions \ref{k_ss}, \ref{k_tec}, \ref{f_zero}, \ref{k_AR} hold, assuming that either \ref{h_um} or \ref{h_dois} holds. The energy functional $I$ is the one associated with \eqref{P}, defined as
\begin{equation*}
	I(u) = \frac{1}{2}\int _{\mathbb{R}^N} |\nabla u |^2 + V(x) u^2 \dx - \int _{\mathbb{R}^N} F(x,u) + K(x,u) \dx,\quad u \in H^1(\mathbb{R}^N),
\end{equation*}
where $K(x,s) = \int_0^s k(x,t) \dt$. The periodic functional $I_P$ is defined analogously, with $V_P$, $F_P$, and $K_P$. We observe that the preliminary results and the profile decomposition analysis from Section \ref{s_profile} adapt to our current setting with minor modifications.
\subsection{Boundedness of Palais-Smale sequences} The proof of Proposition \ref{p_psbounded}, which ensures that Palais-Smale (Cerami) sequences are bounded, extends directly to the nonautonomous case. The lower growth \ref{k_AR} is sufficient to replicate the contradiction argument, replacing $b(x)G(u)$ with $K(x,u).$ In fact, in that proof we replace \eqref{gen_um} by the estimate
\begin{equation*}
	\int _{\mathbb{R}^N} \frac{K(x,u_k)}{\| u_k \|_V^2} \dx \geq \hat{\lambda} _\ast \int _{U} |d^{-1}_{k,n} u_k|^{2^\ast -2} |d^{-1}_{k,n} v_k|^2 \dx,
\end{equation*}
and the right-hand side goes to $+\infty,$ as $k \rightarrow \infty.$
\subsection{Asymptotic behavior of the critical term} The most crucial result to adapt is Proposition \ref{p_convast}, which describes the behavior of the critical energy term under the profile decomposition. The original proof for $b(x)G(s)$ relied on Lemma \ref{l_geprop}, where the key inequality is derived from the self-similarity condition \ref{g_selfsimilar}. For the nonautonomous term $K(x,s)$, hypothesis \ref{k_tec} provides precisely this required inequality. Furthermore, condition \ref{k_ss} shows that $k(x,s)$ has a self-similar behavior at infinity.
\begin{proposition}\label{p_convastgen}Let $(u_k)\subset H^1(\mathbb{R}^N)  $ be a bounded sequence and $(w ^{(n)}) _{n \in \mathbb{N}_{\ast} }$ be the profiles given by Theorem \ref{teo_tinta}. Denote, for $n \in \mathbb{N}_+,$ $g_n = g _{z^{(n)}}$ if $(y_k^{(n)})_k$ is bounded, with $z^{(n)} = \lim_{k \rightarrow \infty } y_k^{(n)}$, and $g_n = g _{\infty },$ if $|y_k^{(n)}| \to \infty,$ as $k \rightarrow \infty$ (see \ref{k_ss}). For $G_n(s) = \int_0^s g_n(t)\dt,$ we have
	\begin{equation*}
		\lim _{k \rightarrow \infty} \int _{\mathbb{R}^N} K(x, u_k) \dx = \int_{\mathbb{R}^N}K(x,w^{(1)}) \dx + \sum _{n \in \mathbb{N}_0\setminus \{ 1\} } \int _{\mathbb{R} ^N} K _P (x,w^{(n)}) \dx + \sum _{n \in \mathbb{N} _{+ }} \int _{\mathbb{R}^N} G_n(w^{(n)}) \dx,
	\end{equation*}
\end{proposition}
\begin{proof}
Let us denote $ \hat{ \Phi } _\ast (u) = \int _{\mathbb{R}^N} K(x,u) \dx,$ $\hat{ \Phi } _{\ast,P} (u) = \int _{\mathbb{R}^N} K_P(x,u) \dx $ and $\hat{\Phi}_{\ast,n }(u) = \int _{\mathbb{R}^N} G_n (u) \dx ,$ $u \in D^{1,2}(\mathbb{R}^N).$ Following the proof of Proposition \ref{p_convast}, we have that \eqref{gen_abacaxi} still holds with $\hat{\Phi}_\ast$ instead of $\Phi_\ast.$ Moreover, using \ref{k_ss} the same argument applies and
\begin{align*}
	&\lim _{k \rightarrow \infty} \left[\sum _{n \in \mathbb{N}_0} \hat{\Phi }_\ast (w^{(n)}  ) -  \hat{\Phi }_\ast (w^{(1)})  - \sum _{n \in \mathbb{N}_0 \setminus \{  1 \} } \hat{\Phi }_{\ast,P}(w^{(n)}) \right] =0,\\ &\lim _{k \rightarrow \infty} \left[ \sum _{n \in \mathbb{N}_+} \hat{\Phi }_\ast ( w^{(n)} )   - \sum _{n \in \mathbb{N}_+} \hat{\Phi }_{\ast,n} (w^{(n)}) \right]=0,
\end{align*}
substitute \eqref{gen_banana}. We use \ref{k_tec} to prove the analogous convergence of \eqref{giant1}, where $\hat{\Phi}_\ast $ replaces $\Phi _\ast$.
\end{proof}
\subsection{Behavior of the energy levels at infinity} It is evident that in all of our argument to prove existence of solutions for Eq. \eqref{P} it is crucial to control the energy levels of the profile decomposition by the associated minimax level in the sense of \eqref{c_ddois}. Here we consider the critical energy functional $J_n$ in its general form given by \eqref{jotaast} together with the notation described in Proposition \ref{p_convastgen}. Therefore we propose to present a more general version of this fact that can be applied in all different settings in the proofs of Theorems \ref{teo_periodic2}--\ref{teo_gs2}.
\begin{proposition}\label{p_convfinal}
	Let $(u_k) \subset H^1(\mathbb{R}^N)$ be a bounded sequence such that $I(u_k) \rightarrow d $ and $\|I'(u_k) \|_\ast \rightarrow 0.$ Let $(w ^{(n)}) _{n \in \mathbb{N}_{\ast} }$ be the profiles given by Theorem \ref{teo_tinta}. Then,
	\begin{equation*}
		d \geq I(w^{(1)}) + \sum _{n \in \mathbb{N}_0 \setminus \{ 1 \}} I_P(w^{(n)}) + \sum _{n \in \mathbb{N}_+} J_{\nu _n}(w^{(n)}),
	\end{equation*}
	where $J_{\nu _n}(u) = (1/2) \int _{\mathbb{R}^N} |\nabla u |^2 \dx-\int _{\mathbb{R}^N} G_n (u) \dx.$
\end{proposition}
\begin{proof}
It follows by \eqref{tinta3} together with Propositions \ref{p_bef}, \ref{p_convv} and \ref{p_convastgen}.
\end{proof}
\subsection{Proof of the main results} Equipped with these observations, we can now prove our main theorems.
\begin{proof}[Proof of Theorem \ref{teo_periodic2}]
	The proof follows the same five-step structure as the proof of Theorem \ref{teo_periodic}. Here we consider $J_{\nu _n}$ instead of $J_n.$
	\begin{itemize}
		\item[]\textit{Step 1)}:  The proof that each profile $w^{(n)}$, with $n \in \mathbb{N}_0$, is a critical point of $I_P$ is identical, using the Lebesgue theorem, which applies due to \ref{f_geral} and \ref{k_geral}.
		\item[] \textit{Step 2)}: The argument that there exists at least one non-zero profile $w^{(n_0)}$ is the same in Step 2) of the proof of Theorem \ref{teo_periodic} and follows by contradiction. In fact, the nonautonomous version of the minimax level estimate (Lemma \ref{l_minimax}) still holds, replacing $\bar{G},$ $\lambda_\ast$ and $\mu_\ast$ by $\hat{G},$ $\hat{\lambda}_\ast$ and $\hat{\mu}_\ast,$ respectively.
		\item[] \textit{Steps 3), 4), and 5)}: The proof that the ground state level $\mathcal{G}_S(I_P)$ is achieved, and its identification with the minimax level under hypothesis \ref{h_minimax} follow a similar line of reasoning, replacing \eqref{c_ddoisP} with Proposition \ref{p_convfinal}, $b_P(x)g(s)$ with $k_P(x,s)$ and $b_P(x)G(s)$ with $K_P(x,s)$. \qedhere
	\end{itemize}
\end{proof}

\begin{proof}[Proof of Theorem \ref{teo_compact2}]
\noindent\textit{Part I: Case \ref{h_um}.} The proof mirrors that of Theorem \ref{teo_compact}.
		\begin{itemize}
		\item[]\textit{Step 1)}: The profiles $w^{(1)}$, $w^{(n)}$ ($n \in \mathbb{N}_0 \setminus \{1\}$), and $w^{(m)}$ ($m \in \mathbb{N}_+$) are critical points of $I$, $I_P$, and $J_{\nu _m}$, respectively. In fact, by \ref{k_ss},
		\begin{align*}
		o_k(1) &= I'(u_k) \cdot d_{k,m} \varphi	\\
		&= (d^{-1}_{k,m} u_k , \varphi )_{D^{1,2} (\mathbb{R} ^N)} + \int _{\mathbb{R}^N } V(x) u_k \,  d_{k,m} \varphi \dx \\
		&\qquad- \int _{\mathbb{R}^N } f(x,u_k) d_{k,m} \varphi \dx  -  \int _{\mathbb{R}^N} \gamma ^{-N j_k^{(m)}} k(\gamma ^{-j_k^{(m)}} \cdot   + y_k^{(m)} , \gamma ^{\frac{N-2}{2} j_k^{(m)} }   d^{-1}_{k,m} u_k) \varphi \dx\\
		&= \int _{\mathbb{R}^N}  \nabla w^{(m)}\cdot \nabla \varphi  \dx - \int _{\mathbb{R}^N} g_m (w^{( m )}) \varphi \dx + o_k(1),\quad \forall \varphi \in C^\infty _0(\mathbb{R}^N).
		\end{align*}
		The proof is the same.
		\item[]\textit{Step 2)}: The non-existence of profiles $w^{(n)} \neq 0$ for $n \in \mathbb{N}_0 \setminus \{1\}$ is proven by contradiction. The logic is identical: if such a profile existed, we would have $c(I) \geq I_P(w^{(n)}) \geq c(I_P)$, contradicting \ref{ce}.
		\item[]\textit{Step 3)}: The proof that $w^{(1)} \neq 0$ is the same, using the contradiction with the minimax level estimate .
		\item[]\textit{Step 4)}: To show that $w^{(m)} = 0$ for all $m \in \mathbb{N}_+$, we use the inequality $c(I) \leq c(J_{\nu _{m_0}})$. If a profile $w^{(m_0)} \neq 0$ existed, we would have $c(J_{\nu _{m_0}}) \leq J_{\nu _{m_0} } (w^{(m_0)})$. The energy decomposition given by Proposition \ref{p_convfinal} would lead to $c(I) \geq I(w^{(1)}) + J_{\nu _{m_0}}(w^{(m_0)})$. Since $I(w^{(1)}) > 0$, we would obtain $c(I) > J_{\nu _{m_0}}(w^{(m_0)}) \geq c(J_{\nu _{m_0} } )$, which, combined with $c(I) \leq c(J _{\nu _{m_0} })$, leads to a contradiction.
		\item[]\textit{Step 5)}: With all profiles being zero, the strong convergence $u_k \rightarrow w^{(1)}$ in $H^1_V(\mathbb{R}^N)$ is guaranteed.  
	\end{itemize}
\noindent\textit{Part II: Case \ref{h_dois}.} The proof for this case primarily differs from the previous one in the adaptation needed for Step 3. Steps 1 and 2 can be followed as before, yielding the energy decomposition
\begin{equation}\label{finalum}
	c(I) \geq I(w^{(1)}) + \sum _{n \in \mathbb{N}_0 \setminus \{ 1 \}} I_P(w^{(n)}) + \sum _{n \in \mathbb{N}_+} J_{\nu _{n} }(w^{(n)}).
\end{equation}
Next we assume the existence of $m_0 \in \mathbb{N}_+$ with $w^{(m_0)} \neq 0.$ As shown previously, inequality \eqref{finalum} implies $c(I) \geq J_{\nu _{m_0} } (w^{(m_0)}) \geq c(J_{\nu _{m_0} } ),$ which contradicts \ref{h_dois}. Thus, we must have $w^{(n)} = 0$ for all $n \in \mathbb{N}_\ast \setminus \{ 1 \}.$ By \eqref{tinta4}, this leads to $u_k \rightarrow w^{(1)}$ in $L^p(\mathbb{R}^N)$ for every $p \in (2,2^\ast ].$ Since $c(I)>0,$ this convergence also ensures that $w^{(1)} \neq 0.$ Finally, repeating the argument from the previous case, the strong convergence $u_k \rightarrow w^{(1)}$ in $H^1_V(\mathbb{R}^N)$ follows.\end{proof}

\begin{proof}[Proof of Theorem \ref{teo_gs2}]\noindent\textit{Part I: \ref{h_um} holds, and either \ref{ce} or $I \leq I_P$ is satisfied.} At this point it becomes clear that the proof of Theorem \ref{teo_gs2} directly adapts the proof of Theorem \ref{teo_gs}, with similar modifications as above and simply by replacing $J_n$ with $J_{\nu _n}$.
	
	\begin{enumerate}[label= \roman*):]
		\item Let $(u_k)$ be a minimizing sequence for $\mathcal{G}_S(I)$ and assume \ref{ce}. Because $\mathcal{G}_S(I) \geq 0,$ the sequence $(u_k)$ is bounded. The proof that $(u_k)$ converges to a ground state $w^{(1)}$ follows the steps of the proof of Theorem \ref{teo_gs}. In fact, by Proposition \ref{p_convfinal} we have
		\begin{equation}\label{setefim}
			\mathcal{G}_S(I) \geq I(w^{(1)}) + \sum _{n \in \mathbb{N}_0 \setminus \{ 1 \}} I_P(w^{(n)}) + \sum _{n \in \mathbb{N}_+} J_{\nu _n}(w^{(n)}),
		\end{equation}	
		and Lemma \ref{l_salvoufim} still holds. From this point on, the main tool is Lemma \ref{lgs} (the proof is the same), which establishes $\mathcal{G}_S(I) < \mathcal{G}_S(I_P)$ and implies that $w^{(n)} = 0,$ for all $n \in \mathbb{N}_\ast  \setminus \{ 1 \}.$ Consequently, $u_k \rightarrow w^{(1)}$ in $H^1_V(\mathbb{R}^N).$
		
		\item If we assume $I \leq I_P$, the argument from Step 6) in the proof of Theorem \ref{teo_gs} is applied. The existence of a nontrivial profile $w^{(n_0)},$ $n_0 \in \mathbb{N}_0 \setminus \{1\}$ would imply
		\begin{equation*}
			\mathcal{G}_S(I) \leq c(I) \leq \max_{t \geq 0} I(\chi(t)) \leq \max_{t \geq 0} I_P(\chi(t)) = c(I_P)\leq \mathcal{G}_S(I),
		\end{equation*}
		for a suitable path $\chi(t)$ depending on $w^{(n_0)},$ which is based on \ref{h_minimax}. Particularly, $\mathcal{G}_S(I) = c(I) = \max_{t \geq 0} I( \chi(t))$. Lemma \ref{l_localmp} then guarantees the existence of a critical point at the level $\mathcal{G}_S(I)$, which is, by definition, a ground state. If $w^{(n)} = 0, $ $n \in \mathbb{N}_0\setminus\{1 \}$, convergence is guaranteed as in case i).
	\end{enumerate}
\noindent\textit{Part II: Case \ref{h_dois} and \ref{ce}.} Let $(u_k)$ be a minimizing sequence for $\mathcal{G}_S(I),$ which is bounded. As established in the proof of Theorem \ref{teo_gs}, Steps 1) and 2) remain valid under \ref{ce} (via Lemma \ref{lgs} ), ensuring $w^{(n)}=0$ for $n \in \mathbb{N}_0 \setminus \{1\}$. We also have \eqref{setefim} which allows the usual argument: if there is $m_0 \in \mathbb{N}_+$ with $w^{(m_0)} \neq 0,$ then $c(I)\geq \mathcal{G}_S(I) \geq J_{\nu _{m_0} }(w^{(m_0)}) \geq c(J_{ \nu _{m_0} }),$ contradicting \ref{h_dois}. Consequently, we must have $w^{(n)} = 0$ for all $n \in \mathbb{N}_+$. Combining these results, we find that $w^{(n)}=0$ for all $n \in \mathbb{N}_\ast \setminus \{1\}$. By \eqref{tinta4}, this implies $u_k \rightarrow w^{(1)}$ in $L^p(\mathbb{R}^N)$ for $p \in (2, 2^\ast]$. The standard arguments then show that $u_k \rightarrow w^{(1)}$ in $H^1_V(\mathbb{R}^N).$ Moreover, the suitable version of Lemma \ref{l_salvoufim} yields $w^{(1)} \neq 0.$ By the continuity of $I$, we have $I(w^{(1)}) = \mathcal{G}_S(I)$.
\end{proof}

\appendix
\section{Additional remarks}\label{s_app}
\subsection{Construction of an oscillatory subcritical nonlinearity} Here we complete the construction of the nontrivial example $f $ given in \eqref{exampleB}. Let us take $q_0,$ $\alpha,$ $\beta >0,$ $p_0 \in (2,2^\ast).$ We define a $C^\infty  (\mathbb{R})$ function $\eta$ such that $ \beta + \alpha \leq \eta (s) \leq q_0$ and
\begin{equation*}
\eta (s)=
\left\{ 
\begin{aligned}
&q_0,&&\quad \text{if }|s| \leq 1/4,\\
&\beta + \alpha, &&\quad \text{if } |s| \geq 1/2,
\end{aligned}
\right.
\end{equation*}
where the parameters are chosen to satisfy
\begin{equation*}
    p_0\leq \beta - \alpha \leq \beta + \alpha \leq q_0 <2^\ast,\ 4\left[q_0 - (\beta + \alpha ) \right] <  e \left[  (\beta - \alpha ) - p_0 \right] \text{ and }\sup _{ \frac{1}{4} \leq |s| \leq \frac{1}{2}} |\eta' (s)| \leq e \left[  (\beta - \alpha ) - p_0\right].
\end{equation*}
Next, we consider $L(s) = C_0 L_0(s)$ with $C_0 = (\pi /2)(1/(\ln(\ln (2) +1)   ))$ and $L_0(s) = \ln ( \ln (|s|+1)  +1).$ We define $\varrho \in C^1(\mathbb{R})\cap L^\infty (\mathbb{R})$ by 
\begin{equation*}
\varrho (s)=
\left\{ 
\begin{aligned}
&\eta (s),&&\quad \text{if }|s| \leq 1,\\
&\beta + \alpha \sin (L(s)), &&\quad \text{if } |s| \geq 1.
\end{aligned}
\right.
\end{equation*}
The continuous function $f$ given by \eqref{exampleB} satisfies \ref{f_geral}, \ref{f_tang}, and \ref{f_porbaixo}. Indeed,
\begin{equation*}
    F (s) := \int ^s _0 f(t)\dt =\lambda (1+|s|^{\varrho(s) - p_0}) |s|^{p_0}.
\end{equation*}
By construction, we have $F(s) \ge \lambda |s|^{p_0}$, satisfying \ref{f_porbaixo}. If we also choose $2<\beta - (C_0+1)\alpha <p_0$, the function $F$ satisfies the Ambrosetti-Rabinowitz condition \eqref{AR} for $2 < \mu \leq \beta -(C_0 +1) \alpha < p_0$, which in turn implies \ref{f_tang}. On the other hand, because $s \mapsto \varrho ' (s) s \ln (|s|)+ \varrho (s) $ is uniformly bounded, we have {\color{blue}$(f'_1)$} and {\color{blue}$(f''_1)$}, thus satisfying hypothesis \ref{f_geral}.
%4\left[q_0 - (\beta + \alpha ) \right] <  e \left[  (\beta - \alpha ) - p_0 \right], se trocar por \leq, garante a generalidade, mas dificuldade na leitura
\subsection{Example of a self-similar oscillatory nonlinearity with critical growth} In this section, we detail the choice of parameters $A, B, E > 0$ for which the function $g$ given in \eqref{example_ss} satisfies hypotheses \ref{g_selfsimilar}--\ref{g_AR}, accounting for the coefficient $b \in C(\mathbb{R}^N) \cap L^\infty (\mathbb{R}^N)$ from \ref{g_AR}. The self-similarity \ref{g_selfsimilar} is readily satisfied by choosing $\gamma = \exp \left\{ 4\pi /(\omega (N-2) )\right\}.$ To verify the remaining conditions, we analyze the associated Sobolev-type constants. Let $G_E(s) = g(s)s = (\theta(s) + E) |s|^{2^\ast }$, $G_0(s) = \theta(s) |s|^{2^\ast }$, and $G_\ast (s) = |s|^{2^\ast}.$ We denote their respective infima by $\mathbb{S}_E,$ $\mathbb{S}_0,$ and $\mathbb{S}$ and their respective suprema by $\mathbb{K}_E,$ $\mathbb{K}_0,$ and $\mathbb{K}$ (see \eqref{sup_tinta}). Since $\sup \theta = A+B > 0$, the results of \cite[Proposition 2.2]{MR2465979} and \cite[Theorem 5.2]{tintabook} ensure these constants are well-defined, positive, and attained. Our strategy is to relate $\mathbb{S}_E$ to $\mathbb{S}_0$ and $\mathbb{S}$ and study its behavior as $E \to 0^+$. We begin with some key lemmas for this approach.
\begin{lemma}\label{l_lemaum}
	$(\mathbb{S}/\mathbb{S}_0) ^{2^\ast /2} \leq  \sup \theta .$
\end{lemma}
\begin{proof}
We are going to prove that $\mathbb{S} \leq (\sup \theta ) ^{2/2^\ast} \mathbb{S}_0.$ We start from the definition of $\mathbb{S}$  and the fact that $G_0(u) \leq (\sup \theta) |u|^{2^\ast}$, which gives $\mathbb{S} \left((1/\sup\theta )   \int _{\mathbb{R}^N} G_0 (u) \dx \right)^{2/2^\ast} \leq \mathbb{S} \left( \int _{\mathbb{R}^N}  |u|^{2^\ast} \dx \right)^{2/2^\ast} \leq \| \nabla u \|_2 ^2,$ where $u \in D^{1,2}(\mathbb{R}^N)$ and $\int _{\mathbb{R}^N } G_0(u) \dx >0.$ Rearranging the terms yields $\mathbb{S} (\sup \theta )^{-2/2^\ast } \leq \left(  \int _{\mathbb{R}^N} G_0 (u) \dx \right) ^{-2/2^\ast } \| \nabla u \|_2 ^2 .$ In particular, taking $u_0$ as the minimizer of $\mathbb{S}_0,$ the inequality yields $\mathbb{S} (\sup \theta )^{-2/2^\ast } \leq \mathbb{S}_0.$
\end{proof}
\begin{lemma}\label{l_lemadois}
	$\mathbb{S} = \mathbb{K} ^{-2 /2^\ast} , $ $\mathbb{S}_0 = \mathbb{K}_0 ^{-2 /2^\ast}$ and $\mathbb{S}_E = \mathbb{K}_E ^{-2 /2^\ast}.$
\end{lemma}
\begin{proof}
It suffices to prove that $\mathbb{S}_0 = \mathbb{K}_0 ^{-2 /2^\ast},$ as the remaining identities follow the same argument. In fact, one has $\mathbb{S}_0 \left( \int _{\mathbb{R}^N } G_0(u) \dx \right)^{2/2^\ast } \leq \| \nabla u \|_2^2,$ for any $u \in D^{1,2}(\mathbb{R}^N),$ with $\int _{\mathbb{R}^N } G_0(u) \dx >0.$ Taking $\bar{u}_0$ as the maximizer of $\mathbb{K}_0,$ we conclude that $\mathbb{S}_0 \mathbb{K}_0 ^{2/2^\ast } \leq 1.$ On the other hand, $\left( \int _{\mathbb{R}^N} G_0(u) \dx \right) ^{2/2^ \ast }\leq \mathbb{K}_0^{2/2^\ast }\|\nabla u \|_2 ^2 ,$ $u \in D^{1,2}(\mathbb{R}^N)$ and $\int _{\mathbb{R}^N } G_0(u) \dx >0.$ Considering $u_0$ as the minimizer of $\mathbb{S}_0,$ we have $1\leq \mathbb{K}_0^{2/2^\ast } \mathbb{S}_0. $
\end{proof}
Next, since $G_E(s) = G_0(s) + EG_{\ast}(s),$ we have the estimate $\mathbb{K}_E \leq \mathbb{K}_0 + E \mathbb{K}.$ By Lemma \ref{l_lemadois}, we can write this in the equivalent form $\mathbb{S}_E^{-2^\ast /2} \leq \mathbb{S}_0^{-2^\ast /2} + E \mathbb{S}^{-2^\ast /2}.$ Consequently, $\mathbb{S}/\mathbb{S}_E \leq \mathbb{S} (  \mathbb{S}_0^{-2^\ast /2} + E \mathbb{S}^{-2^\ast /2} ) ^{2/2^\ast } \rightarrow \mathbb{S}/\mathbb{S}_0,$ as $E \rightarrow 0^+.$ In particular, there exists $E_0>0$ such that,
\begin{equation}\label{apen_des1}
	\frac{\mathbb{S}}{\mathbb{S}_E}\leq \mathbb{S} \left(  \mathbb{S}_0^{-2^\ast /2} + E \mathbb{S}^{-2^\ast /2} \right) ^{2/2^\ast } \leq 2 \frac{\mathbb{S}}{\mathbb{S}_0} ,\quad \text{for }0<E<E_0.
\end{equation}
From this point on, we fix $0<E<E_0$ such that \eqref{apen_des1} holds. Next we consider the increasing function $\phi (s) = 2N/(N-2s),$ $s \in (0,N/2).$ We will show that for any given $\varepsilon \in (0,1),$ there are $A,$ $B>0$ such that, for a suitable $\lambda _\ast,$ we have $\kappa_\ast <\varepsilon$ and $2=\phi (0) < \phi (\varepsilon ) < \mu _\ast.$ If this holds, then $\phi (\kappa _\ast ) < \phi(\varepsilon) < \mu _\ast, $ which proves \ref{g_AR}. We begin by choosing $\lambda_\ast$ to satisfy \ref{g_porbaixo}. Note that $G_E(s) \geq (1/2^\ast ) (\inf \theta +E)|s|^{2^\ast} = (1/2^\ast )(A-B+E)|s|^{2^\ast}.$ Thus, considering $B/(A+E) < 1/2,$ we can set $\lambda _\ast = (1/2^\ast )(A-B+E),$ which also implies $g(s)s>0.$ Since $\sup \theta = A+B,$ now we estimate $\mathbb{S}/\mathbb{S}_0$ using Lemma \ref{l_lemaum} and our choice of $\lambda _\ast,$ 
\begin{equation}\label{estimativa_S0}
	\frac{(\mathbb{S}/\mathbb{S}_{0})^{ 2^\ast /2}}{ \lambda_\ast} \leq \left( \frac{1 + \frac{B}{A+E}}{1-\frac{B}{A+E}} \right) \frac{A+B}{A+B+E}.
\end{equation}
The right-hand side of \eqref{estimativa_S0} can be made arbitrarily small, provided that $A+B$ is sufficiently small (while keeping $A-B+E > 0$). By combining \eqref{apen_des1} and \eqref{estimativa_S0}, we can estimate $\kappa_\ast.$ Indeed, given $\varepsilon \in (0,1),$ if $A+B$ is small enough, then
\begin{equation*}
	\kappa_\ast = \left( \frac{(\mathbb{S}/\mathbb{S}_{E})^{N/(N-2)}}{2^\ast \lambda_\ast} \frac{\| b \|_\infty }{ b_0} \right)^{\frac{N-2}{2}} \leq \left( \frac{(2 \mathbb{S}/\mathbb{S}_{0})^{N/(N-2)}}{2^\ast \lambda_\ast} \frac{\| b \|_\infty }{ b_0} \right)^{\frac{N-2}{2}} < \varepsilon.
\end{equation*}
It remains to show that $\phi(\varepsilon) < \mu_\ast.$ We can choose $\mu_\ast$ as the lower bound of the ratio $g(s)s / G(s),$ i.e., taking $\mu _\ast  =  2^\ast (A-B+E) / (A+B+E),$ it is clear that $g(s)s/G(s) \geq \mu _\ast.$ Therefore, the condition $\phi (\varepsilon ) < \mu _\ast$ is equivalent to
\begin{equation}\label{fffinal}
	 \frac{1 + \frac{B}{A+E}}{1-\frac{B}{A+E}} < \frac{N-2\varepsilon}{N-2}.
\end{equation}
Because $(N-2\varepsilon)/(N-2) > 1,$ we can take $B/(A+E)$ sufficiently small for which \eqref{fffinal} holds. Summing up, for these $A$ and $B,$ we have $2N/(N-2\kappa_\ast)= \phi (\kappa_\ast) < \phi (\varepsilon) <\mu _\ast,$ implying the validity of \ref{g_selfsimilar}--\ref{g_AR} for $g(s) =  (\theta(s) + E) |s|^{2^\ast -2}s.$
\section{Relation between the minimax levels}\label{s_app_minimax} We dedicate this appendix to establishing the previously mentioned relation between the minimax level $c(I)$ and the minimax levels related to the limits originating from the lack of compactness. Here we assume the full set of hypotheses \ref{V_autovalor}, \ref{f_geral}--\ref{f_tang}, \ref{k_geral}, \ref{k_ss}, \ref{k_tec}, \ref{f_zero}, \ref{k_AR} and \ref{h_minimax} for the nonautonomous case. Let us consider the functional $J_\nu : D^{1,2}(\mathbb{R}^N) \rightarrow \mathbb{R}$ given by \eqref{jotaast}. Condition \ref{k_AR} ensures the autonomous limit problem has a well-posed variational structure and allows us to apply the following existence result from \cite{MR2465979}.
\begin{lemmaletter}\cite[Proposition 2.4]{MR2465979} There exists a weak solution $u_\nu \in D^{1,2}(\mathbb{R}^N)$ of $-\Delta u = g_\nu (u)$ in $\mathbb{R}^N,$ such that $J_\nu (u_\nu) = c(J_\nu ).$
\end{lemmaletter}
The following energy estimate is crucial for controlling the behavior of the profiles $w^{(n)},$ $n \in \mathbb{N}_{+}.$
\begin{proposition}\label{p_jota}
	$c(I) \leq c(J_\nu ).$
\end{proposition}
\begin{proof}By standard elliptic regularity theory \cite{brezis-kato,struwebook}, it is known that $u_\nu \in C(\mathbb{R}^N).$ Take any $\eta \in C^\infty(\mathbb{R},[0,1])$ such that $\eta (s) = 1,$ if $|s|\leq 1,$ and $\eta (s)=0,$ if $|s|>2.$ Define $\eta _n (x) = \eta (|x|^2/n^2).$ Clearly $\eta _n \in C^\infty _0(\mathbb{R}^N),$ $\eta_n \leq 1,$ $\| \nabla \eta _n\|_\infty \leq C/n $ and $|x||\nabla \eta _n(x)| \leq C,$ for some $C>0.$  Define $u_n = \eta _n u_\nu \in H^1(\mathbb{R}^N),$ which has compact support. Also define $\zeta ^{(n)}_k (t) = u_n ^{(k)} (\cdot /t),$ where $u_n^{(k)} = \gamma ^{\frac{N-2}{2} j_k} u_n (\gamma ^{j_k}( \cdot - y_k) ),$ $(j_k)\subset \mathbb{Z}$ is such that $j_k \rightarrow + \infty$ and $(y_k)$ is given in the definition of \ref{k_ss} ($y_k = \nu$ or $|y_k| \rightarrow \infty$). We prove that $\zeta ^{(n)}_k  \in \Gamma _I,$ for $k$ and $n$ sufficiently large. We start by noticing that a change of variable, together with the fact that $g_\nu$ is self-similar, leads to
	\begin{equation}\label{conv_g}
		\lim _{k \rightarrow \infty }\int _{\mathbb{R}^N} K(x,\zeta ^{(n)}_k (t))\dx = \lim _{k \rightarrow \infty } \int _{\mathbb{R}^N} \gamma ^{-Nj_k} K( \gamma^{-j_k} x +y_k ,\gamma ^{\frac{N-2}{2} j_k} u_n (\cdot / t)) \dx = t^N\int _{ \mathbb{R}^N }  G _\nu  (  u_n )\dx.
	\end{equation}
	Moreover, by \ref{f_geral} and Lebesgue convergence theorem, we have
	\begin{equation*}
		\lim _{k \rightarrow \infty } \int _{\mathbb{R}^N}F(x, \zeta ^{(n)}_k (t) ) \dx =t^N\lim _{k \rightarrow \infty } \int _{\mathbb{R}^N} \gamma ^{-N j_k} F( (\gamma ^{-j_k} x -y_k )t, \gamma ^{\frac{N-2}{2} j_k } u_n  ) \dx=0.
	\end{equation*}
	Similarly, since $V \in L^\infty (\mathbb{R}^N),$
	\begin{equation}\label{conv_ve}
		\lim _{k \rightarrow \infty }\int _{\mathbb{R}^N} V(x) (\zeta ^{(n)}_k (t))^2  \dx = t^N \lim _{k \rightarrow \infty } \int _{\mathbb{R}^N} \gamma ^{-2j_k} V((\gamma ^{-j_k} x -y_k )t ) u_n^2 \dx = 0.
	\end{equation}
	Nevertheless, because $u_\nu$ is a critical point of $J_\nu ,$ the following Pohozaev identity holds,
	\begin{equation}\label{poho_ge}
		\int _{\mathbb{R}^N } G_\nu (u _\nu ) \dx = \frac{N-2}{2N} \int_{\mathbb{R}^N } |\nabla u_\nu  |^2 \dx >0.
	\end{equation}
	Furthermore, \ref{k_geral} and the Lebesgue convergence theorem yield
	\begin{equation*}
		\lim _{n \rightarrow \infty } \int _{\mathbb{R}^N } G_\nu (u_n) \dx = \int _{\mathbb{R}^N } G_\nu (u_\nu)\dx.
	\end{equation*}
	Thus, there is $n_0\in \mathbb{N}$ such that for $n \geq n_0,$ the right-hand side of \eqref{conv_g} is positive. On the other hand, since $F(x,s) \geq 0,$ for a given $n\geq n_0,$ there is $k_n$ such that, for $k \geq k_n,$ we have
	\begin{multline}\label{eq_gostei}
		I(\zeta ^{(n)}_k (t)) \leq \frac{t^{N-2}}{2}  \int _{\mathbb{R}^N} |\nabla u_n |^2\dx +  \frac{t^N}{2}  \int _{\mathbb{R}^N} \gamma ^{-2j_k} V(  (\gamma ^{-j_k} x -y_k )t) u_n^2 \dx  \\ - t^{N} \int _{ \mathbb{R}^N }  \gamma ^{-Nj_k} K(  (\gamma ^{-j_k} x -y_k )t ,\gamma ^{\frac{N-2}{2} j_k} u_n  ) \dx \rightarrow -\infty,\quad \text{as }t \rightarrow \infty.
	\end{multline}
	That is, $\zeta ^{(n)}_k \in \Gamma _I,$ for $n \geq n_0$ and $k \geq k_n.$ In this case, we infer the existence of $t_k^{(n)} > 0 $ such that $\max _{t \geq 0} I(\zeta _k ^{(n)} (t)) = I(\zeta _k ^{(n)} (t_k^{(n)}))$ and replacing $t$ by $t_k^{(n)}$ in \eqref{eq_gostei}, we can see that $(t_k^{(n)})_{k \geq k_n}$ is bounded. Up to a subsequence, $t_k^{(n)} \rightarrow \tau ^{(n)} \geq 0$ and we have $\lim_{k \rightarrow \infty} I(\zeta ^{(n)}_k (t)) = J_\nu (u_n(\cdot / \tau ^{(n)})),$ because the convergences in \eqref{conv_g}--\eqref{conv_ve} are uniformly on compact sets of $\mathbb{R}.$ Summing up, 
	\begin{equation*}
		c(I) \leq \lim _{k \rightarrow \infty} \max_{t \geq 0} I(\zeta ^{(n)}_k (t) )  = J_\nu (u_n(\cdot / \tau ^{(n)})),\quad \text{for }n \geq n_0.
	\end{equation*}
	By a change of variables and \eqref{poho_ge} it is clear that $(\tau ^{(n)})_{n \geq n_0}$ is bounded and $\tau ^{(n)} \rightarrow \bar{t} \geq 0,$ up to a subsequence. Thus, by Lebesgue convergence theorem, \eqref{poho_ge} and Remark \ref{r_compar} we have
	\begin{equation*}
		c(I) \leq \lim _{n \rightarrow \infty }J _\nu (u_n(\cdot / \tau ^{(n)})) = J_\nu (u_\nu(\cdot / \bar{t} ) \leq \max _{t \geq 0 }J_\nu (u_\nu(\cdot / t  ) ) = J _\nu (u_\nu ) = c(J_\nu). \qedhere
	\end{equation*}
\end{proof}
We now establish the natural ordering between the minimax levels of the original problem \eqref{P} and its periodic counterpart \eqref{PP}. 
\begin{proposition}\label{p_ccp}
	$c(I) \leq c(I_P).$ 
\end{proposition}
\begin{proof}
	Let $u_P \in H^1 (\mathbb{R}^N)$ be a nontrivial solution of \eqref{PP} such that $I_P(u)=c(I_P).$ Define $u_n = \eta _n u_P,$ which has a compact support $U_n,$ where $\eta _n \in C^\infty _0 (\mathbb{R}^N)$ is given in the proof of Proposition \ref{p_jota}. Moreover, by elliptic regularity theory \cite{brezis-kato,struwebook}, $u_n \in C(\mathbb{R}^N).$ Next, we take $(y_k) \subset \mathbb{Z}^N$ with $|y_k|\rightarrow \infty.$\\
	\noindent \textit{Case i)}: $ s \mapsto \left(f_P(x,s)+k_P(x,s) \right)|s|^{-1}$ is strictly increasing, for all $x \in \mathbb{R}^N.$ Define $\xi^{(n)} _k(t) = t u_n (\cdot -y_k),$ $k,$ $n\in \mathbb{N}.$ Following the same argument used to prove Lemma \ref{l_mpgeometry}, we have
	\begin{equation}\label{lim_te}
		\frac{I(\xi^{(n)} _k(t)) }{t^2} = \frac{1}{2}\| \nabla u_n \|^2_2 + \frac{1}{2}\int _{\mathbb{R}^N} V(x+y_k) u_n^2\dx - \int _{\mathbb{R}^N} \frac{H(x+y_k , t u_n)}{t^2} \dx\rightarrow - \infty, \text{ as }t \rightarrow \infty,
	\end{equation}
	where $H(x,s) = F(x,s) + K(x,s).$ Thus $\xi ^{(n)}_k \in \Gamma _I.$ In particular, there is $t_k^{(n)} > 0$ such that
	\begin{equation}\label{lim_tete}
		c(I) \leq \sup _{t \geq 0}  I(\xi_k ^{(n)} (t) ) = I(\xi_k ^{(n)} (t_k^{(n)})).
	\end{equation}
	If, up to a subsequence, $t_k^{(n)} \rightarrow \infty,$ as $k\rightarrow \infty, $ then we can use \eqref{lim_te} again (because \ref{h_tang1} is uniform in $x$), replacing $t$ by $t_k^{(n)}$ and obtain the following contradiction: $0<c(I) \leq \lim _{k \rightarrow \infty} I(\xi^{(n)} _k(t)) =-\infty.$ Consequently, $\lim _{k \rightarrow \infty}t_k^{(n)} = \tau _{n} \geq 0,$ up to a subsequence. Now we define
	\begin{align*}
		&X (u) = \int _{\mathbb{R}^N} H(x,u) \dx,\ X _P(u) = \int _{\mathbb{R}^N} H_P(x,u) \dx,\\
		&Y (u) = \int _{\mathbb{R}^N}V(x) u^2 \dx \ \text{and} \ Y_P (u) = \int _{\mathbb{R}^N}V_P(x) u^2 \dx,\ \text{for } u \in H^1 (\mathbb{R}^N),
	\end{align*}
	where $H_P(x,s) = F_P(x,s) + K_P(x,s).$ We use \ref{h_infinito} to see that $\lim _{k \rightarrow \infty}X (\xi_k ^{(n)} (t) ) = X_P(t u_n)$ and $\lim _{k \rightarrow \infty} Y (\xi_k ^{(n)} (t) ) = Y_P( t u_n),$ uniformly in $t$ on compact sets of $\mathbb{R}.$ This guarantees that $\lim _{k \rightarrow \infty}I (\xi_k ^{(n)} (t) ) = I_P(t u_n),$ in the same sense, and one can take the limit as $k \rightarrow \infty$ in \eqref{lim_tete}, to obtain
	\begin{equation}\label{lim_tetete}
		c(I) \leq I_P( \tau _{n} u_n )= \frac{\tau _{n}^2}{2}\| \nabla u_n \|^2_2 + \frac{\tau _{n}^2}{2}\int _{\mathbb{R}^N} V_P(x) u_n^2\dx - \int _{\mathbb{R}^N} H_P (x , \tau _{n} u_n) \dx.
	\end{equation}
	Using the argument of \eqref{lim_te} in \eqref{lim_tetete}, we see that $(\tau _{n})$ is bounded and $\tau _{n} \rightarrow \bar{t}>0,$ up to a subsequence. Therefore, we can apply the Lebesgue convergence theorem, \ref{h_minimax} and Remark \ref{r_compar} to conclude
	\begin{equation*}
		c(I) \leq \lim _{n \rightarrow \infty}I_P( \tau _{n} u_n ) = I_P( \bar{t} u_P ) \leq \sup _{t \geq 0} I_P (t u_P) = I_P (u_P) = c(I_P).
	\end{equation*}
	
	\noindent\textit{Case ii)}: Eq. \eqref{PP} is independent of $x,$ that is, $V_P(x) = V_P >0,$ $f_P(x,s) = f_P(s)$ and $k_P(x,s) = k_P(s).$ In this case, we consider $\zeta _k^{(n)} (t) =u_n ( (\cdot -y_k)/t),$ $k,$ $n \in \mathbb{N},$ with $\zeta _k^{(n)} (0):=0.$ We prove that there are $n \geq n_0$ and $k \geq k_n$ such that $\zeta _k^{(n)} \in \Gamma _I,$ for $n_0$ and $k_n$ large enough. To do this, we start by pointing out that $u_P$ satisfies the Pohozaev identity
	\begin{equation}\label{pohoci}
		\int _{\mathbb{R}^N} H_P(u_P) - \frac{V_P}{2} u^2_P\dx= \frac{N-2}{2N} \int _{\mathbb{R}^N} |\nabla u_P|^2 \dx,
	\end{equation}
	which implies that the left-hand side of \eqref{pohoci} is positive. Since
	\begin{equation}\label{limitepoho}
		\lim _{n \rightarrow \infty }\int _{\mathbb{R}^N} H_P(u_n) -\frac{V_P}{2}u^2_n \dx = \int _{\mathbb{R}^N} H_P(u_P) -\frac{V_P}{2}u^2_P   \dx,
	\end{equation}
	there is $n_0$ such that 
	\begin{equation*}
		\int _{\mathbb{R}^N} H_P(u_n) -\frac{V_P}{2}u^2_n \dx >0,\quad \text{for }n \geq n_0.
	\end{equation*}
	Let us consider, by a contradiction argument, the existence of $k_\ast,$ $n_\ast\geq n_0$ such that $\zeta _{k_\ast} ^{(n_\ast)} \not \in \Gamma _I.$ In this case, there is $E_0>0$ and $t_m \rightarrow \infty$ with $I(\zeta_{k_\ast}^{(n_\ast)}(t_m) ) \geq -E_0,$ for any $m \in \mathbb{N}.$ By \ref{f_geral}, \ref{g_geral} and \ref{h_infinito}, Lebesgue convergence theorem implies
	\begin{equation*}
		\lim _{m \rightarrow \infty }\int _{\mathbb{R}^N} H(t_m x +y_{k_\ast}, u_{n_\ast})  - \frac{V(t_mx + y_{k_\ast})}{2} u^2_{n_\ast} \dx = \int _{\mathbb{R}^N}H_P (u_{n_\ast}) - \frac{V_P}{2}u^2_{n_\ast} \dx.
	\end{equation*}
	This leads to the contradiction that
	\begin{align}\label{ezero0}
		-E_0 & \leq I(\zeta_{k_\ast}^{(n_\ast)}(t_m) ) \nonumber \\
		&= \frac{t^{N-2}_m}{2}\| \nabla u_{n_\ast} \|^2_2 - t_m^N \left[\int _{\mathbb{R}^N} H(t_m x +y_{k_\ast}, u_{n_\ast})  - \frac{V(t_mx + y_{k_\ast})}{2} u^2_{n_\ast} \dx   \right]\rightarrow - \infty,\ \text{as }m \rightarrow \infty.
	\end{align}
	Thus, there is $t_k^{(n)}>0,$ $k,$ $n\geq n_0$ such that
	\begin{equation}\label{lim_tetegamma}
		c(I) \leq \sup _{t \geq 0}  I(\zeta_k ^{(n)} (t) ) = I(\zeta_k ^{(n)} (t_k^{(n)})).
	\end{equation}
	The sequence $(t_k^{(n)})_k$ is bounded, because if it were not, we could argue as in \eqref{ezero0} to obtain the same contradiction by considering that $t_k >|y_k|^2,$ up to a subsequence, allowing us to have $|t_k x + y_k| \rightarrow \infty.$ Likewise, $\lim _{k \rightarrow \infty }t_k^{(n)} =  \tau_{n}\geq 0,$ up to a subsequence.
	Arguing as before, under the same notation as above, $\lim _{k \rightarrow \infty}X (\zeta_k ^{(n)} (t) ) = X_P( u_n(\cdot /t))$ and $\lim _{k \rightarrow \infty} Y (\zeta_k ^{(n)} (t) ) = Y_P( u_n(\cdot /t)),$ uniformly in $t$ on compact sets of $\mathbb{R},$ which leads to $\lim _{k \rightarrow \infty}I (\zeta_k ^{(n)} (t) ) = I_P(u_n(\cdot /t)),$ in the same sense. Taking the limit as $k \rightarrow \infty$ in \eqref{lim_tetegamma}, we have
	\begin{equation*}
		c(I) \leq I_P (u_n (\cdot / \tau_n ))= \frac{\tau^{N-2}_n }{2}\| \nabla u_n \|^2_2 - \tau_n ^N \left[\int _{\mathbb{R}^N} H_P(u_n)  - \frac{V_P}{2} u^2_n \dx   \right].
	\end{equation*}
	By \eqref{limitepoho}, the sequence $(\tau _n)$ is bounded and $\lim_{n \rightarrow \infty } \tau _n= \bar{t}>0,$ up to a subsequence. Since $n \geq n_0$ is arbitrary, one can also take the limit as $n \rightarrow \infty$ and use \ref{h_minimax} together with Remark \ref{r_compar} to get
	\begin{equation*}
		c(I) \leq \lim _{n \rightarrow \infty} I_P (u_n (\cdot / \tau_n )) = I_P (u_P (\cdot / \bar {t} )) \leq \sup _{t \geq 0} I_P (u_P (\cdot /  t )) = I_P (u_P )= c(I_P). \qedhere
	\end{equation*}
\end{proof}
\begin{corollary}
The periodic problem \eqref{PP} cannot admit a nontrivial solution $u_P$ at the mountain pass level with an energy satisfying $I_P (u_P) <c(I).$
\end{corollary}
\begin{proposition}\label{A_estrito}
	Suppose in addition that \ref{Hast} holds. Then $c(I) < c(I_P).$
\end{proposition}
\begin{proof} From \cite{brezis-kato,struwebook}, we know that $u_P \in C(\mathbb{R}^N).$ 
	
	\noindent \textit{Case i)}: $ s \mapsto \left(f_P(x,s)+k_P(x,s) \right)|s|^{-1}$ is strictly increasing, for all $x \in \mathbb{R}^N.$ Following the argument of Lemma \ref{l_mpgeometry}, we have that $\xi(t) = t u_P$ belongs to $\Gamma _I.$ By Remark \ref{r_compar},
	\begin{equation*}
		c(I) \leq \sup _{t \geq 0} I(\xi (t)) = I(t _0 u_P) < I_P (t_0 u_P) \leq \sup _{t \geq 0} I_P(\xi (t)) = I_P(u_P) = c(I_P).
	\end{equation*}
	
	\noindent\textit{Case ii)}: Eq. \eqref{PP} is independent of $x,$ that is, $V_P(x) = V_P >0,$ $f_P(x,s) = f_P(s)$ and $k_P(x,s) = k_P(s).$ In this case, we consider $\zeta (t) = u_P (\cdot / t).$ We use \eqref{pohoci} to prove that $\zeta \in \Gamma _I,$ 
	\begin{equation*}
		I(\zeta(t)) \leq \frac{t^{N-2}}{2}\| u_P \|_2^2 - t^N \int _{\mathbb{R}^N} F_P(u_P) + K_P(u_P) - \frac{V_P}{2} u^2_P\dx \rightarrow - \infty,\quad \text{as }t \rightarrow \infty.
	\end{equation*}
	Using Remark \ref{r_compar} again, we have
	\begin{equation*}
		c(I) \leq \sup _{t \geq 0} I(\zeta (t)) = I(u_P(\cdot / t_0)) < I_P (u_P(\cdot / t_0)) \leq \sup _{t \geq 0} I_P(\zeta (t)) = I_P(u_P) = c(I_P).\qedhere
	\end{equation*}
\end{proof}

%	\bibliographystyle{ams_ex}
%	\bibliography{references}
\end{document}